\documentclass{article}

\usepackage{geometry}
\usepackage{manyfoot}
\usepackage{graphicx,color}
\usepackage[outercaption]{sidecap}
\usepackage{xfrac}
\usepackage{subfig}
\usepackage{tikz}
\usepackage{enumitem}
\usepackage{pdfsync}

\usepackage{fancyhdr}
\usepackage[nottoc,notlof,notlot]{tocbibind}
\usepackage[titles,subfigure]{tocloft}

\usepackage{amssymb}
\usepackage{amsmath}
\usepackage{latexsym}
\usepackage{amsthm}
\usepackage{eucal}
\usepackage{eufrak}
\usepackage{amsthm}
\theoremstyle{plain}
\newtheorem{theorem}{Theorem}[section]
\newtheorem{corollary}[theorem]{Corollary}
\newtheorem{lemma}[theorem]{Lemma}
\newtheorem{proposition}[theorem]{Proposition}
\theoremstyle{definition}
\newtheorem{definition}[theorem]{Definition}
\newtheorem{example}[theorem]{Example}
\newtheorem{remark}[theorem]{Remark}

\def\e{\varepsilon}
\def\rr{{\mathbb R}}
\def\dx{\,dx}
\def\NN{{\mathbb N}}
\def\ZZ{{\mathbb Z}}
\def\ie{{; i.e.,}}
\def\dist{{\rm dist}}
\def\div{{\rm div}}
\def\divv{{\rm Div}}
\def\supp{{\rm supp}}
\DeclareMathOperator*{\argmin}{arg\,min}

\numberwithin{equation}{section}

\title{A variational theory of convolution-type functionals}
\author{
{\sc Roberto Alicandro}
\\
\small DIEI, Universit\`a di Cassino e del Lazio Meridionale,\\
\small via Di Biasio 43, 03043 Cassino (FR), Italy\\
\\
{\sc Nadia Ansini}
\\ \small Dipartimento di Matematica, Sapienza Universit\`a di Roma,\\
\small P.le Aldo Moro 2, 00185 Rome, Italy\\
\\
{\sc Andrea Braides}
\\ \small Dipartimento di Matematica,
 Universit\`a di Roma `Tor Vergata',\\
 \small
via della Ricerca Scientifica, 00133 Rome, Italy\\
\\
{\sc Andrey Piatnitski}
\\ \small
The Arctic University of Norway, UiT,  Campus
Narvik, P.O. Box 385, Narvik 8505, Norway \\ \small and  Institute for Information Transmission Problems
of RAS, 127051 Moscow, Russia 
\\
\\
{\sc Antonio Tribuzio}
\\ \small Dipartimento di Matematica,
  Universit\`a di Roma `Tor Vergata',\\
 \small
via della Ricerca Scientifica, 00133 Rome, Italy
}
\date{
}                                      

\begin{document}

\maketitle

\noindent
{\bf Abstract.}
We provide a general treatment of perturbations of a class of functionals modeled on convolution energies with integrable kernel which approximate the $p$-th norm of the gradient as the kernel is scaled by letting a small parameter $\e$ tend to $0$. We first provide the necessary functional-analytic tools to show coerciveness in $L^p$. The main result is a compactness and integral-representation theorem which shows that limits of convolution-type energies is a standard local integral functional with $p$-growth defined on a Sobolev space. This result is applied to obtain periodic homogenization results, to study applications to functionals defined on point-clouds, to stochastic homogenization and to the study of limits of the related gradient flows.



\section{Introduction}
In this paper we consider a general class of non-local energies whose
prototype are {\em convolution functionals}; i.e., functionals of the form
\begin{equation}\label{intro1}
{1\over\e^{d+p}}\int_{\Omega\times\Omega} a\Bigl({x-y\over\e}\Bigr)|u(x)-u(y)|^pdx\,dy,
\end{equation}
where  $\Omega$ is a Lipschitz domain in $\rr^d$ and $a$ is a sufficiently integrable positive kernel; namely, $a$ satisfies
\begin{equation}\label{ipoa-p}
\int_{\rr^d} a(\xi)(1+|\xi|^p)d\xi<+\infty.
\end{equation}
Functionals as in \eqref{intro1} can be seen as an approximation and a generalization of an $L^p$-norm of the gradient of $u$, and as such have been used e.g.~in non-local approaches to phase-transition problem (see Alberti and Bellettini \cite{albbel}). Indeed, if $u$ is of class $C^1(\Omega)$, we may approximate $u(x)-u(y)$ with $\langle \nabla u(x),x-y\rangle$,
so that, up to an error which can be neglected as $\e\to 0$, energies \eqref{intro1} can be rewritten as
\begin{equation}
{1\over\e^{d}}\int_{\Omega\times\Omega} a\Bigl({x-y\over\e}\Bigr)\Bigl|\Bigl\langle \nabla u(x),{x-y\over\e}\Bigr\rangle\Bigr|^pdx\,dy.
\end{equation}
After the change of variables $y=x-\e\xi$ and letting $\e\to 0$, we obtain
\begin{equation}\label{intro2}
\int_\Omega\|\nabla u\|_a^p\,dx,
\hbox{ where } \|z\|_a^p= \int_{\rr^d}a(\xi)|\langle z,\xi\rangle|^pd\xi.
\end{equation}
This argument will be made rigorous as a very particular case of the results in the following.

Limits of energies similar to \eqref{intro1}, of the form
\begin{eqnarray}\label{def-Fe-solid-Brezis}
{1\over \e^{d}}\int_{\Omega\times \Omega} a\Bigl({y-x\over\e}\Bigr)
\Bigl|{u(y)-u(x)\over y-x}\Bigr|^p dy \dx,
\end{eqnarray}
have also been studied by Bourgain {et al.}~\cite{boubremir} as an alternative definition of the $L^p$-norm of the gradient of a Sobolev function (see e.g.~also \cite{pon}), as part of a number of works stemming from a general interest towards nonlocal variational problems arisen in the last twenty years (see e.g.~the survey paper \cite{dinpalval} or the book \cite{bucval}); a higher-order development of this limit process has been recently studied by Chambolle et al.~\cite{chanovpag}.
It is worth recalling that a non-linear version of functionals \eqref{intro1} with truncated quadratic potentials had been previously proposed as an approximation of the Mumford-Shah energy by De Giorgi and subsequently studied by Gobbino \cite{gob} (see also \cite{gobmor}).
Furthermore, discrete energies of the form
\begin{eqnarray}\label{def-Fe-lattice}
{1\over \e^{d+p}}\sum_{i,j\in \cal L} a_{ij}|u_i-u_j|^p,
\end{eqnarray}
where $\cal L$ is a $d$-dimensional lattice, $u:\e{\cal L}\to \rr^m$ and $u_i=u(\e i)$,
have been widely investigated as a discrete approximation of integral functionals of $p$-growth  (see e.g.~\cite{piarem,alicic,bra3,brakre}). Such energies can be seen as a discrete version of functionals \eqref{intro1}.

In the case $p=2$, some energies of the form \eqref{intro1} derive from models in population dynamics where macroscopic properties can be reduced to studying the evolution of the first-correlation functions describing the population density $u$ in the system \cite{KKP2008, FKK2012}. With that interpretation in mind, in the simplest formulation one can consider their perturbations
 \begin{equation}\label{intro1-bis}
{1\over\e^{d+2}}\int_{\Omega\times\Omega} b_\e(x,y)\,a\Bigl({x-y\over\e}\Bigr)|u(x)-u(y)|^2dx\,dy,
\end{equation}
that may take into account inhomogeneities of the environment encoded in the (non-negative) coefficient $b_\e$.
Functionals modelled on these energies, with $b_\e$ obtained by scaling a given periodic function $b$ or with  a random coefficient $b_\e=b_\e^\omega$ depending on the realization of a random variable, have been considered in the context of homogenization for $u$ scalar in \cite{2018BP,2019BP}, producing a limit elliptic homogeneous functional
\begin{equation}\label{intro1-ter}
\int_{\Omega} \langle A_{\rm hom}\nabla u,\nabla u\rangle\dx\,dy.
\end{equation}
This limit can be expressed in terms of $\Gamma$-convergence and implies the convergence of related minimum problems.
Another type of perturbations of functionals \eqref{intro1} is encountered in a different application to Data Science, studied by Garc\'ia Trillos and Slepcev \cite{garsle}, who examine energies approximating the total variation of $u$ of the form
\eqref{intro1} \begin{equation}\label{intro1-quater}
{1\over\e^{d+p}}\int_{\Omega\times\Omega} a\Bigl({T_\e( x)-T_\e( y)\over\e}\Bigr)|u(x)-u(y)|^p dx\,dy,
\end{equation}
when $p=1$, with $T_\e:\Omega\to\Omega$ and discuss its stability in terms of the convergence of $T_\e$ to the identity.
This result has been extended to the case of energies with $p$-growth for $p>1$ in \cite{thorpe-et-al} (see also the recent paper \cite{carchasle} in the context of free-discontinuity problems).

\medskip
In this paper, we provide an ample treatment of `convolution-type' energies modelled on  \eqref{intro1} and \eqref{intro1-bis} (and including also the case \eqref{intro1-quater} with $p>1$) and allowing for a general non-linear dependence on $u(x)-u(y)$ and inhomogeneity on $x$ and $y$. More precisely, the functionals that we will consider are of the form
\begin{equation}\label{intro1-gen}
{1\over\e^{d+p}}\int_{\Omega\times\Omega} f_\e(x,y,u(x)-u(y))dx\,dy,
\end{equation}
with $p>1$.
The functions $f_\e:\Omega\times \Omega\times \rr^m\to \rr$ are quite general.
In order to compare functionals \eqref{intro1-gen} with the energies defined in  \eqref{intro1} we
assume that $\Omega$ is a bounded domain and that
there exist two kernels $a_1$ and $a_2$  such that
\begin{equation}\label{intro2-gen}
a_1\Bigl({x-y\over\e}\Bigr)(|z|^p-1) \le f_\e(x,y,z)\le a_2\Bigl({x-y\over\e}\Bigr)(|z|^p+1).
\end{equation}
A non-degeneracy condition for the limit is ensured e.g.~by assuming that
$$
a_1(\xi)\ge c\quad \hbox{ if } |\xi|\le r_0
$$
for some $c,r_0>0$, while a decay condition in $a_2$ provides that the limit be finite exactly on $W^{1,p}(\Omega;\rr^m)$.
Note that for a wider applicability of this analysis considering a dependence on $\e$ of the kernel $a_2=a_2^\e$ is also necessary; since this may cause the limit energy to be non-local, some uniform conditions on the decay of the $a_2^\e$ must be required to ensure that the limit be a local integral energy.
These assumptions will be stated precisely in Section \ref{setting}. The central result of the paper is that, up to subsequences, the energies above converge to an energy of the form
\begin{equation}\label{intro-limform}
\int_\Omega f_0(x,\nabla u)\dx,
\end{equation}
with domain $W^{1,p}(\Omega;\rr^m)$. This convergence is expressed as a $\Gamma$-limit with respect to the $L^p$-topology in $\Omega$.
This is justified by a Compactness Theorem, which states that, if $a$ is the characteristic function of a ball, then any sequence $\{u_\e\}$ bounded in $L^p(\Omega)$ with equibounded energies \eqref{intro1} admits a subsequence converging to some $u\in W^{1,p}(\Omega;\rr^m)$ with respect to the $L^p(\Omega;\rr^m)$-topology.
This result is complemented by the validity of suitable Poincar\'e inequalities, which allow to prove the equi-coerciveness of the functionals subjected to boundary data and the application of the direct methods of $\Gamma$-convergence to the asymptotic description of minimum problems.

We also include in the paper various applications. First, to the homogenization of non-local functionals of the form
\begin{equation}\label{intro1-hom}
{1\over\e^{d+p}}\int_{\Omega\times\Omega}f\Bigl({x\over\e},{y\over\e},u(x)-u(y)\Bigr)dx\,dy,
\end{equation}
with $f$ periodic in the first  variable. In this case the limit integrand $f_0$ in \eqref{intro-limform} is independent of $x$ and can be characterized by a non-local asymptotic formula, which can be further simplified to a non-local cell-problem formula if $f$ is convex in the last variable. A second application is to a class of non-local functionals of the form
\begin{equation}\label{intro1-hom-quater}
{1\over\e^{d+p}}\int_{\Omega\times\Omega}f_\e (T_\e(x),T_\e(y),u(x)-u(y) )\rho(x)\rho(y)\,dx\,dy,
\end{equation}
which generalize \eqref{intro1-quater}. If the image of $T_\e$ is discrete these energies can be interpreted as a continuum interpolation of discrete energies. In particular, following \cite{garsle} we can use these functionals to describe the behavior of energies defined on point clouds. A third application is a stochastic homogenization theorem; i.e.,  the characterization of the limit of functionals in \eqref{intro1-hom} when the integrand $f=f(\omega)$ is a statistically homogeneous in the first variable  random function defined through a measure-preserving ergodic dynamical system. We characterize the limit using an asymptotic non-local homogenization formula, and prove that the limit is deterministic under ergodicity assumptions. Related results in a discrete setting can be found e.g.~in \cite{ACG,BLL1,BLL2,2014BP}.
Finally, we also treat some evolutionary problems using the methods of minimizing movements if the functions $f_\e$ are convex in the last variable. In particular, we consider the homogenization case \eqref{intro1-hom}, and show that the solutions of gradient flows for those energies, which take the form
$$
\partial_t u_\e(t,x)= -
\frac{1}{\e^{d+1}} \int_\Omega \Bigl(\partial_z f\Big(\frac{y}{\e},\frac{x}{\e},\frac{u_\e(t,x)-u_\e(t,y)}{\e}\Big)-\partial_z f\Big(\frac{x}{\e},\frac{y}{\e},\frac{u_\e(t,y)-u_\e(t,x)}{\e}\Big)\Bigr) dy,
$$
converge to the solution of the corresponding gradient flow for the limit homogenized energy. Previously, homogenization problems for parabolic equations with linear periodic not necessary symmetric convolution type operators have been studied in \cite{PiZh_par}, where the homogenization results were obtained by means of two-scale expansions technique.

\smallskip
The plan of the paper is as follows. In Section \ref{setting} we introduce the necessary notation (in particular we change the formal appearance of our energies so as to highlight the range of the interactions) and introduce the class of convolution-type energies under examination, comparing the hypotheses with the corresponding ones for integral functionals and highlighting some differences. We state a weaker version of the coerciveness assumption that may be of use when dealing e.g.~with perforated domains (see \cite{BCD}). We also introduce the special convolution energies $G_\e[a]$ of type \eqref{intro1} and in particular $G^r_\e$ when $a=\chi_{B_r}$, which are used as comparison energies throughout the paper (in particular, since they are a lower bound for the energies we consider, it is sufficient to state compactness results for families  of functions $\{u_\e\}$ with $G^r_\e(u_\e)$ bounded). Section \ref{priers} contains some general results, that mirror the analog results in Sobolev spaces, of extension from Lipschitz sets, compactness with respect the $L^p$ convergence and Poincar\'e inequalities, where the role of the $p$-th norm of the gradient is played by $G^r_\e$ as $\e\to 0$. The fundamental Lemma \ref{boundlemma} allows to control long-range interactions with short-range interactions, while the Compactness Theorem \ref{kolcom} guarantees both a compact embedding in $L^p$ for sequences with equibounded $G^r_\e$-energies, and that their limits belong to $W^{1,p}$. In Section \ref{particular} we consider the particular case of the limit of energies $G_\e[a_\e]$ with varying $a_\e$, characterizing their limits. In particular,  when $a_\e=a$ we obtain the $\Gamma$-convergence to energy \eqref{intro2}. This limit is used to provide lower and upper bounds for the general case. The main result of the paper is contained in Section \ref{integralrep}, where the general compactness and integral-representation Theorem \ref{reprthm} is proved using a variation of the localization method of $\Gamma$-convergence, which is possible since, even though convolution-type functionals are non-local, their non-locality `vanishes' as $\e\to0$. An important technical result formalizing this observation is Lemma \ref{lim-truncated-lemma}, which states, in terms of functionals \eqref{intro1}, that it is not restrictive to deal with kernels $a$ with bounded support, up to a truncation argument. Section \ref{co-mi-pro} deals with the convergence of minimum problems with Dirichlet boundary conditions.
Note that for convolution-type functionals such conditions must be imposed on a neighbourhood of size of order $\e$ of the boundary. Section \ref{homogenization} specializes the description of the $\Gamma$-limit in the case of periodically oscillating energies, using the integral-representation result, the truncation argument and the convergence of minimum problems to obtain homogenization formulas for the energy function of the $\Gamma$-limit, both of asymptotic type (Theorem \ref{hom-thm} and formula \eqref{homform}) taking into account interactions within cubes of diverging size and on periodic functions in the convex case (Theorem \ref{cell-form-thm} and formula \eqref{cellform}). Note that in the latter we take into account interactions $u(x)-u(y)$ with $x$ in the periodicity set and $y$ in the whole space.  In Section \ref{pe-co-fu} we consider functionals of the form \eqref{intro1-hom-quater} and prove their equivalence to the corresponding functionals \eqref{intro1-gen} when $T_\e$ approaches the identity up to an error of order $o(\e)$, under some technical conditions on $f_\e$. This result is then applied to the analysis of energies defined on point clouds in Section \ref{pointclouds}. Section \ref{stocks} contains a homogenization theorem in the stochastic setting (Theorem \ref{erg-the}), whose proof generalizes the arguments utilized for the deterministic homogenization result, using a subadditive ergodic theorem by Krengel and Pyke (Theorem \ref{KP-thm}) to characterize an asymptotic homogenization formula. Finally, in Section \ref{dynamic} we study the convergence of the gradient flows associated to our energies in the convex case. A general approach by minimizing movements allows to deduce in particular the convergence of gradient flows in the case of the homogenizationto the gradient flow of the homogenized limit, which is a standard parabolic equation (Theorem \ref{stability-hom-N}).

\section{Notation, setting of the problem and comments}\label{setting}
We let $\lfloor t\rfloor$ denote the integer part of $t\in\rr$. Given  $x\in\mathbb{R}^d$, we let $\lfloor x\rfloor$ denote  the vector in $\mathbb{Z}^d$ whose components are the integer parts of the components of $x$; that is,
$\lfloor x\rfloor=(\lfloor x_1\rfloor ,\dots, \lfloor x_d\rfloor)$.
If $z\in\rr^d$ then $|z|$ denotes the norm of $z$ and $\langle x,w\rangle$ the scalar product between $z$ and $w$.
Moreover,  we will use the notation
$B_r(z)$ (if $z=0$, simply $B_r$) for the open ball of centre $z$ and radius $r$.
If $A$ and $B$ are subsets of $\rr^d$ then
dist$(z,A)=\inf \{|z-x|: x\in A\}$ and dist$(A, B)=\inf \{|z-x|: x\in A, z\in B\}$ denote the distance of $z$ from $A$ and from $A$ to $B$, respectively.
$\mathcal{A}(\Omega)$ will be the family of all open subsets of $\Omega$ and $\mathcal{A}^{\rm reg}(\Omega)$ the subfamily of open subsets with Lipschitz boundary. By
 $A\Subset B$ we mean that the closure of $A$ is a compact subset of $B$. $\mathbb{R}^{m\times d}$ denotes the space of $m\times d$ matrices with real entries; if $M\in\mathbb{R}^{m\times d}$ and $z\in \rr^d$ then $Mx\in \rr^m$ is defined by the usual row-by-column multiplication.
Given $g:\rr^{m\times d}\to\rr$ we use the notation
$$
Dg(M)=\Big(\frac{\partial g(M)}{\partial M_{i,j}}\Big)_{i\in\{1,\ldots,m\}  j\in\{1,\ldots,d\}}\in\rr^{m\times d}
$$
for the gradient of $g$.
Given a matrix-valued map $\Sigma:\Omega\to\rr^{m\times d}$ we define $\divv(\Sigma):\Omega\to\rr^m$ as
$$
\big(\divv(\Sigma)\big)_i = \div(\Sigma_i) \quad \text{for every } i=1,\dots,m.
$$

 {We use standard notation for Lebesgue and Sobolev spaces, and their local versions. If $u$ is an integrable function on a measurable set $E\subset\Omega$,
$$u_E:=\frac{1}{|E|}\int_E u(x)dx$$
denotes the average of $u$ on $E$.
If $A\Subset B$, a {\em cut-off function} between $A$ and $B$ is a (smooth) function $\varphi$ with $0\le\varphi\le 1$, $\varphi=0$ on $\partial B$ and $\varphi=1$ on $A$.

We use standard notation for $\Gamma$-convergence \cite{bra2}, indicating the topology with respect to which it is performed. 

Unless otherwise stated, the letter $C$ denotes a generic strictly positive constant independent of the parameters of the problem taken into account. }

\subsection{Convolution-type energies}
In the following we fix $d,m\in\mathbb{N}$ natural numbers, $p>1$ and we let $\Omega\subset\mathbb{R}^d$ be a bounded open set with Lipschitz boundary.
Note that for many results this regularity assumption on $\Omega$ may be removed up to considering local arguments.

 Given  $\e>0$ and $f_\e:\Omega\times\mathbb{R}^d\times\mathbb{R}^m\to[0,+\infty)$ a positive Borel functions, we introduce the non-local functional $F_\e:L^p(\Omega;\mathbb{R}^m)\to[0,+\infty]$ defined as
\begin{equation}\label{functionals}
F_\e(u):=\int_{\mathbb{R}^d}\int_{\Omega_\e(\xi)} f_\e\Bigl(x,\xi,\frac{u(x+\e\xi)-u(x)}{\e}\Bigr) dx \, d\xi
\end{equation}
where $\Omega_\e(\xi):=\{x\in\Omega\,|\,x+\e\xi\in\Omega\}$.
 Note that definition \eqref{functionals} is equivalent to \eqref{intro1-gen} up to the change of variable $y=x+\e\xi$. The reason to rewrite our energies in this apparently more complicated way is in order to state more clearly the hypotheses in the following, highlighting the `microscopic interaction length' $\xi$.
A particular class of functionals of the form \eqref{functionals}, which will be used in the following as comparison energies, is introduced below.

\begin{definition}[the {\em convolution functionals} {$G_\e[a]$} and $G_\e^r$]\label{convfun}
Given a measurable function $a:\mathbb{R}^d\to[0,+\infty)$, we set
\begin{equation}\label{conv-functionals}
G_\e[a](u) := \int_{\mathbb{R}^d}a(\xi)\int_{\Omega_\e(\xi)}\Bigl|\frac{u(x+\e\xi)-u(x)}{\e}\Bigr|^p dx\, d\xi.
\end{equation}
Moreover, we define the local versions  $G_\e[a](u,A)$ as for $F_\e$ in \eqref{loc-functionals}.

In the case in which $a=\chi_{B_r}$ we will simply write $G_\e^r$ instead of $G_\e[\chi_{B_r}]$.
\end{definition}

 We will study the behaviour of energies $F_\e$ as $\e\to0$, under proper assumptions on $f_\e$, by proving a compactness result with respect to $\Gamma$-convergence in the $L^p$-topology. In this context the role of convolution functionals $G_\e[a]$ will be the analog of that of the integral of the $p$-th power of the gradient as a comparison energy for local integral functional with $p$-growth.

For any $u\in L^p(\Omega;\rr^m)$ and $A\in\mathcal{A}(\Omega)$ we also introduce a local version of the functionals in \eqref{functionals} by setting\begin{equation}\label{loc-functionals}
F_\e(u,A):=\int_{\mathbb{R}^d}\int_{A_\e(\xi)} f_\e\Bigl(x,\xi,\frac{u(x+\e\xi)-u(x)}{\e}\Bigr) dx \, d\xi,
\end{equation}
where $A_\e(\xi):=\{x\in A\,|\, x+\e\xi\in A\}$.


\medskip

In what follows we use the standard notation
$$F'(u,A):=\Gamma\text{-}\liminf_{\e\to0} F_\e(u,A),\quad F''(u,A):=\Gamma\text{-}\limsup_{\e\to0} F_\e(u,A)$$
for the upper and lower $\Gamma$-limits (cf.~\cite{bra2})
performed with respect to the strong topology of  $L^p(\Omega;\mathbb{R}^m)$.

\begin{remark}\label{locality}
Note that in this setting the upper and lower $\Gamma$-limits of  $F_\e(\cdot,A)$ performed with respect to the strong $L^p(A;\rr^m)$  topology or with respect to the strong $L^p(\Omega;\rr^m)$ topology are the same. Indeed, for any sequence $u_\e$ converging to $u\in L^p(\Omega;\mathbb{R}^m)$ strongly in $L^p(A;\rr^m)$, we can take
$$\tilde u_\e(x)=\begin{cases}u_\e(x) & \text{if } x\in A \\
u(x) & \text{if } x\in\Omega\backslash A,\end{cases}$$
which converges to $u$ strongly in $L^p(\Omega;\rr^m)$ and leaves the energy unchanged; that is, $F_\e(u_\e,A)=F_\e(\tilde u_\e,A)$.
\end{remark}


%

\subsection{Assumptions}\label{mainass}

We consider the following hypotheses on the functions $f_\e$ introduced above: there exist two strictly positive constants $r_0,c_0$  and two families of non-negative functions $\rho_\e:B_{r_0}\to [0,+\infty)$ and $\psi_\e:\mathbb{R}^d\to[0,+\infty)$  satisfying
\begin{equation}\limsup_{\e\to0}\int_{B_{r_0}}\rho_\e(\xi)d\xi<+\infty,\end{equation}
\begin{equation}\label{ass-bound}
C_1:=\limsup_{\e\to0}\int_{\mathbb{R}^d}\psi_\e(\xi)(|\xi|^p+1)d\xi<+\infty,
\end{equation} such that
\begin{equation}
 c_0 (|z|^p-\rho_\e(\xi))\le f_\e(x,\xi,z) \quad \text{ if }|\xi|\le r_0; \tag{H0}
\end{equation}
\begin{equation}
f_\e(x,\xi,z) \le \psi_\e(\xi)(|z|^p+1)\,.
\tag{H1}
\end{equation}
Moreover for any $\delta>0$ there exists $r_\delta>0$ such that
\begin{equation}\label{ass-local}
\limsup_{\e\to0} \int_{B_{r_\delta}^c}\psi_\e(\xi)|\xi|^p d\xi<\delta. \tag{H2}
\end{equation}

By hypotheses (H0) and (H1), the localized functionals satisfy
\begin{equation}\label{growth-cond}
c\big(G_\e^{r_0}(u,A)-|A| \big) \le F_\e(u,A) \le G_\e[\psi_\e](u,A)+C|A|
\end{equation}
for any $A\in\mathcal{A}(\Omega)$, for all sufficiently small  $\e$, and for some $0<c<C$.

\begin{remark}[convolution functionals with kernels of polynomial decay] A simple class of kernels $\psi_\e$ satisfying \eqref{ass-bound} and (H2) are those satisfying
$$
\psi_\e(\xi)\le {C_\alpha\over 1+ |\xi|^\alpha}
$$
for some given $\alpha>p+d$ and $C_\alpha>0$.
\end{remark}

\begin{remark}\label{conf-cond}
The hypotheses above can be compared with the standard $p$-growth assumptions for a family of integral functionals $\int_\Omega f_\e(x,\nabla u)\dx$, which read
\begin{equation}\label{growth-cond_int_fun}
a_1(|\xi|^p-a_0)\le f_\e(x,\xi) \le a_2(1+|\xi|^p)
\end{equation}
for some positive constants $a_0,a_1$, and $a_2$. The polynomial lower-bound in \eqref{growth-cond_int_fun} implies the boundedness of the $L^p$ norm of the gradients of functions with equibounded energies and hence provides weak compactness in $W^{1,p}$ spaces. Analogously, we will show that condition (H0) provides strong compactness in $L^p$  of functions with equibounded energies and yields that any limit function is in $W^{1,p}(\Omega;\rr^m)$. Condition (H1) and \eqref{ass-bound} are the analog of the polynomial upper-bound in \eqref{growth-cond_int_fun} and ensure that the $\Gamma$-limits are finite on $W^{1,p}$-functions. Condition (H2) is crucial to deduce the locality of the $\Gamma$-limits, in that it forbids relevant long-range interactions, and has no direct analog in terms of condition  \eqref{growth-cond_int_fun}.
\end{remark}



For the validity of the integral representation result in Theorem \ref{representationthm} condition (H0)
can be weakened requiring only  that:
\begin{description}[font=\normalfont]
\item[(H0$'$)] the computation of $F'(u,A)$ and $F''(u,A)$ can be restricted to families $\{u_\e\}\subset L^p(\Omega;\mathbb{R}^m)$ satisfying, for every open  $A'\Subset A$,
\begin{equation}\label{relax-growth-cond}
G_\e^{r_0}(u_\e,A') \le C(F_\e(u_\e,A)+|A|)
\end{equation}
when $\e$ is small enough, where $C$ is a constant depending on $A\in\mathcal{A}(\Omega)$.
\end{description}
Such condition is clearly implied by the lower bound in \eqref{growth-cond}.
Assuming (H0$'$) in place of (H0)  allows to include in our analysis a wide varieties of cases which are not covered by (H0), such as convolution energies on perforated domains, where \eqref{relax-growth-cond} is obtained by an extension theorem (see \cite{2018BP}).
In the two following examples we exhibit two further classes of convolution functionals satisfying \eqref{relax-growth-cond}.

\begin{remark}[control from below provided by a {\em translated kernel}]\rm
We now show that for the validity of \eqref{relax-growth-cond}  it suffices that $f_\e(x,\cdot,z)$ be controlled from below by an interaction kernel not necessarily centered in the origin. More precisely, assume that there exist $r_0,c_0>0$ and $\xi_0\in\mathbb{R}^d$ such that
\begin{equation}\label{remk-h0'}
c_0 (|z|^p-\rho_\e(\xi))\le f_\e(x,\xi,z) \quad \text{ if }|\xi-\xi_0|\le r_0.
\end{equation}
Take $r<r_0/2$ and an open $A'\Subset A$.
For every $\eta\in B_r(\xi_0)$, by Jensen's inequality we have
\begin{eqnarray*}
G_\e^r(u,A') &\le& 2^{p-1} \Big( \int_{B_r} \int_{A'} \Big|\frac{u(x+\e(\xi+\eta))-u(x)}{\e}\Big|^p dx\, d\xi \\
&&\qquad\qquad+ \int_{B_r} \int_{A'} \Big|\frac{u(x+\e(\xi+\eta))-u(x+\e\xi)}{\e}\Big|^p dx\, d\xi \Big).
\end{eqnarray*}
Note that, for $\e$ small enough, the previous expression is well defined since $A'+\e(\xi+\eta)\subset A$, for every $\xi$ and $\eta$ as above.
Averaging with respect to $\eta$ we get
\begin{eqnarray*}
G_\e^r(u,A') &\le& \frac{2^{p-1}}{|B_r|} \Big( \int_{B_r(\xi_0)} \int_{B_r} \int_{A'} \Big|\frac{u(x+\e(\xi+\eta))-u(x)}{\e}\Big|^p dx\, d\xi\, d\eta \\
&&\qquad\qquad+ \int_{B_r(\xi_0)} \int_{B_r} \int_{A'} \Big|\frac{u(x+\e(\xi+\eta))-u(x+\e\xi)}{\e}\Big|^p dx\, d\xi\, d\eta \Big).
\end{eqnarray*}
By the change of variable $\xi'=\xi+\eta$ and from the fact that $B_r(\eta)\subset B_{r_0}(\xi_0)$ we have
\begin{eqnarray*}
\frac{1}{|B_r|} \int_{B_r(\xi_0)} \int_{B_r} \int_{A'} \Big|\frac{u(x+\e(\xi+\eta))-u(x)}{\e}\Big|^p dx\, d\xi\, d\eta \\
\le \int_{B_{r_0}(\xi_0)} \int_{A'} \Big|\frac{u(x+\e\xi')-u(x)}{\e}\Big|^p dx\, d\xi'\, .
\end{eqnarray*}
Using the change of variable $x'=x+\e\xi$ and the fact that $A'+\e\xi\subset A_\e(\eta)$ for every $\eta\in B_r(\xi_0)$, $\xi\in B_r$ we also get
\begin{eqnarray*}
\frac{1}{|B_r|} \int_{B_r(\xi_0)} \int_{B_r} \int_{A'} \Big|\frac{u(x+\e(\xi+\eta))-u(x+\e\xi)}{\e}\Big|^p dx\, d\xi\, d\eta \\
\le  \int_{B_r(\xi_0)} \int_{A_\e(\eta)} \Big|\frac{u(x'+\e\eta)-u(x')}{\e}\Big|^p dx'\, d\eta\, .
\end{eqnarray*}
Gathering all the inequalities above, for any open $A'\Subset A$ we obtain that
$$G_\e^r(u,A') \le 2^p G_\e[\chi_{B_{r_0}(\xi_0)}](u,A)$$
for every $\e$ small enough and \eqref{remk-h0'} yields \eqref{relax-growth-cond}.
\end{remark}

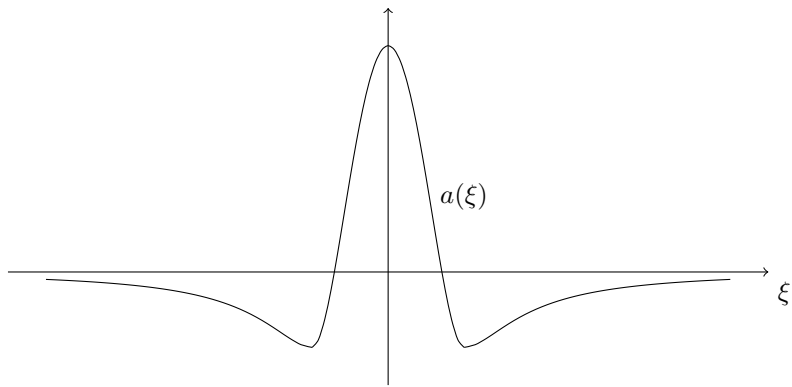
\begin{figure}[h!]\label{ex-growth}
\begin{center}
\begin{tikzpicture}
\draw [->] (-5,0) -- (5,0) node[below right] {$\xi$};
\draw [->] (0,-1.5) -- (0,3.5);
\draw [smooth, domain=-1:1] plot (\x, { 4*(\x*\x-1)^2-1 } ) node[yshift=2cm] {$a(\xi)$};
\draw [smooth, domain=1:4.5] plot (\x, { (\x*\x-1)^2/((\x*\x)^2+1)-1 } );
\draw [smooth, domain=-4.5:-1] plot (\x, { (\x*\x-1)^2/((\x*\x)^2+1)-1 } );
\end{tikzpicture}
\end{center}
\caption{Example of an interaction kernel $a$ with changing sign
}
\label{ex-growth-fig}
\end{figure}

\begin{remark}[sign-changing kernels]\rm\label{signch}
In case of convolution(-type) energies the assumption that the kernels be non-negative can be relaxed to some extent. Note that this has no direct counterpart in condition \eqref{growth-cond_int_fun} for integral functionals, since the negativeness of $f_\e$ on a set of $x$ of positive measure would give a $\Gamma$-limit identically equal to $-\infty$.

A simple example is obtained by taking $r$ and $r_0$ as in the previous example, and
$$
f(x,\xi,z)= \Bigl(\chi_{B_{r_0}(\xi_0)}(\xi)-\gamma\chi_{B_r(0)}(\xi)\Bigr)|z|^p,
$$
with $\gamma<2^{-p}$. This function is negative for $|\xi|<r_0$. Nevertheless, by the previous example, we obtain
$$
(2^{-p}-\gamma)\int_{\rr^d}\int_{\{|\xi|\le r\}}|u(x+\e\xi)-u(x)|^p\,d\xi\dx\le
\int_{\rr^d}\int_{\rr^d}f(\xi,u(x+\e\xi)-u(x))\,d\xi\dx,
$$
which gives a bound for the convolution energies on the left-hand side for families with equibounded energies on the whole $\rr^d$. From this bound it is possible to derive compactness properties. Note however that it does not immediately imply condition \eqref{relax-growth-cond} for the localized energies, so that arguments requiring local estimates, such as integral-representation theorems, must be reworked.

\bigskip
We may also treat the case when the domain is not the whole $\rr^d$ if we make some assumptions on the integration domain; e.g. convexity.

Let $A$ be an open convex set.
Applying Jensen's inequality and switching the roles of $x$ and $y$ we get
\begin{multline*}
\int\int_{\{(x,y)\in A\times A \,:\, \e<|x-y|<2\e\}} \Big|\frac{u(y)-u(x)}{\e}\Big|^p dx\, dy \\
\le 2^p \int\int_{\{(x,y)\in A\times A \,:\, \e<|x-y|<2\e\}} \Big|\frac{u\big(\frac{x+y}{2}\big)-u(x)}{\e}\Big|^p dx\, dy.
\end{multline*}
From the convexity of $A$, $(x+y)/2\in A$ and $|(x+y)/2-x|<\e$. Hence, by the change of variable $y'=(x+y)/2$ we obtain
\begin{multline*}
\int\int_{\{(x,y)\in A\times A \,:\, \e<|x-y|<2\e\}} \Big|\frac{u(y)-u(x)}{\e}\Big|^p dx\, dy \\
\le 2^{p+d} \int\int_{\{(x,y)\in A\times A \,:\, |x-y'|<\e\}} \Big|\frac{u(y')-u(x)}{\e}\Big|^p dx\, dy'.
\end{multline*}
Now if we consider
$$
f(x,\xi,z)= \Bigl(2\chi_{B_1}(\xi)-\gamma\chi_{B_2\setminus B_1}(\xi)\Bigr)|z|^p,
$$
with $\gamma<2^{-p-d}$, the previous computations yield
$$F_\e(u,A) \ge G_\e^1(u,A).$$
This argument can be extended to sign-changing $a$ such as the one pictured in Fig.~\ref{ex-growth-fig}; i.e., with $a(\xi)$ positive for small values of $|\xi|$.
The arguments above are no longer true in the non-convex case and their generalization to every open set $A\in\mathcal{A}(\Omega)$ is not immediate.
%
\end{remark}

\begin{remark}[a technical remark on localization techniques]\rm Alternatively to \eqref{loc-functionals}, given an open set $\Omega$, for any $u\in L^p(\Omega;\rr^m)$ and $A\in\mathcal{A}(\Omega)$ we may introduce another (partially-)local version of the functionals in
\eqref{functionals} by setting\begin{equation}\label{loc-functionals-2}
\widetilde F_\e(u,A):=\int_{\mathbb{R}^d}\int_{A_\e(\Omega;\xi)} f_\e\Bigl(x,\xi,\frac{u(x+\e\xi)-u(x)}{\e}\Bigr) dx \, d\xi,
\end{equation}
where $A_\e(\Omega;\xi):=\{x\in A\,|\, x+\e\xi\in \Omega\}$. This version has the advantage of being additive in the set $A$, but loses the locality property (i.e., $\widetilde F_\e(u,A)$ depends also on the values of $u$ on $\Omega\setminus A$), so that its $\Gamma$-limits must be computed with respect to the convergence in $L^p(\Omega;\rr^m)$.
\end{remark}

\section{Asymptotic embedding and compactness results}\label{priers}

In this section we include some results of independent interest that extend to the case of convolution energies corresponding results in Sobolev spaces (c.f. \cite{Leoni}).

In Theorem \ref{kolcom} we prove the compactness of sequences of functions for which both the $L^p$ norms and the energies are uniformly bounded, whose proof is a non-local counterpart of that of the classical Riesz-Fr\'echet-Kolmogorov Theorem. We will see that the role played by the $L^p$ norm of gradients in the standard compact immersion results of Sobolev spaces is played in our context by the energies $G_\e^r$.
A similar result in the case of convolution-type energies with truncated quadratic potentials has been proved in \cite[Theorem 5.4]{gob}. In Theorems \ref{poincare} and \ref{pwineq} we show the validity of Poincar\'e-type inequalities.
Before proving the compactness result, we provide some preliminary results of independent interest concerning extension operators and the possibility of controlling long-range interactions with short-range interactions.

\subsection{An extension result}

By mimicking a standard procedure of extensions of Sobolev functions, we provide the following result.

\begin{theorem}\label{ext}
Let $A$ be any open Lipschitz set of $\rr^d$ with $\partial A$ bounded and let $r>0$. Then there exist an open set $\tilde A\Supset A$, a linear continuous map
$$E:L^p(A;\mathbb{R}^m)\to L^p(\tilde A;\mathbb{R}^m)$$
and three positive constants $C=C(A), r_1=r_1(r,A)$, $\e_0=\e_0(r,A)$ such that
$$
\|E u\|^p_{L^p(\tilde A;\rr^m)} + G_\e^{r_1}(E u,\tilde A) \le C \big( \|u\|^p_{L^p(A;\rr^m)} + G_\e^{r}(u,A) \big)
$$
for all $\e<\e_0$.
\end{theorem}

\begin{proof}
We follow the construction of the extension of functions in fractional Sobolev spaces (see~\cite{dinpalval}).
From the Lipschitz regularity of the boundary we can find a finite open covering $\{U_i\}_{i=1}^n$ of $\partial A$ and Lipschitz invertible maps $H_i:Q\to U_i$ with $\|DH_i\|_{L^\infty(A)}\le L$ and $\|DH_i^{-1}\|_{L^\infty(U_i)}\le L$, such that
$$H_i(Q^+)=A\cap U_i,\quad H_i(Q^-)=A^c\cap U_i,\quad H_i(Q_0)=\partial A\cap U_i,$$
where $Q=(-1,1)^d$, $Q^{\pm}=\{x\in Q\,|\,\pm x_1>0\}$ and $Q^0=\{x\in Q\,|\, x_1=0\}$.
Let $\{\varphi_i\}_{i=0}^n\subset C_0^\infty(\rr^d)$ be a partition of the unity, that is $V_i:=\supp(\varphi_i)\Subset U_i$ and
$$
\sum_{i=0}^n \varphi_i(x)=1, \text{ for every } x\in A,
$$
where $U_0\supset A\backslash \bigcup_{i=1}^n U_i$ is an open set.
Define the maps $R_i:A^c\cap U_i\to A\cap U_i$ as follows
$$R_i(x)=H_i(-y_1,y_2,\dots,y_d),$$
where $H_i^{-1}(x)=y=(y_1,\dots,y_d)\in Q^-$.
Note that $R_i$ are invertible Lipschitz maps of Lipschitz constant less than $L^2$.
Given $u\in L^p(A;\mathbb{R}^m)$, define
$$
u_i(x) :=
\begin{cases}
u(x) & x\in A\cap U_i \\
u(R_i(x)) & x\in A^c\cap U_i,
\end{cases}
\quad
\tilde u_i(x) := \varphi_i(x) u_i(x).
$$
Then $\tilde u:=\sum_{i=0}^n \tilde u_i$ extends $u$ on $\tilde A=\bigcup_{i=0}^n U_i$.
For every $0\le i\le n$ we have
\begin{equation}\label{norm-control}
\int_{U_i}|u_i(x)|^p dx \le C\int_{U_i\cap A}|u(x)|^p dx,
\end{equation}
where $C>0$ depends only on $A$, which yields $\|\tilde u\|_{L^p(\tilde A;\rr^m)}^p\le C\|u\|^p_{L^p(A;\rr^m)}$.
Now, we show that for every $1\le i\le n$
\begin{equation}\label{reflections-control}
G_\e^{r_1}(u_i,U_i) \le C G_\e^r(u,U_i\cap A)
\end{equation}
with $r_1=r/(1+2L^2)$, which is trivial when $i=0$.
Using the change of variable $y=x+\e\xi$, it is convenient to rewrite the energy functionals as
\begin{align}\nonumber
G_\e^{r_1}(u_i,U_i) &= \frac{1}{\e^d} \int_{U_i}\int_{U_i\cap B_{\e r_1}(x)} \Big|\frac{u_i(y)-u_i(x)}{\e} \Big|^p dy\, dx \\ \nonumber
&\le G_\e^r(u,U_i\cap A) \\ \label{ext-ineq1}
&\quad\quad + \frac{1}{\e^d}\int_{U_i\cap A}\int_{(U_i\backslash A)\cap B_{\e r_1}(x)}\Big|\frac{u_i(y)-u_i(x)}{\e}\Big|^p dy\,dx
\\ \label{ext-ineq2}
&\quad\quad +\frac{1}{\e^d}\int_{U_i\backslash A}\int_{U_i\cap A\cap B_{\e r_1}(x)}\Big|\frac{u_i(y)-u_i(x)}{\e}\Big|^p dy\,dx
\\ \label{ext-ineq3}
&\quad\quad + \frac{1}{\e^d}\int_{U_i\backslash A}\int_{(U_i\backslash A)\cap B_{\e r_1}(x)}\Big|\frac{u_i(y)-u_i(x)}{\e}\Big|^p dy dx.
\end{align}
For every $1\le i\le n$ and $x\in U_i$, we claim that $|R_i(y)-x| \le \e r$ for any $y\in(U_i\backslash A)\cap B_{\e r_1}(x)$.
Indeed, if $z\in U_i\cap B_{\e r_1}(x)\cap\partial A$ then $R_i(z)=z$ and therefore
$$
|R_i(y)-x| \le |R_i(y)-R_i(z)| + |z-x| \le 2L^2 |y-z| + |z-x|.
$$
Thus, $R_i((U_i\backslash A)\cap B_{\e r_1}(x))\subset U_i\cap A\cap B_{\e r}(x)$, and, after the change of variable $y'=R_i(y)$, we obtain
$$\int_{U_i\cap A} \int_{(U_i\backslash A)\cap B_{\e r_1}(x)} \Big|\frac{u_i(y)-u_i(x)}{\e}\Big|^p dy\,dx
\le L^{2d}\int_{U_i\cap A}\int_{U_i\cap A\cap B_{\e r}(x)}\Big|\frac{u(y')-u(x)}{\e}\Big|^p dy' dx,$$
that controls the term in \eqref{ext-ineq1}.
Applying the same argument, we obtain an analogous estimate for the integral in \eqref{ext-ineq2}.
By the changes of variable $y'=R_i(y)$ and $x'=R_i(x)$ we also get
$$
\int_{U_i\backslash A}\int_{(U_i \backslash A)\cap B_{\e r_1}(x)}\Big|\frac{u_i(y)-u_i(x)}{\e}\Big|^p dy\,dx \le L^{4d^2}\int_{U_j\cap A}\int_{U_i\cap A\cap B_{\e r}(x)}\Big|\frac{u(y)-u(x)}{\e}\Big|^p dy\,dx
$$
and the integral in \eqref{ext-ineq3} is controlled as well.
Hence \eqref{reflections-control} holds.
Set
$$
\e_0 = \e_0(r,A) = \frac{1}{r_1}\min_{0\le i\le n}\dist(V_i,U_i^c)>0.
$$
For every $\e<\e_0$ and $0\le i\le n$, by \eqref{reflections-control} and \eqref{norm-control}, summing and subtracting $\varphi_i(x+\e\xi)\tilde u_i(x)$ we get
\begin{equation}\label{cutoff-control}
\begin{aligned}
G_\e^{r_1}(\tilde u_i, \tilde A) &\le 2^{p-1}\int_{B_{r_1}}\int_{(U_i)_\e(\xi)} \Big|\frac{u_i(x+\e\xi)-u_i(x)}{\e}\Big|^p dx\, d\xi \\
& \quad \quad + 2^{p-1}\int_{B_{r_1}}\int_{(U_i)_\e(\xi)}|u_i(x)|^p \Big|\frac{\varphi_i(x+\e\xi)-\varphi_i(x)}{\e}\Big|^p dx\, d\xi \\
&\le C \big( G_\e^r(u,U_i\cap A) + \|u\|_{L^p(U_i\cap A;\rr^m)}^p \big).
\end{aligned}
\end{equation}
Eventually for every $\e<\e_0$ from \eqref{cutoff-control} we obtain
$$
G_\e^{r_1}(\tilde u,\tilde A) \le \sum_{i=0}^n G_\e^{r_1}(\tilde u_i,\tilde A) \le C \big( G_\e^r(u,A)+\|u\|_{L^p(A;\rr^m)}^p \big)
$$
and conclude the proof.
\end{proof}
\subsection{Control of long-range interactions with short-range interactions}

With the following key result we show that interactions of any admissible range can be suitably controlled by the short-range energy $G_\e^{r}$. This is an analogue of Lemma 3.6 in \cite{alicic}, which deals with lattice interactions.

\begin{lemma}\label{boundlemma}
For every $r>0$ there exists a positive constant $C$ such that, for any open set $E\subset\Omega$
and for every $\xi\in\mathbb{R}^d$, $\e>0$ such that
\begin{equation}\label{eps-small}
\e \,r < \dist(E+(0,\e)\xi,\Omega^c)
\end{equation}
and $u\in L^p(\Omega;\mathbb{R}^m)$,
there holds
$$\int_E\Big|\frac{u(x+\e\xi)-u(x)}{\e}\Big|^p dx \le C(|\xi|^p+1) G_\e^r(u,E_{\e,\xi})$$
with $E_{\e,\xi}=E+(0,\e)\xi+B_{\e r}$ and $(0,\e)\xi=\cup\{\epsilon \xi\,:\, \epsilon\in(0,\e)\}$.
\end{lemma}
\begin{proof}
First, note that condition \eqref{eps-small} implies that $E_{\e,\xi}\subset\Omega$ and thus the terms in the inequality above are well defined.
For notational reasons we set $\e'=\e r/\sqrt{d+3}$.
Let $R_\xi\in\mathbb{R}^{d\times d}$ be a rotation matrix such that $R_\xi e_1=\xi/|\xi|$.
We introduce the lattice $\mathcal{L}_\e:=\{R_\xi i\,|\, i\in\e'\mathbb{Z}^d\}$ and, for any $j\in\mathcal{L}_\e$ define $Q_\e^j:=j+R_\xi (-\e'/2,\e'/2)^d$ and set
$$\tilde E_{\e,\xi}:=\bigcup_{j\in\mathcal{L}_\e}\big\{Q_\e^j\,|\,Q_\e^j\cap (E+(0,\e)\xi)\not=\emptyset\big\}.$$
Let  $k=\lceil \e|\xi|/\e'\rceil+1$, and, for any $j_0\in\mathcal{L}_\e$ and $0\le l\le k-1$, define
$j_l=j_0+l\e'{\xi\over|\xi|}$. Denote by $x_l$ any point in $Q_\e^{j_l}$ and, for the sake of simplicity of notation, $x_k=x_0+\e\xi$ and $Q_\e^{j_k}=Q_\e^{j_0}+\e\xi$. Note that $x_0+\e\xi\in Q_\e^{j_{k-2}}\cup Q_\e^{j_{k-1}}$  for every $x_0\in Q_\e^{j_0}$,.

Using the inequality
$$\Bigl|\frac{u(x_k)-u(x_0)}{\e}\Bigr|^p\le k^{p-1} \sum_{l=1}^k\Bigl|\frac{u(x_l)-u(x_{l-1})}{\e}\Bigr|^p$$
and integrating in every variable we get
\begin{align*}
\int_{Q_\e^{j_0}}\Bigl|\frac{u(x_0+\e\xi)-u(x_0)}{\e}\Bigr|^p dx_0 &\le \frac{k^{p-1}}{(\e')^d} \sum_{l=1}^k \int_{Q_\e^{j_{l-1}}}\int_{Q_\e^{j_l}}\Bigl|\frac{u(x_l)-u(x_{l-1})}{\e}\Bigr|^p dx_l\, dx_{l-1} \\
&\le \frac{k^{p-1}}{(\e')^d} \sum_{l=1}^k \int_{Q_\e^{j_{l-1}}}\int_{B_{\e r}(x_{l-1})}\Bigl|\frac{u(y)-u(x_{l-1})}{\e}\Bigr|^p dy\, dx_{l-1}\, .
\end{align*}
Then, by the change of variable $y=x_{l-1}+\e\xi'$
\begin{align*}
\int_{Q_\e^{j_0}}\Big|\frac{u(x_0+\e\xi)-u(x_0)}{\e}\Big|^p dx &\le \Big(\frac{\e}{\e'}\Big)^d k^{p-1} \sum_{l=1}^k \int_{B_r}\int_{Q_\e^{j_{l-1}}}\Big|\frac{u(x_{l-1}+\e\xi')-u(x_{l-1})}{\e}\Big|^p dx_{l-1}\, d\xi' \\
&\le C k^{p-1}\int_{B_r}\int_{Q_\e(j_0)}\left|\frac{u(x+\e\xi')-u(x)}{\e}\right|^p dx\, d\xi'\, ,
\end{align*}
with $Q_\e(j_0)=\bigcup_{l=1}^{k-1} Q_\e^{j_l}$ and $C$ a positive constant depending on $r$ and $d$.
Since the sets $\{Q_\e(j_0) \,|\, j_0\in\mathcal{L}_\e\}$ overlap at most $k-1$ times, by summing over $j_0\in\mathcal{L}_\e$ such that $Q_\e^{j_0}\cap E\not=\emptyset$ we get that
$$\int_E\left|\frac{u(x+\e\xi)-u(x)}{\e}\right|^pdx \le C (|\xi|^p+1) G_\e^r(u,\tilde E_{\e,\xi}).$$
Since dist$(E,\tilde E_{\e,\xi}^c)<\e r$ then $\tilde E_{\e,\xi}\subset E_{\e,\xi}$ and the result follows.
\end{proof}

As a consequence of Lemma \ref{boundlemma} and Theorem \ref{ext} we infer the following result.

\begin{corollary}\label{boundlemma-lip}
For any open set $A\in\mathcal{A}^{\rm reg}(\Omega)$ and $r>0$ there exist two positive constants $C=C(A)$ and $\e_0=\e_0(r,A)$ such that for every $\xi\in\mathbb{R}^d$ and $u\in L^p(A;\mathbb{R}^m)$ there holds
$$\int_{A_\e(\xi)}\left|\frac{u(x+\e\xi)-u(x)}{\e}\right|^p dx \le C(|\xi|^p+1) \big( G_\e^r(u,A)+\|u\|_{L^p(A;\rr^m)}^p \big),$$
for every $\e<\e_0$.
\end{corollary}

\begin{remark}[short-range control]\label{boundremark}
From assumption (H1) and Corollary \ref{boundlemma-lip} we deduce that
for every $A\in\mathcal{A}^{\rm reg}(\Omega)$ there exists a positive constant $C=C(A)$ such that
for every $u\in L^p(\Omega;\mathbb{R}^m)$ there exist $\e_0,r'$ depending on $\|u\|_{L^p}$ such that
$$F_\e(u,A) \le C (G_\e^{r'}(u,A)+1)$$
for every $\e<\e_0$.
\end{remark}

\subsection{Compactness}\label{compactness}

We now discuss the compactness in the strong $L^p$ topology of sequences of functions with uniformly bounded energy.


\begin{theorem}\label{kolcom}
Let $A$ be any open Lipschitz set of $\rr^d$ with $\partial A$ bounded. Let $\{u_\e\}_\e\subset L^p(A;\mathbb{R}^m)$ be such that for some $r>0$
$$\sup_{\e>0} \left\{\| u_\e\|_{L^p(A;\rr^m)}+G_\e^{r}(u_\e,A)\right\}<+\infty.$$
If $A$ is unbounded, assume in addition that for any $\eta>0$ there exists $r_\eta>0$ such that
\begin{equation}\label{boundedtails}
\sup_{\e>0}\|u_\e\|_{L^p(A\backslash B_{r_\eta})}<\eta.
\end{equation}
Then, given $\e_j\to 0$, $\{u_{\e_j}\}_j$ is relatively compact in $L^p(A;\mathbb{R}^m)$ and every limit of a converging subsequence is in $W^{1,p}(A;\mathbb{R}^m)$.
\end{theorem}
\begin{proof}
By Theorem \ref{ext}, there exists $\tilde A\Supset A$, $\tilde r>0$ and  $\tilde u_\e \in L^p(\tilde A;\rr^m)$ such that $\tilde u_\e=u_\e$ on $A$ and
 $$
 \|\tilde u_\e\|^p_{L^p(\tilde A;\rr^m)}+G_\e^{\tilde r}(\tilde u_\e,\tilde A)\le C\|u_\e\|^p_{L^p(A;\rr^m)} + C G_\e^{r}(u_\e,A).
 $$
Set $R=$dist$(A,\tilde A^c)$ and let $\{\phi_\eta\}_{\eta<R}$ be a family of mollifiers; i.e., $\phi_\eta(x)=\phi_1(x/\eta)/\eta^d$, where $\phi_1\in C^\infty_c(\mathbb{R}^d)$, $\phi_1\ge 0$, supp$(\phi_1)\subset B_1$ and $\|\phi_1\|_{L^1(B_1)}=1$.
Note that the convolution product $\tilde u_\e*\phi_\eta(x)$ is well defined for every $x\in A$ and by the standard  properties of the convolution $\tilde u_\e*\phi_\eta\in C^\infty(A;\mathbb{R}^m)$ and
\begin{equation}\label{AAest}
\begin{aligned}
& \| \tilde u_\e*\phi_\eta\|_{L^\infty(A;\rr^m)}\le \| \tilde u_\e\|_{L^p(\tilde A;\rr^m)}\|\phi_\eta\|_{L^{p'}(B_\eta)}=\| \tilde u_\e\|_{L^p(\tilde A;\rr^m)}\|\phi_1\|_{L^{p'}(B_1)}\eta^\frac{d(1-p')}{p'} \\
& \| \nabla(\tilde u_\e*\phi_\eta)\|_{L^\infty(A;\rr^m)}\le \|\tilde  u_\e\|_{L^p(\tilde A;\rr^m)}\|\nabla\phi_\eta\|_{L^{p'}(B_\eta)}=\| \tilde u_\e\|_{L^p(\tilde A;\rr^m)}\|\nabla\phi_1\|_{L^{p'}(B_1)}\eta^{\frac{d(1-p')}{p'}-1}.
\end{aligned}
\end{equation}
Moreover, by Jensen's inequality we have
\begin{equation}\label{modcont1}
\begin{aligned}
\| u_\e-\tilde u_\e*\phi_\eta\|_{L^p(A;\rr^m)}^p &= \int_{A}\Big |u_\e(x)-\int_{B_\eta}\tilde u_\e(x-y)\phi_\eta(y)dy\Big|^p dx \\
&\le \int_{A}\int_{B_\eta}|\tilde u_\e(x)-\tilde u_\e(x-y)|^p\phi_\eta(y)dy\,dx
 \\
&= \int_{B_1}\int_{A}|\tilde u_\e(x)-\tilde u_\e(x+\eta\xi)|^p\phi_1(\xi) dx \, d\xi.
\end{aligned}
\end{equation}
Taking $\eta=m\tilde r\e$ with $m\in\mathbb{N}$, by \eqref{modcont1}  and Jensen's inequality we get
\begin{align*}
\| u_\e-\tilde u_\e*\phi_{m\tilde r\e}\|_{L^p(A)}^p &\le \int_{B_1}\int_{A}(m\e)^p\left|\sum_{l=0}^{m-1}\frac{\tilde u_\e(x+(l+1)\e \tilde r\xi)-\tilde u_\e(x+l\e \tilde r\xi)}{m\e}\right|^p dx\,d\xi \\
&\le m^{p-1}\e^p\sum_{l=0}^{m-1}\int_{B_1}\int_{A}\left|\frac{\tilde u_\e(x+(l+1)\e \tilde r\xi)-\tilde u_\e(x+l\e \tilde r\xi)}{\e}\right|^p dx\,d\xi \\
&= m^{p-1}\e^p\sum_{l=0}^{m-1}\int_{B_1}\int_{A-l\e\tilde r\xi}\left|\frac{\tilde u_\e(x'+\e \tilde r\xi)-\tilde u_\e(x')}{\e}\right|^p dx' d\xi,
\end{align*}
where in the last equality we have used the change of variable $x'=x+l\e\tilde r\xi$.
Since $A-l\e \tilde r\xi\subset\tilde A_\e(\tilde r\xi)$ for every $0\le l\le m-1$, we get
$$\| u_\e-\tilde u_\e*\phi_{m\tilde r\e}\|_{L^p(A;\rr^m)}^p \le (m\e)^p \int_{B_1}\int_{\tilde A_\e(\tilde r\xi)}\left|\frac{\tilde u_\e(x'+\e \tilde r\xi)-\tilde u_\e(x')}{\e}\right|^p dx' d\xi,$$
and, through the change of variable $\xi'=\tilde r\xi$, we obtain
\begin{equation}\label{modcont2}
\| u_\e-\tilde u_\e*\phi_{m\tilde r\e}\|_{L^p(A;\rr^m)}^p \le \frac{(m\e)^p}{\tilde r^d}G_\e^{\tilde r}(\tilde u_\e, \tilde A).
\end{equation}
By \eqref{AAest} and \eqref{boundedtails} if $A$ is unbounded, we can find $A_\eta\Subset A$ such that
\begin{equation}\label{localAA}
\sup_{\e>0}\|\tilde u_\e*\phi_\eta\|_{L^p(A\setminus A_\eta;\rr^m)}<\eta.
\end{equation}
Given  $\e_j<\tilde r^{-1}\eta$ and setting $\eta_j=\lceil\eta/(\tilde r\e_j)\rceil \tilde r\e_j\ge \eta$, from \eqref{AAest} we get in particular
$$\|\tilde u_{\e_j}*\phi_{\eta_j}\|_{L^\infty(A_\eta;\rr^m)}\le C\eta^{\frac{d(1-p')}{p'}},\quad\|\nabla(\tilde u_{\e_j}*\phi_{\eta_j})\|_{L^\infty(A_\eta;\rr^m)}\le C\eta^{\frac{d(1-p')}{p'}-1}.$$
Hence by Ascoli-Arzel\'a's Theorem, for every fixed $\eta$ the sequence $\{\tilde u_{\e_j}*\phi_{\eta_j}\}_j$ is precompact in $C(A_\eta;\mathbb{R}^m)$ and therefore it is totally bounded, that is there exists a finite set of functions $\{g_k\}_{k=1}^L\subset C(A_\eta;\mathbb{R}^m)$ such that for every $j\in\NN$
\begin{equation}\label{totlim}
\| \tilde u_{\e_j}*\phi_{\eta_j}-g_{k}\|_{L^p(A_\eta;\rr^m)}\le |A_\eta|^{1\over p}\| \tilde u_{\e_j}*\phi_{\eta_j}-g_{k}\|_{L^\infty(A_\eta;\rr^m)}<\eta
\end{equation}
for some $k\in\{1,\dots, L\}$.
So, by \eqref{modcont2}, \eqref{localAA} and \eqref{totlim}, denoting by  $\tilde g_k$ the extension of $g_k$ which equals zero outside $A_\eta$, we have
\begin{align*}
\| u_{\e_j}-\tilde g_{k}\|_{L^p(A;\rr^m)} &\le \|u_{\e_j}-\tilde u_{\e_j}*\phi_{\eta_j}\|_{L^p(A;\rr^m)}+\| \tilde u_{\e_j}*\phi_{\eta_j}-\tilde g_{k}\|_{L^p(A;\rr^m)} \\
&= \|u_{\e_j}-\tilde u_{\e_j}*\phi_{\eta_j}\|_{L^p(A;\rr^m)}+\| \tilde u_{\e_j}*\phi_{\eta_j}-g_{k}\|_{L^p(A_\eta;\rr^m)}\\
&\quad  +\|\tilde u_{\e_j}*\phi_{\eta_j}\|_{L^p(A\setminus A_\eta)}\\
& \le \eta_j\Big(\frac{1}{\tilde r^{d/p-1}} G_{\e_j}^{\tilde r }(\tilde u_{\e_j}, A)^{1\over p}+2\Big) \le C \eta;
\end{align*}
i.e., $\{u_{\e_j}\}_j$ is totally bounded and hence relatively compact in $L^p(A;\mathbb{R}^m)$. Finally Proposition \ref{G-liminf-lem} yields that every limit function is in $W^{1,p}(A;\mathbb{R}^m)$.
\end{proof}

\begin{remark}\label{compactness-remk}
From the previous theorem we deduce a compactness result on any open set $A$ (without regularity assumptions on the boundary) with respect to the local $L^p$-topology.
Namely, that every bounded sequence $\{u_\e\}$ with bounded $G_\e^r$-energy on $A$ is precompact in $L^p(A';\rr^m)$ for every $A'\Subset A$. Indeed, it suffices to apply the previous theorem with a set $A''$ with Lipschitz boundary in the place of $A$, with $A'\Subset A''\Subset A$.
\end{remark}

\subsection{Poincar\'e inequalities}

We complete this section with the analog of Poincar\'e inequalities.

\begin{proposition}[Poincar\'e inequality]\label{poincare}
Let $r>0$ and let  $A$ be an open set in $\rr^d$ such that $A\subseteq (a,b)\times\rr^{d-1}$ for some $a,b\in\rr$.
Then there exists a positive constant $C=C(A)$ such that
$$\int_A |u(x)|^p dx\le C G_\e^{r}(u,A)$$
for every $\e>0$ and for any $u\in L^p(A;\mathbb{R}^d)$ with $u(x)=0$ for almost every $x\in A$ with $\dist(x,A)\leq  \e r$.
\end{proposition}
\begin{proof}
We identify $u$ with its extension to the whole $\rr^d$ that equals zero outside $A$. Up to a change of variables we may assume $a=0$ and $b=1$.  Set $r':=\frac{r}{\sqrt{d+3}}$ and for every $j\in \{0,\dots,N:=\lfloor1/r'\e\rfloor\}$ and $l\in\mathbb{Z}^{d-1}$ denote by $x_j^l$ an independent variable lying in
$$Q_\e^{j,l}:=(jr'\e,(j+1)r'\e)\times Q_\e^l,\text{ with }Q_\e^l:=l+(0,r'\e)^{d-1}.$$
By the boundary assumption, for any $0\le k\le N$ we have $u(x_k^l)=\sum\limits_{j=1}^k (u(x_j^l)-u(x_{j-1}^l))$. Thus,  by Jensen's inequality we get
$$|u(x_k^l)|^p\le k^{p-1}\sum_{j=1}^k |u(x_j^l)-u(x_{j-1}^l)|^p,$$
and, integrating in $x_0^l, x_1^l,\dots, x_N^l$,
$$(r'\e)^{d(N-1)}\int_{Q_\e^{k,l}}|u(x_k^l)|^p dx_k^l\le (r'\e)^{d(N-2)}k^{p-1}\sum_{j=1}^k\int_{Q_\e^{j-1,l}}\int_{Q_\e^{j,l}}|u(x_j^l)-u(x_{j-1}^l)|^p dx_j^l dx_{j-1}^l.$$
Now, since  $Q_\e^{j,l}\subset B_{r\e}(x_{j-1}^l)$, we have
\begin{align*}
\int_{Q_\e^{k,l}}|u(x_k^l)|^p dx_k^l &\le \frac{k^{p-1}}{(r'\e)^d}\sum_{j=1}^k\int_{Q_\e^{j-1,l}}\int_{B_{r\e}(x_{j-1}^l)}|u(y)-u(x_{j-1}^l)|^p dy dx_{j-1}^l \\
&= \frac{k^{p-1}}{(r'\e)^d}\int_{(0,kr'\e)\times Q_\e^l}\int_{B_{r\e}(x)}|u(y)-u(x)|^p dy\,dx
 \\
&\le \frac{N^{p-1}}{r'^d}\int_{(0,1)\times Q_\e^l}\int_{B_{r}}|u(x+\e\xi)-u(x)|^p d\xi dx.
\end{align*}
By summing over $0\le k\le N$ and $l\in\mathbb{Z}^{d-1}$ both the left- and the right-hand sides  we get
$$\int_{A}|u(x)|^p dx\le \frac{1}{r'^{d+p}}\int_{A}\int_{B_{r}}\left|\frac{u(x+\e\xi)-u(x)}{\e}\right|^p d\xi dx=\frac{1}{r'^{d+p}}G_\e^r(u,A),$$
which proves the claim.
\end{proof}

\begin{proposition}[Poincar\'e-Wirtinger inequality]\label{pwineq}
Let $r>0$ and let $A$ be a connected open set of $\rr^d$ with Lipschitz boundary. Then for every measurable set $E\subset A$ with $|E|>0$ there exists a positive constant $C=C(A,E)$ such that for any $u\in L^p(A;\mathbb{R}^m)$ and  $\e\in (0,1)$
$$\int_A |u(x)-u_E|^p dx\le C G^{r}_\e(u,A).$$
\end{proposition}
\begin{proof}
We argue by contradiction. Suppose that for any positive integer $j$ there exists $\e_j>0$ and $u_j\in L^p(A;\mathbb{R}^m)$ such that
$$\int_{A}|u_j(x)-(u_j)_E|^p dx > j G_{\e_j}^{r}(u_j,A).$$
Thus, letting
$$
\tilde u_j:=\frac{u_j-(u_j)_E\quad\quad }{\|u_j-(u_j)_E\|_{L^p(A;\rr^m)}},
$$
we have  $\|\tilde u_j\|_{L^p(A;\rr^m)}\equiv1$, $(\tilde u_j)_E\equiv0$ and
\begin{equation}\label{contradiction}
G_{\e_j}^{r}(\tilde u_j,A)<\frac{1}{j}.
\end{equation}
We may assume, up to passing to a subsequence, that $\e_j\to \e_0\in [0,1]$. If $\e_0=0$, Theorem \ref{kolcom} yields that up to subsequences $\tilde u_j\to u$ in $L^p(A;\rr^m)$ with $u\in W^{1,p}(A;\mathbb{R}^m)$,  $\|u\|_{L^p(A;\rr^m)}=1$ and $u_E=0$.
By \eqref{contradiction} and Proposition \ref{G-liminf-lem}, we get
$$\int_{B_{r}}\int_A |Du(x)\xi|^p dx \, d\xi=0.$$
Hence, $u$ is constant almost everywhere and we reach a contradiction.
If otherwise $\e_0>0$, up to passing to a further subsequence, $\tilde u_j\rightharpoonup u$ weakly in $L^p(A;\rr^m)$ and so for any $\xi\in B_r$
$$\frac{\tilde u_j(x+\e_j\xi)-\tilde u_j(x)}{\e_j}\rightharpoonup\frac{u(x+\e_0\xi)-u(x)}{\e_0}\ \text{weakly in } L^p(A;\rr^m).$$
Fatou's Lemma yields
$$0=\liminf_{j\to\infty} G_{\e_j}^{r}(\tilde u_j,A) \ge \int_{B_{r}}\liminf_{j\to\infty}\int_{A_{\e_j}(\xi)}\left|\frac{\tilde u_j(x+\e_j\xi)-\tilde u_j(x)}{\e_j}\right|^pdx \, d\xi \ge G_{\e_0}^{r}(u,A).$$
Hence,  $u$ is constant almost everywhere on $A$, which again leads to a contradiction.
\end{proof}

\section{The  $\Gamma$-limit of  $G_\e[a_\e]$}\label{particular}

In this section we prove the $\Gamma$-convergence of the family of functionals $G_\e[a_\e]$ (in the notation introduced in Definition \ref{convfun}), under suitable assumptions on the convolution kernels $a_\e$. Such a result on the one hand shows a non-trivial example of $\Gamma$-converging functionals, on the other hand it allows, by comparison, to deduce that the $\Gamma$-limits of the family of functionals $F_\e$ satisfy standard $p$-growth assumptions in a Sobolev-space setting.

Throughout this section, if $a\in L^1(\rr^d)$ we use the notation $\mu_a$ for the measure
\begin{equation}\label{measures-a}
\mu_a(A)=\int_A a(\xi)\,d\xi.
\end{equation}

\begin{theorem}\label{G-conv-conv}
Let $A\in\mathcal{A}^{\rm reg}(\Omega)$ and  let $\{a_\e\}\subset L^1({\mathbb{R}^d})$ be a family of non-negative functions  such that
\begin{equation}\label{gh1}
\limsup_{\e\to0}\int_{\mathbb{R}^d}a_\e(\xi)(|\xi|^p+1)d\xi < +\infty
\end{equation}
and such that for every $\delta>0$ there exists $r_\delta$ satisfying
\begin{equation}\label{gh2}
\limsup_{\e\to0}\int_{B_{r_\delta}^c}a_\e(\xi)|\xi|^pd\xi < \delta.
\end{equation}
Suppose that the measures $\mu_{a_\e}$  weakly* converge to a measure $\mu_a$ for some non-trivial $a\in L^1(\mathbb{R}^d)$.
Then
$$\Gamma(L^p)\text{-}\lim_{\e\to0} G_\e[a_\e](u,A)=\begin{cases}\displaystyle
\int_{\mathbb{R}^d}a(\xi)\int_A |Du(x)\xi|^p dx\,d\xi &\text{ if }u\in W^{1,p}(A;\mathbb{R}^m)\\ \displaystyle
+\infty &\text{otherwise.}\end{cases}$$
\end{theorem}
\begin{proof}
It is a straightforward consequence of Propositions \ref{G-liminf-lem} and \ref{G-limsup-lem} below.
 \end{proof}
\begin{remark}
Note that under the the assumptions of Theorem \ref{G-conv-conv} $a$ satisfies \eqref{ipoa-p},
from which we have that $\Gamma(L^p)\text{-}\lim_{\e\to0} G_\e[a_\e](u,A)$ is finite on $W^{1,p}(A;\mathbb{R}^m)$. Moreover, observe that \eqref{gh2} corresponds to assumption (H2), which is crucial  to ensure the locality of the  $\Gamma$-limit, as shown by the example below.
\end{remark}

\begin{example}
Given $\Omega=(0,1)$, consider the functions
$$a_\e(\xi)=\begin{cases}
1 & |\xi|\leq1 \\
\e^p & {1\over2\e}-1<\xi<{1\over2\e}+1 \\
0 &\text{otherwise.}
\end{cases}$$
The densities $f_\e(x,\xi,z)=a_\e(\xi)|z|^p$ satisfy hypotheses (H0), (H1) and conditions \eqref{ass-bound} but not the locality assumption \eqref{ass-local}.
Now, for any given $u\in W^{1,p}(0,1)$, and any $u_\e$ converging to $u$ in $L^p(0,1)$ as $\e\to0$, we have
\begin{align*}
G_\e[a_\e](u_\e) &= \int_{-\infty}^{+\infty}a_\e(\xi)\int_{0\vee(-\e\xi)}^{1\wedge(1-\e\xi)}\left|\frac{u_\e(x+\e\xi)-u_\e(x)}{\e}\right|^p dx \, d\xi \\
&= G_\e^1(u_\e)+\int_{\frac{1}{2\e}-1}^{\frac{1}{2\e}+1}\int_0^{1-\e\xi}|u_\e(x+\e\xi)-u_\e(x)|^p dx \, d\xi.
\end{align*}
Taking the limit as $\e\to0$, by Proposition \ref{G-conv-conv} and Lebesgue's Dominated Convergence Theorem we get
$$\Gamma\text{-}\lim_{\e\to0}G_\e[a_\e](u)=\frac{2}{p+1}\int_0^1|u'(x)|^p dx+2\int_0^{1\over2}\Big|u\Big(x+{{1\over2}}\Big)-u(x)\Big|^p dx.$$
The limit functional is still defined on $W^{1,p}(0,1)$ but does not have a local representation.
\end{example}

In the following two lemmas we deal with the $\Gamma$-$\liminf$ and the $\Gamma$-$\limsup$ of the family of functionals $G_\e[a_\e]$ separately.

\begin{proposition}\label{G-liminf-lem}
Let $\{a_\e\}\subset L^1(\mathbb{R}^d)$ be a family of non-negative functions such that the measures $\mu_{a_\e}$, defined in \eqref{measures-a}, weakly* converge to a measure $\mu_a$ for some non-trivial $a\in L^1(\mathbb{R}^d)$.
Then for every $u\in L^p(\Omega;\mathbb{R}^m)$ and $A\in\mathcal{A}(\Omega)$
\begin{equation}\label{Gliminf}
\Gamma\text{-}\liminf_{\e\to0} G_\e[a_\e](u,A) \ge
\begin{cases} \displaystyle
\int_{\mathbb{R}^d}a(\xi)\int_A |Du(x)\xi|^p dx\,d\xi &\text{if }u\in W^{1,p}(A;\mathbb{R}^m) \\
+\infty &\text{otherwise.}
\end{cases}
\end{equation}
\end{proposition}
\begin{proof}
When needed, we will identify the functions with their extensions equal to $0$ outside $\Omega$. We will use a slicing procedure, so we first deal with the one-dimensional case; i.e., when $\Omega\subset\mathbb{R}$ and, for the sake of simplicity, $A=(0,1)$. In this case the energy reads as
$$G_\e[a_\e](u,A) = \int_{-\infty}^{+\infty}a_\e(\xi)\int_{0\vee(-\e\xi)}^{(1-\e\xi)\wedge1}\Bigl|\frac{u(x+\e\xi)-u(x)}{\e}\Bigr|^p dx\,d\xi.$$
For every fixed $\xi>0$, setting $N=\lfloor1/(\e\xi)\rfloor-1$ we have
\begin{align*}
\int_0^{1-\e\xi}\Bigl|\frac{u(x+\e\xi)-u(x)}{\e}\Bigr|^p dx &\ge \sum_{k=1}^{N-1} \int_{k\e\xi}^{(k+1)\e\xi} \Bigl|\frac{u(x+\e\xi)-u(x)}{\e}\Bigr|^p dx \\
&= \sum_{k=1}^{N-1} \e\xi \int_0^1 \Bigl|\frac{u(t\e\xi+(k+1)\e\xi)-u(t\e\xi+k\e\xi)}{\e}\Bigr|^p dt \\
&= \xi^p\int_0^1 \e\xi\sum_{k=1}^{N-1} \Bigl|\frac{u(t\e\xi+(k+1)\e\xi)-u(t\e\xi+k\e\xi)}{\e\xi}\Bigr|^p dt,
\end{align*}
where we have used the change of variable $x=\e\xi(t+k)$.
Let $u_{\e,\xi,t}:\rr\to\mathbb{R}^m$ be the piecewise-affine function that interpolates the values $\{u((k+t)\e\xi)\}_k$. We then have
$$\int_0^{1-\e\xi}\Bigl|\frac{u(x+\e\xi)-u(x)}{\e}\Bigr|^p dx \ge \xi^p \int_0^1 \int_0^{N\e\xi}|u_{\e,\xi,t}'(s)|^p ds\,dt.$$
The same analysis when $\xi<0$ implies
$$\int_{0\vee(-\e\xi)}^{(1-\e\xi)\wedge1} \Bigl|\frac{u(x+\e\xi)-u(x)}{\e}\Bigr|^p dx \ge |\xi|^p \int_0^1\int_{2\e\xi}^{1-2\e\xi} |u'_{\e,\xi,t}(s)|^p ds\,dt,$$
for every $\xi\in\mathbb{R}$.
We can generalize the previous argument to any $A\subset\Omega$ open subset and the inequality above reads
\begin{equation}\label{1D-est}
\int_{A_\e(\xi)}\Bigl|\frac{u(x+\e\xi)-u(x)}{\e}\Bigr|^p dx \ge |\xi|^p \int_0^1\int_{\tilde A_\e(\xi)}|u'_{\e,\xi,t}(s)|^p ds\,dt,
\end{equation}
for every $\xi\in\mathbb{R}$, where $\tilde A_\e(\xi)=\{s\in A\,|\text{ dist}(s,A^c)>2\e|\xi|\}$.

We now extend \eqref{1D-est} to any dimension by a slicing method.
For any $\xi\in\mathbb{R}^d\backslash\{0\}$ define $\Pi_{\xi}=\{y\in\mathbb{R}^d\,|\,y\cdot\xi=0\}$, and for any $y\in\Pi_{\xi}$ set
\begin{align}
& \Omega_{y,\xi} = \Bigl\{s\in\mathbb{R}\,|\, y+s{\xi\over|\xi|}\in \Omega\Bigl\}, & A_{y,\xi} = \Bigl\{s\in\mathbb{R}\,|\, y+s{\xi\over|\xi|}\in A\Bigl\}, \nonumber\\
& A_{y,\xi}^\e = \Bigl\{s\in\mathbb{R}\,|\, y+s{\xi\over|\xi|}\in A_\e(\xi)\Bigl\}, & \tilde A_{y,\xi}^\e = \Bigl\{s\in\mathbb{R}\,|\text{ dist}(s,A_{y,\xi}^c)>2\e|\xi|\Bigl\},\nonumber\\
&u_{y,\xi}(s):=u\Big(y+s{\xi\over|\xi|}\Big) \quad s\in \rr.&\label{slice}
\end{align}
Fubini's Theorem yields
$$\int_{A_\e(\xi)}\Bigl|\frac{u(x+\e\xi)-u(x)}{\e}\Bigr|^p dx=\int_{\Pi_\xi}\int_{A_{y,\xi}^\e}\Bigl|\frac{u_{y,\xi}(s+\e|\xi|)-u_{y,\xi}(s)}{\e}\Bigr|^p ds\, dy.$$
By \eqref{1D-est} we get
\begin{equation}\label{slicing-G-liminf}
\int_{A_\e(\xi)}\Bigl|\frac{u(x+\e\xi)-u(x)}{\e}\Bigr|^p dx \ge |\xi|^p \int_{\Pi_\xi} \int_0^1 \int_{\tilde A_{y,\xi}^\e} |u_{y,\e,\xi,t}'(s)|^p ds\,dt\,dy,
\end{equation}
where $u_{y,\e,\xi,t}:\e|\xi|\mathbb{Z}\cap \Omega_{y,\xi}\to\mathbb{R}^m$ denotes the piecewise-affine function that interpolates the values $\{u_{y,\xi}((k+t)\e|\xi|)\}_k$.




Let $\{\e_j\}$ be a sequence converging to $0$ and  let $u_j\to u$ in $L^p(\Omega;\mathbb{R}^m)$. Without loss of generality we may assume that $G_{\e_j}[a_{\e_j}](u_j,A)$ is uniformly bounded. Let $(u_j)_{y,\xi}$ denote the corresponding slicing functions defined by \eqref{slice}. Then we have that $(u_j)_{y,\xi}\to u_{y,\xi}$ in $L^p(\Omega_{y,\xi};\mathbb{R}^m)$. Moreover,  by \cite{bra1} Lemma 3.36  and \cite{bra1}  Remark 3.37, $(u_j)_{y,\e_j,\xi,t}\to u_{y,\xi}$ in $L^p(\Omega_{y,\xi};\mathbb{R}^m)$ for almost every $y\in\Pi_\xi$ and for almost every $t\in(0,1)$.
Now, we set
$$\varphi_j(\xi) := \int_{\Pi_\xi} \int_0^1 \int_{\tilde A_{y,\xi}^{\e_j}}|(u_j)'_{y,\e_j,\xi,t}(s)|^p ds\,dt\,dy,$$
and claim that
$$
\liminf_{j\to+\infty}\varphi_j(\xi)<+\infty \quad \text{for almost every } \xi\in E,
$$
where $E$ is the support of $a$.
Indeed, reasoning by contradiction let $E'\subset E$ with $|E'|>0$ be such that $\varphi_j(\xi)\to+\infty$ for every $\xi\in E'$.
Then, we can find a set $E''\subset E'$ with $|E''|>0$ such that for every $m>0$ there exist $j_m\in\mathbb{N}$ such that $\varphi_j(\xi)\ge m$ for every $j>j_m$ and $\xi\in E''$.
Thus, from \eqref{slicing-G-liminf} and the boundedness of $G_{\e_j}[a_{\e_j}](u_j,A)$, for every $R>0$ we get
$$
\int_{B_R\cap E''} a_{\e_j}(\xi)|\xi|^p\,d\xi \le \frac{C}{m}.
$$
Taking the limit as $j\to+\infty$, the arbitrariness of $m$ and $R$ lead to a contradiction.
Now, for every $\xi\in E$
we have that
\begin{equation}\label{semi-cont-inf}
\liminf_{j\to\infty} \varphi_j(\xi) \ge \int_{\Pi_\xi} \int_{A_{y,\xi}}|u_{y,\xi}'(s)|^p ds\, dy\, .
\end{equation}
Indeed,
for almost every $y\in\Pi_\xi$ and for almost every $t\in(0,1)$ up to subsequences $(u_j)_{y,\e_j,\xi,t}'$ is bounded in $L^p$ and so $(u_j)_{y,\e_j,\xi,t}$ weakly converges to $u_{y,\xi}$ in $W^{1,p}(A_{y,\xi};\mathbb{R}^m)$.
Then, Fatou's Lemma and the lower semicontinuity of the $L^p$-norm with respect to the weak convergence yields \eqref{semi-cont-inf}.
Now set $\tilde\varphi_j(\xi)=\inf_{k\ge j}\varphi_k(\xi)$, which converges almost everywhere to $\liminf_j \varphi_j(\xi)$, and let $R>0$ be fixed.
By Egorov's and Lusin's theorems, for any $\delta>0$ there exists $E_\delta\subset E\cap B_R$ with $|E\cap B_R\backslash E_\delta|<\delta$ such that the functions $\tilde\varphi_j$ are continuous and converge uniformly to $\liminf_j \varphi_j$ on $E_\delta$.
From \eqref{semi-cont-inf} and by the weak* convergence of $\mu_{a_\e}$, we get
$$\lim_{j\to\infty} \int_{E_\delta} a_{\e_j}(\xi)|\xi|^p \tilde\varphi_j(\xi) \,d\xi \ge \int_{E_\delta} a(\xi)|\xi|^p \int_{\Pi_\xi}\int_{A_{y,\xi}}|u'_{y,\xi}(s)|^p \,ds\,dy\,d\xi\, .$$
Multiplying both members of \eqref{slicing-G-liminf} by $a_{\e_j}(\xi)$, integrating in $\xi$ and then taking the limit as $j\to\infty$, we obtain
$$\liminf_{j\to\infty} G_{\e_j}[a_{\e_j}](u_{\e_j},A) \ge \int_{\mathbb{R}^d} a(\xi)|\xi|^p \int_{\Pi_\xi}\int_{A_{y,\xi}}|u'_{y,\xi}(s)|^p \,ds\,dy\,d\xi$$
by the arbitrariness of $\delta$ and $R$.
If $u\in W^{1,p}(A;\mathbb{R}^m)$
$$u'_{y,\xi}(s)=\frac{1}{|\xi|}Du(x)\xi,\quad x=y+s\frac{\xi}{|\xi|}$$
then Fubini's Theorem leads to \eqref{Gliminf}.
Thus it remains to prove that  $u\in W^{1,p}(A;\mathbb{R}^m)$. 
If $\liminf_{\e\to0}G_\e[a_\e](u,A)<+\infty$ then
$$\int_E a(\xi)|\xi|^p \int_A \Bigl|\frac{\partial u}{\partial \xi}(x)\Bigr|^p  dx \, d\xi<+\infty$$
and therefore,
\begin{equation}\label{domain-est}
\int_A\left|\frac{\partial u}{\partial \xi}(x)\right|^p dx<+\infty, \quad\text{for almost every }\xi\in E.
\end{equation}
In particular, we can find linearly independent points $\xi_1,\dots,\xi_d\in E$ for which \eqref{domain-est} is satisfied, since otherwise $E$ would be contained in some hyperplane and it would be negligible.
Hence $u\in W^{1,p}(A;\mathbb{R}^m)$, and the thesis follows.
\end{proof}

\begin{proposition}\label{G-limsup-lem}
Let $\{a_\e\}\subset L^1(\mathbb{R}^d)$ be a family of non-negative functions satisfying \eqref{gh1}.
Then, for every $A\in\mathcal{A}^{\rm reg}(\Omega)$ and $u\in W^{1,p}(A;\mathbb{R}^m)$
\begin{equation}\label{upp-growth-cond}
\Gamma\text{-}\limsup_{\e\to0}G_\e[a_\e](u,A)\le C\int_A|Du(x)|^p dx\, ,
\end{equation}
for some constant $C>0$.
Suppose in addition that, for every $\delta>0$ there exists $r_\delta$ such that \eqref{gh2} holds
and the measures $\mu_{a_\e}$ as in \eqref{measures-a}  weakly* converge to a measure $\mu_a$ for some non-trivial $a\in L^1(\mathbb{R}^d)$.
Then, for every $A\in\mathcal{A}^{\rm reg}(\Omega)$ and $u\in W^{1,p}(A;\mathbb{R}^m)$
\begin{equation}\label{Glimsup}
\Gamma\text{-}\limsup_{\e\to0} G_\e[a_\e](u,A) \le \int_{\mathbb{R}^d}a(\xi)\int_A |Du(x)\xi|^p dx \, d\xi.
\end{equation}
\end{proposition}
\begin{proof}
By a density argument, we may restrict to the case $u\in C^\infty_c(\mathbb{R}^d;\mathbb{R}^m)$.
By Remark \ref{boundremark}, in order to prove \eqref{upp-growth-cond} it is sufficient to estimate the upper limit of $G_\e^{r'}(u,A)$, for some $r'>0$.
For every $x\in\mathbb{R}^d$ we have
$$\frac{u(x+\e\xi)-u(x)}{\e}=\int_0^1 Du(x+s\e\xi)\xi ds,$$
and by Jensen's inequality and Fubini's Theorem we get
\begin{align*}
G_\e^{r'}(u,A) &\le \int_{B_{r'}}|\xi|^p\int_{A_\e(\xi)}\int_0^1|Du(x+s\e\xi)|^p ds dx \, d\xi \\
&= \int_{B_{r'}}|\xi|^p\int_0^1 \int_{A+B_{\e r_0}}|Du(x)|^p dx\, ds\, d\xi \\
&\le \int_{B_{r'}}|\xi|^p \int_A|Du(x)|^p dx\, d\xi + o(1).
\end{align*}
Taking the $\limsup$ as $\e\to0$ we obtain \eqref{upp-growth-cond}.

We now prove \eqref{Glimsup} under the additional assumption 
\eqref{gh2} and the weak* convergence of $\mu_{a_\e}$ to $\mu_a$.
We split $G_\e[a_\e](u,A)$ as follows
\begin{multline*}
G_\e[a_\e](u,A) = \int_{B_{r_\delta}}a_\e(\xi)\int_{A_\e(\xi)}\Bigl|\frac{u(x+\e\xi)-u(x)}{\e}\Bigr|^p dx \, d\xi \\
+ \int_{B_{r_\delta}^c}a_\e(\xi)\int_{A_\e(\xi)} \Bigl|\frac{u(x+\e\xi)-u(x)}{\e}\Bigr|^p dx \, d\xi
\end{multline*}
for any $\delta>0$, where $r_\delta$ is defined as in assumption \eqref{ass-local}.
Expanding $u(x)$ at the first order when $|\xi|<r_\delta$ we get
$$\int_{B_{r_\delta}}a_\e(\xi)\int_{A_\e(\xi)}\Bigl|\frac{u(x+\e\xi)-u(x)}{\e}\Bigr|^p dx \, d\xi=\int_{B_{r_\delta}}a_\e(\xi) \Bigl(\int_{A_\e(\xi)}|D u(x)\xi|^pdx + o(1)\Bigr) \, d\xi$$
and for $|\xi|>r_\delta$ by assumption \eqref{ass-local}
\begin{multline*}
\int_{B_{r_\delta}^c}a_\e(\xi)\int_{A_\e(\xi)} \Bigl|\frac{u(x+\e\xi)-u(x)}{\e}\Bigr|^p dx \, d\xi \\
\le \int_{B_{r_\delta}^c}a_\e(\xi)\int_A (\| Du\|_{L^\infty(\rr^d)} |\xi|)^p dx \, d\xi \le |A| \| Du\|_{L^\infty(\mathbb{R}^d)}^p\delta.
\end{multline*}
Hence, gathering the inequalities above we obtain
\begin{equation}\label{G-limsup-est}
G_\e[a_\e](u,A) \le \int_{B_{r_\delta}}a_\e(\xi)\int_{A}|Du(x)\xi|^p dx \, d\xi+C\delta
\end{equation}
for some $C>0$.
Letting $\e\to0$  in \eqref{G-limsup-est}, we get
$$\limsup_{\e\to0}G_\e[a_\e](u,A) \le \int_{B_{r_\delta}}a(\xi)\int_{\Omega}|Du(x)\xi|^p dx \, d\xi + C\delta$$
and the arbitrariness of $\delta$ implies \eqref{Glimsup}.
\end{proof}
From the right-hand-side inequality in \eqref{growth-cond}, \eqref{relax-growth-cond}, Propositions \ref{G-liminf-lem} and \ref{G-limsup-lem} the following estimates hold.

\begin{proposition}\label{growthcondremk}
Given $A\in\mathcal{A}^{\rm reg}(\Omega)$, let $\{F_\e(\cdot,A)\}$ be the family of functionals defined by \eqref{loc-functionals} and assume that {\rm(H0$'$), (H1)} and {\rm(H2)} hold. If $F'(u,A)$ is finite, then  $u\in W^{1,p}(A;\mathbb{R}^m)$. Moreover for every $u\in W^{1,p}(A;\mathbb{R}^m)$ we have
\begin{align}\label{growth-cond-below}
F'(u,A) &\ge c\,(\|Du\|_{L^p(A)}^p-|A|), \\
\label{growth-cond-above}
F''(u,A) &\le C(\|Du\|_{L^p(A)}^p+|A|),
\end{align}
for some positive constants $c,C$.
\end{proposition}

\section{An integral-representation result}\label{integralrep}

The main result of this section is the following compactness and  integral-representation result for the $\Gamma$-limits of the families $\{F_\e(\cdot,A)\}_\e$.

\begin{theorem}\label{reprthm}
Given $\Omega$ a bounded open set with Lipschitz boundary, let $F_\e$, $F_\e(\cdot,A)$ be defined by \eqref{functionals} and \eqref{loc-functionals}, respectively, and let assumptions  {\rm(H0$'$), (H1)} and {\rm(H2)} be satisfied.
Then, for every $\e_j\to 0$ there exists a subsequence $\{\e_{j_k}\}\subset\{\e_j\}$ and a Carath\'edory function $f_0:\Omega\times\mathbb{R}^{m\times d}\to[0,+\infty)$ which is quasiconvex in the second variable and satisfies the growth condition
\begin{equation}\label{growthf0}
C_0(|M|^p-1)\le f_0(x,M)\le C_1(|M|^p+1)
\end{equation}
for almost every $x\in\Omega$ and every $M\in\mathbb{R}^{m\times d}$, such that
\begin{equation}\label{limres}
\Gamma(L^p)\text{-}\lim_{k\to +\infty}F_{\e_{j_k}}(u,A)=\begin{cases} \displaystyle \int_A f_0(x,Du(x))dx & \text{if } u\in W^{1,p}(A;\rr^m)\cr
+\infty & \text{otherwise,}
\end{cases}
\end{equation}
for every $A\in\mathcal{A}^{\rm reg}(\Omega)$.
\end{theorem}
For the proof of Theorem \ref{reprthm} we will follow the standard strategy described in \cite[Chapter 9]{bradef}.
%
%
In particular we will apply the following  result.

\begin{theorem}[Theorem 9.1 \cite{bradef}]\label{representationthm}
Let $\Omega$ be a bounded open set, and let  $F:W^{1,p}(\Omega;\mathbb{R}^m)\times\mathcal{A}(\Omega)\to[0,+\infty)$ satisfy:

{\rm(i)} for any $A\in\mathcal{A}(\Omega)$ $F(u,A)=F(v,A)$ if $u=v$ almost everywhere on $A$;

{\rm(ii)}  for any $u\in W^{1,p}(\Omega;\mathbb{R}^m)$ the set function $F(u,\cdot)$ is a restriction of a Borel measure on $\mathcal{A}(\Omega)$;

{\rm(iii)} there exists a constant $C>0$ and $a\in L^1(\Omega)$ such that
$$F(u,A)\le C\int_A (a(x)+|Du(x)|^p)\,dx;$$

{\rm(iv)} $F(u+z,A)=F(u,A)$ for any $u\in W^{1,p}(\Omega;\rr^m)$, $z\in\mathbb{R}^m$ and $A\in\mathcal{A}(\Omega)$;

{\rm(v)} $F(\cdot,A)$ is weakly lower semicontinuous for any $A\in\mathcal{A}(\Omega)$.

\noindent
Then there exists a Carath\'eodory function $f:\Omega\times\mathbb{R}^{m\times d}\to[0,+\infty)$, quasiconvex in the second variable, with $0\le f(x,M)\le c(a(x)+|M|^p)$ for almost every $x\in A$ and every $M\in\mathbb{R}^{m\times d}$, such that
$$F(u,A)=\int_A f(x,Du(x))dx$$
for every $A\in\mathcal{A}(\Omega)$ and $u\in W^{1,p}(\Omega;\mathbb{R}^d)$.
\end{theorem}
We postpone the proof of Theorem \ref{reprthm} to the end of the section as it will be a direct consequence of some propositions that show that the limit functionals satisfy the hypotheses of Theorem \ref{representationthm}.

\subsection{Truncated functionals}
In this subsection we introduce the truncated functionals obtained by limiting the range of interaction in \eqref{functionals} to a fixed threshold $T>0$. We will show that, to some extent, the limit as $T\to +\infty$ and the $\Gamma$-limit as $\e\to 0$ commute. This result will allow to limit our analysis to truncated functionals in the proofs of the results of the following sections, in particular in that of Theorem \ref{reprthm}, leading to significant simplifications.

\begin{definition}[the \emph{truncated functionals} $F_\e^T$]\label{truncated-functionals-def}
For any $A\in\mathcal{A}(\Omega)$ and  $T>0$ the functional $F_\e^T(\cdot, A):L^p(A;\mathbb{R}^m)\to[0,+\infty]$ is defined by
\begin{equation}\label{truncated-functionals}
F_\e^T(u) := \int_{B_T}\int_{A_\e(\xi)} f_\e\Big(x,\xi,\frac{u(x+\e\xi)-u(x)}{\e}\Big) dx\, d\xi\,.
\end{equation}
\end{definition}
Note that Proposition \ref{growthcondremk} clearly applies also to the truncated functionals $F^T(\cdot,A)$, since they comply with all the hypotheses of Subsection \ref{mainass}. In what follows we use the notation
$$F'^{ ,T}(u,A):=\Gamma\text{-}\liminf_{\e\to0} F^T_\e(u,A),\quad F''^{ ,T}(u,A):=\Gamma\text{-}\limsup_{\e\to0} F^T_\e(u,A).$$
\begin{lemma}\label{lim-truncated-lemma}
Let $F_\e(\cdot,A)$ and $F_\e^T(\cdot, A)$ be defined by \eqref{loc-functionals}  and \eqref{truncated-functionals}, respectively, and let assumptions {\rm(H0)--(H2)} be satisfied. Then for every $A\in\mathcal{A}^{\rm reg}(\Omega)$ and $u\in L^p(\Omega;\mathbb{R}^m)$
\begin{align*}
F'(u,A) = \lim_{T\to+\infty} F'^{ ,T}(u,A) \quad\hbox{ and }\quad
F''(u,A) = \lim_{T\to+\infty} F''^{ ,T}(u,A). 
\end{align*}
In particular, given $T_j\to+\infty$ such that $F_\e^{T_j}(\cdot,A)$ $\Gamma$-converge to $F^{T_j}(\cdot,A)$ for every $j\in\NN$,
$$\Gamma\text{-}\lim_{\e\to0}F_\e(u,A) = \lim_{j\to+\infty} F^{T_j}(u,A)$$
for every $u\in L^p(\Omega;\mathbb{R}^m)$.
\end{lemma}
\begin{proof}
Note first that, since $F_\e^T(u,A)\le F_\e(u,A)$ for every $u\in L^p(\Omega;\mathbb{R}^m)$ and $A\in\mathcal{A}(\Omega)$, one inequality in the statement is trivial. Thanks to Proposition \ref{growthcondremk}, it is sufficient to prove the opposite inequality for every $u\in W^{1,p}(A;\mathbb{R}^m)$. Hence, let $u_\e\to u$ in $L^p(\Omega;\mathbb{R}^m)$. Without loss of generality   we may assume that $F_\e(u_\e,A)$ is uniformly bounded.
We have that
\begin{align}\label{est1}
F_\e(u_\e,A) &= F_\e^T(u_\e,A) + \int_{B_T^c}\int_{A_\e(\xi)} f_\e\Big(x,\xi,\frac{u_\e(x+\e\xi)-u_\e(x)}{\e}\Big) dx\, d\xi ,\nonumber\\
\intertext{and from assumption (H1) we get}
F_\e(u_\e,A) &\le F_\e^T(u_\e,A) + \int_{B_T^c} \int_{A_\e(\xi)} \psi_\e(\xi)\Big(\Big|\frac{u_\e(x+\e\xi)-u_\e(x)}{\e}\Big|^p+1\Big) dx\, d\xi.
\end{align}
Thanks to Corollary \ref{boundlemma-lip}, for $\e$ small enough
\begin{multline}\label{est2}
\int_{B_T^c} \int_{A_\e(\xi)} \psi_\e(\xi)\Big(\Big|\frac{u_\e(x+\e\xi)-u_\e(x)}{\e}\Big|^p+1\Big) dx\, d\xi \\
\le \int_{B_T^c}\psi_\e(\xi) \big(C(|\xi|^p+1) \big(G_\e^{r_0}(u_\e,A)+\|u_\e\|_{L^p(A;\rr^m)}^p\big)+|A|\big) d\xi\, .
\end{multline}
Since by \eqref{growth-cond} $G_\e^{r_0}(u_\e,A)$ is bounded, by \eqref{est1}, \eqref{est2} and (H2) we have
$$F_\e(u_\e,A) \le F^T_\e(u_\e,A)+C \delta+o(1)$$
for every $T>r_\delta$, where $r_\delta$ is chosen as in (H2), and the thesis follows letting first $\e$ and then $\delta$ tend to $0$.
\end{proof}

\subsection{Fundamental estimates}

 In what follows, with a slight  abuse of notation, $F''(\cdot,\cdot)$ will denote the $\Gamma$-$\limsup$ of both the family of functionals
 $\{F_{\e}\}_\e$ and the sequence $\{F_{\e_j}\}_j$  for any $\e_j\to 0$.

A crucial step in order to apply Theorem \ref{representationthm} is provided by the following two propositions.

\begin{proposition}\label{funestthm}
Let $F_\e(\cdot,\cdot)$ be defined by  \eqref{loc-functionals} and assume that {\rm(H0)--(H2)} hold.
Let $A,B\in\mathcal{A}(\Omega)$ and let  $A',B'\in\mathcal{A}(\Omega)$ with  $A'\Subset A$ and $B'\Subset B$. Then
\begin{equation}\label{fundamentalest}
F''(u,A'\cup B')\le F''(u,A)+F''(u,B)
\end{equation}
for every $u\in L^p(\Omega;\mathbb{R}^m)$.
\end{proposition}
\begin{proof}
Without loss of generality we may suppose  that $F''(u,A)$ and $F''(u,B)$ are finite. Moreover, since $F''(u,\cdot)$ is an increasing set function, we may assume that $A', B'\in\mathcal{A}^{\rm reg}(\Omega)$. Let $u_\e,v_\e$ both converge to $u$ in $L^p(\Omega;\mathbb{R}^d)$ and be such that
$$
\lim_{\e\to 0} F_\e(u_\e,A)=F''(u,A),\quad \lim_{\e\to 0}F_\e(v_\e,B)= F''(u,B).
$$
Note that, by \eqref{growth-cond}, $G_\e^{r_0}(u_\e,A)$ and $G_\e^{r_0}(v_\e,B)$ are uniformly bounded.
Let $R:=\dist(A',\mathbb{R}^d\backslash A)$ and, fixed $N\in\NN$, set
$$A_i=\{x\in A\,|\text{ dist}(x,A')<iR/N\},\,1\le i\le N.$$
Let $\varphi^i$ be a cut-off function between $A_i$ and $A_{i+1}$, with $|\nabla\varphi^i|\le 2N/R$,
 and set $w_\e^i:=u_\e\varphi^i+v_\e(1-\varphi^i)$. Note that $w_\e^i\to u$ in $L^p(\Omega;\rr^m)$ for any $i\in\{1,\dots,N\}$.
By adding and subtracting $\varphi^i(x)u_\e(x+\e\xi)+(1-\varphi^i(x))v_\e(x+\e\xi)$ we have
\begin{equation}\label{cut-off-est1}
\begin{aligned}
w_\e^i(x+\e\xi)-w_\e^i(x) &= \varphi^i(x)(u_\e(x+\e\xi)-u_\e(x))+(1-\varphi^i(x))(v_\e(x+\e\xi)-v_\e(x)) \\
& \quad\quad +(\varphi^i(x+\e\xi)-\varphi^i(x))(u_\e(x+\e\xi)-v_\e(x+\e\xi)).
\end{aligned}
\end{equation}
Note that
\begin{equation}\label{cut-off-est2}
w_\e^i(x+\e\xi)-w_\e^i(x) = \begin{cases}
u_\e(x+\e\xi)-u_\e(x),&\text{if }x\in(A_i)_\e(\xi) \\
v_\e(x+\e\xi)-v_\e(x),&\text{if }x\in(\Omega\backslash \overline{A}_{i+1})_\e(\xi),
\end{cases}
\end{equation}
while,  having set
$$S_{\e,\xi}^i:=((A_i)_\e(\xi)\cup(\Omega\backslash \overline{A}_{i+1})_\e(\xi))^c\cap(A'\cup B')_\e(\xi),$$
by \eqref{cut-off-est1} and Jensen's inequality, using the notation $D_\xi^\e g(x)=(g(x+\e\xi)-g(x))/\e$ for the sake of brevity, we get
\begin{align}\nonumber
|D_\xi^\e w_\e^i(x)|^p &\le 2^{p-1}\varphi^i(x)|D_\xi^\e u_\e(x)|^p+2^{p-1}(1-\varphi^i(x))|D_\xi^\e v_\e(x)|^p \\ \nonumber
& \quad\quad +2^{p-1}|D_\xi^\e\varphi^i(x)|^p|u_\e(x+\e\xi)-v_\e(x+\e\xi)|^p \\ \label{cut-off-est3}
&\le 2^{p-1}\left(|D_\xi^\e u_\e(x)|^p+|D_\xi^\e v_\e(x)|^p+\left(\frac{2N}{R}\right)^p|\xi|^p|u_\e(x+\e\xi)-v_\e(x+\e\xi)|^p\right)
\end{align}
for every $x\in S_{\e,\xi}^i$.

Now, we consider the truncated functionals $F^T_\e$ introduced in Definition \ref{truncated-functionals-def}.
By \eqref{cut-off-est2}, \eqref{cut-off-est3} and assumption (H1) we have
\begin{align}
F^T_\e(w_\e^i,A'\cup B') &= F^T_\e(u_\e,A_i)+F^T_\e(v_\e,(\Omega\backslash \overline{A}_{i+1})\cap B')+\int_{B_T}\int_{S_{\e,\xi}^i}f_\e(x,\xi,D_\xi^\e w_\e^i(x))dx \, d\xi \nonumber \\
&\le F_\e(u_\e,A)+F_\e(v_\e,B) \nonumber \\
& \quad\quad +C \int_{B_T}\psi_\e(\xi)\int_{S_{\e,\xi}^i}(|D_\xi^\e u_\e(x)|^p+|D_\xi^\e v_\e(x)|^p+1)dx \, d\xi \label{funest1} \\
& \quad\quad +C N^p\int_{B_T}\psi_\e(\xi)\int_{S_{\e,\xi}^i} |\xi|^p |u_\e(x+\e\xi)-v_\e(x+\e\xi)|^p dx \, d\xi. \label{funest2}
\end{align}
By (H1), the integral in \eqref{funest2} can be estimated from above uniformly in $T$ as follows
\begin{align*}
\int_{B_T}\psi_\e(\xi)\int_{S_{\e,\xi}^i} |\xi|^p |u_\e(x+\e\xi)-v_\e(x+\e\xi)|^p dx \, d\xi &\le \int_{\mathbb{R}^d}\psi_\e(\xi)|\xi|^p\|u_\e-v_\e\|_{L^p(\Omega;\rr^m)}^pd\xi \\
&\le C_1 \|u_\e-v_\e\|_{L^p(\Omega;\rr^m)}^p;
\end{align*}
hence it tends to zero as $\e\to0$, since  $u_\e$ and $v_\e$  both converge to $u$ in $L^p(\Omega;\rr^m)$.

Note that, if $|\xi|<T$ 
$S_{\e,\xi}^i\subset(A_{i+1}\backslash\overline{A_i}+(-\e,\e)\xi)\cap B'$; hence,
\begin{equation}\label{boundaries}
\bigcup_{i=1}^{N-4}S_{\e,\xi}^i\subset (A_{N-2}\backslash \overline{A'})\cap B'.
\end{equation}
Moreover, the sets $\{S_{\e,\xi}^i\}_i$ intersect at most pairwise for $\e$ small enough. Thus, from \eqref{boundaries}
for $\e$ small enough we get
\begin{multline*}
\sum_{i=1}^{N-4}\int_{B_T}\psi_\e(\xi)\int_{S_{\e,\xi}^i}(|D_\xi^\e u_\e(x)|^p+|D_\xi^\e v_\e(x)|^p+1)dx \, d\xi \\
\le 2 \int_{B_T}\psi_\e(\xi)\int_{A_{N-2}\cap B'}(|D_\xi^\e u_\e(x)|^p+|D_\xi^\e v_\e(x)|^p+1)dx \, d\xi.
\end{multline*}
Hence, we apply Lemma \ref{boundlemma} on $u_\e$ and $v_\e$ with $E=A_{N-2}\cap B'$ for any $|\xi|<T$, so that $E_{\e,\xi}\subset A_{N-1}\cap(B'+B_{\e(r_0+T)})$.
Thus for $\e$ small enough $E_{\e,\xi}\subset A\cap B$ and we get
\begin{multline}\label{uppbound0.5}
\sum_{i=1}^{N-4}\int_{B_T}\psi_\e(\xi)\int_{S_{\e,\xi}^i}(|D_\xi^\e u_\e(x)|^p+|D_\xi^\e v_\e(x)|^p+1)dx \, d\xi \\
\le  (G_\e^{r_0}(u_\e,A)+G_\e^{r_0}(v_\e,B)) \int_{B_T}C(|\xi|^p+1) \psi_\e(\xi) d\xi \le C.
\end{multline}
We can choose an index $1\le k_\e\le N-4$ such that
\begin{equation}\label{uppbound1}
\int_{B_T}\psi_\e(\xi)\int_{S_{\e,\xi}^{k_\e}}(|D_\xi^\e u_\e(x)|^p+|D_\xi^\e v_\e(x)|^p+1)dx \, d\xi \le \frac{C}{N-4},
\end{equation}
uniformly in $T$,
and
$$\Gamma\text{-}\limsup_{\e\to 0}F^T_{\e}(u,A'\cup B')\le \limsup_{\e\to 0} F_\e(w_\e^{k_\e}, A'\cup B')\leq F''(u,A)+F''(u,B)+\frac{C}{N-4}.$$
Letting first $N\to\infty$ and then $T\to+\infty$, by Lemma \ref{lim-truncated-lemma} we get the conclusion.
\end{proof}


\begin{proposition}\label{innerreg}
Let  $F_\e(\cdot,\cdot)$ be defined by  \eqref{loc-functionals} and assume that {\rm(H0)--(H2)} hold.
Then $$\sup_{A'\Subset A}F''(u,A') = F''(u,A)$$
for every $A\in\mathcal{A}^{\rm reg}(\Omega)$ and $u\in L^p(\Omega;\mathbb{R}^d)\cap W^{1,p}(A;\mathbb{R}^m)$.
\end{proposition}
\begin{proof}
Since $F''(u,\cdot)$ is an increasing set-function, it suffices to prove that
 $$
 \sup_{A'\Subset A}F''(u,A') \ge F''(u,A).
 $$
To this end we argue as in the proof of Proposition \ref{funestthm}. For any $\delta>0$ let $A_\delta\Subset A$ be an open set such that
$$|A\setminus\overline{A_\delta}|+\|D u\|_{L^p(A\setminus\overline{A_\delta})}^p<\delta$$
and let $A'\in\mathcal{A}(\Omega)$ be such that $A_\delta\Subset A'\Subset A$. Let $u_\e,v_\e \in L^p(\Omega;\mathbb{R}^m)$ both converge to $u$ in $L^p(\Omega;\rr^m)$ and be such that
$$
\lim_{\e\to 0} F_\e(u_\e,A')=F''(u,A'),\quad \lim_{\e\to 0}F_\e(v_\e,A\setminus\overline{A_\delta})= F''(u,A\setminus\overline{A_\delta}).
$$ Thus, by Proposition \ref{growthcondremk} we get
\begin{equation}\label{innreg1}
F_\e(v_\e,A\setminus \overline{A_\delta})\le F''(u,A\setminus \overline{A_\delta})+o_\e(1)\le C\delta+o_\e(1).
\end{equation}
Set $R:=\dist(A_\delta,\Omega\setminus A')$, $A_i=\{x\in A\,|\, \dist(x,A_\delta)<iR/N\}$, and let $\varphi_i$ and $w_\e^i$ be defined as in the proof of Proposition \ref{funestthm}.
Set
$$S_{\e,\xi}^i:=((A_i)_\e(\xi)\cup(A\setminus \overline{A_{i+1}})_\e(\xi))^c\cap A_\e(\xi)\subset (A_{i+1}\setminus\overline{A_i}+(-\e,\e)\xi)\cap A_\e(\xi).$$
Consider first the finite range interaction energies $F_\e^T(\cdot,\cdot)$. Reasoning as in the proof of Proposition \ref{funestthm}, we can find $2\le k_\e\le N-4$ such that
\begin{align*}
F_\e^T(w_\e^{k_\e},A) &= F_\e^T(u_\e, A_{k_\e})+F_\e^T(v_\e,A\setminus \overline{A_{k_\e+1}})+\int_{B_T}\int_{S_{\e,\xi}^{k_\e}} f_\e(x,\xi,D_\xi^\e w_\e^k(x))dx \, d\xi \\
&\le F_\e(u_\e,A')+C\delta+o_\e(1) \\
& \quad\quad +\frac{C}{N-4}\int_{B_T}\psi_\e(\xi)(|\xi|^p+1)(G_\e(u_\e,A')+G_\e(v_\e,A'\setminus \overline{A_\delta}))+|A|)d\xi \\
&\le F_\e(u_\e,A')+o_\e(1)+\frac{C}{N-4}+C\delta.
\end{align*}
Letting first $\e\to 0$ and then $N\to\infty$ we get
$$F''^{ ,T}(u,A) \le F''(u,A')+C\delta\le \sup_{A'\Subset A} F''(u,A')+C\delta.$$
The result follows from the arbitrariness of $\delta$ and $T$ and Lemma \ref{lim-truncated-lemma}.
\end{proof}

\subsection{Proof of Theorem \ref{reprthm}}
\begin{proof}[Proof of Theorem {\rm\ref{reprthm}}]
We divide the proof in two steps, dealing first with the  case  in which (H0) holds and with the general case.

 \emph{Step $1$.} Assume that (H0) holds.
The compactness properties of $\Gamma$-convergence, Proposition \ref{innerreg} and \cite[Theorem 10.3]{bradef} yields the existence of a subsequence $(\e_{j_k})$ such that the $\Gamma$-limit
 $$
 \Gamma(L^p)\text{-}\lim_{k\to\infty}F_{\e_{j_k}}(u,A)=: F(u,A)
 $$
 exists for any $(u,A)\in L^p(\Omega;\rr^m)\times \mathcal{A}^{\rm reg}(\Omega)$. By Proposition \ref{growthcondremk}, $F(u,A)=+\infty$ if and only if $u\not\in W^{1,p}(A;\rr^m)$. We now introduce the inner-regular extension of $F(u,\cdot)$ on the whole family $\mathcal{A}(\Omega)$
 defined by
 $$\tilde F(u,A) :=\sup\{ F(u,A') \,|\, A'\in \mathcal{A}^{\rm reg}(\Omega),\, A'\Subset A\}.$$
Since, by Proposiiton \ref{innerreg},  $\tilde F(u,A)= F(u,A)$ for any $u\in W^{1,p}(A;\rr^m)$ and $A\in\mathcal{A}^{\rm reg}(\Omega)$, it remains to check that $\tilde F$ satisfies all the hypotheses of Theorem \ref{representationthm}.
$\tilde F(u,\cdot)$ is clearly increasing. By Remark \ref{locality}, hypothesis (i) trivially holds.
Proposition \ref{growthcondremk} yields (iii).
Since $F_\e(\cdot,A)$ depends only on incremental ratios, it is translation invariant and so does $\tilde F(\cdot, A)$; thus, (iv) is satisfied.
By the properties of $\Gamma$-convergence $F(\cdot,A)$ is lower semicontinuous with respect to the $L^p(\Omega;\rr^m)$ topology (see for instance \cite{bra2} Proposition 1.28). Thus, Proposition \ref{growthcondremk} yields the weak lower semicontinuity of $F(\cdot, A)$ with respect to the  $W^{1,p}(\Omega;\rr^m)$ topology. By its definition, $\tilde F(\cdot, A)$ inherits the same property, thus (v) holds.
As a consequence of Propositions \ref{funestthm} and \ref{innerreg}, $\tilde F(u,\cdot)$ is subadditive, superadditive on disjoint sets and inner regular. Hence, by the De Giorgi-Letta measure criterion (see \cite{bradef}), $\tilde F(u, \cdot) $ is the restriction on $\mathcal{A}(\Omega$) of a Borel measure, thus (ii) is satisfied and the thesis follows.

\emph{Step $2$.} Assume that only  (H0$'$) holds. For any $n\in \NN$ let $F_{\e,n}:L^p(\Omega;\rr^m)\times \mathcal{A}(\Omega)\to [0,+\infty)$ be defined by
$$
F_{\e,n}(u,A):=F_\e(u,A)+\frac1n G_\e^{r_0}(u,A).
$$
Since the family of functionals $F_{\e,n}$ satisfies (H0)--(H2) for every $n\in\NN$, by Step 1 and a diagonalization argument there exist a subsequence
$(\e_{j_k})$  and a non increasing sequence of functions $f_n:\Omega\times\rr^{m\times d}\to [0,+\infty)$, $n\in\NN$, quasiconvex in the second variable and satisfying \eqref{growthf0}, such that
 $$
 \Gamma(L^p)\text{-}\lim_{k\to\infty}F_{\e_{j_k},n}(u,A)=\begin{cases}\displaystyle\int_A f_n(x,Du)\, dx& \text{if}\ u\in W^{1,p}(A,\rr^m),\\
 +\infty & \text{otherwise}.
 \end{cases}
 $$
 for any $A\in \mathcal{A}^{\rm reg}(\Omega)$. We claim that \eqref{limres} holds with $f_0(x,M):=\inf_{n\in\NN}f_n(x,M)$. Indeed, since $F_\e\leq F_{\e,n}$ for every $\e>0$ and  $n\in\NN$, we have
 $$
 F''(u,A)\leq F(u,A):=\begin{cases}\displaystyle\int_A f_0(x,Du)\, dx& \text{if}\ u\in W^{1,p}(A,\rr^m),\\
 +\infty & \text{otherwise}.
 \end{cases}
$$
It remains to prove that
\begin{equation}\label{lowerest}
F(u,A)\leq F'(u,A).
\end{equation}
By Proposition \ref{growthcondremk}, it suffices to prove \eqref{lowerest} for $u\in W^{1,p}(A;\rr^m)$. Let then $u_k\to u$ in $L^p(\Omega;\rr^m)$ and be such that
$$
\liminf_{k\to +\infty} F_{\e_{j_k}}(u_k,A)=F'(u,A).
$$
Given $A'\Subset A$, by (H0$'$) we may assume that $G_{\e_{j_k}}^{r_0}(u_k,A')$ is uniformly bounded. Hence
$$
\int_{A'}f_0(x,Du)\, dx\leq  \int_{A'}f_n(x,Du)\, dx\leq \liminf_{k\to +\infty} F_{\e_{j_k},n}(u_k,A')\leq F'(u,A) -\frac Cn.
$$
Thus \eqref{lowerest} follows from the arbitrariness of $n\in\NN$ and $A'\Subset A$.
\end{proof}


\begin{remark}\label{G-lim-truncated-remark}
Theorem \ref{reprthm} clearly applies also to the truncated energies $F_\e^T$.  Suppose that  for any $(u,A)\in W^{1,p}(\Omega;\rr^m)\times \mathcal{A}^{\rm reg}(\Omega)$ there exists
\begin{equation}\label{G-lim-truncated}
\Gamma\text{-}\lim_{\e\to0}F_\e^T(u,A) =\int_A f_0^T(x,Du(x))dx.
\end{equation}
Then, by Lemma \ref{lim-truncated-lemma} and monotone convergence, we infer that
$$\Gamma\text{-}\lim_{\e\to0}F_\e (u,A)= \int_A f_0(x,Du(x)) dx$$
where for almost every $x_0\in\Omega$ and every $M\in\mathbb{R}^{m\times d}$
\begin{equation}\label{convdens}
f_0(x_0,M):=\lim_{T\to+\infty}f_0^T(x_0,M).
\end{equation}
\end{remark}

\subsection{Convergence of minimum problems}\label{co-mi-pro}

In this section we prove the convergence of minimum problems under Dirichlet boundary conditions.

\begin{definition}\label{bo-se}
For any $g\in W_{\rm loc}^{1,p}(\rr^d;\mathbb{R}^m)$, $A\in\mathcal{A}^{\rm reg}(\Omega)$ and $r>0$ we set
\begin{equation}\label{set-boudary-data}
\mathcal{D}^{r,g}(A) := \big\{u\in L^p(\Omega;\mathbb{R}^m) \,|\, u(x)=g(x) \text{ for almost every } x\in\Omega \,:\, \dist(x,\Omega\backslash A)<r \big\}
\end{equation}
and define the functionals $F_\e^{r,g}:L^p(\Omega;\mathbb{R}^d)\times\mathcal{A}^{\rm reg}(\Omega)\to[0,+\infty)$
\begin{equation}\label{DBCseq}
F_\e^{r,g}(u,A):=\begin{cases}
F_\e(u,A) &\text{if }u\in\mathcal{D}^{\e r,g}(A) \\
+\infty &\text{otherwise}.
\end{cases}
\end{equation}
When dealing with the affine function $g(x)=Mx$ for some $M\in\mathbb{R}^{m\times d}$ we will use the notation $\mathcal{D}^{r,M}$ and $F_\e^{r,M}$.
\end{definition}

\begin{proposition}\label{DBC-Gamma-conv}
Let $A\in\mathcal{A}^{\rm reg}(\Omega)$, let $F_\e(\cdot,A)$ be defined by  \eqref{loc-functionals} and assume that {\rm(H0)--(H2)} hold.
Let $\e_j\to0$ and let $f_0:\Omega\times\mathbb{R}^{m\times d}\to[0,+\infty)$ be such that
$$
\Gamma(L^p)\text{-}\lim_{j\to +\infty}F_{\e_j}(u,A)=\begin{cases}\displaystyle\int_A f_0(x, Du(x)\, dx & \text{if } u\in W^{1,p}(A;\rr^m)\cr
+\infty & \text{otherwise}
\end{cases}
=:F(u,A).
$$
Then, given $g\in W_{\rm loc}^{1,p}(\rr^d;\mathbb{R}^m)$ and $r>0$, the corresponding sequence of functionals $F_{\e_j}^{r,g}$ defined in \eqref{DBCseq} $\Gamma$-converges with respect to the $L^p(\Omega;\rr^m)$ topology to the functional
\begin{equation}\label{DBC}
F^g(u,A):=\begin{cases}
F(u,A) &\text{if }u-g\in W^{1,p}_0(A;\mathbb{R}^m) \\
+\infty &\text{otherwise.}
\end{cases}
\end{equation}
\end{proposition}
\begin{proof}
Since $F_\e^{r,g}(u, A)\ge F_\e(u,A)$, in order to prove the $\Gamma$-$\liminf$ inequality it suffices to show that if $u_j\to u$ in $L^{p}(A;\rr^m)$ and $\sup_j F_{\e_j}^{r,g}(u_j, A)$ is finite, then $u-g\in W_0^{1,p}(A;\mathbb{R}^m)$. Denote by  $\tilde u_j$ and $\tilde u$ the extension of $u_j$ and $u$ on the whole $\rr^d$ obtained by setting  $\tilde u_j=g$, $j\in\NN$, and $\tilde u=g$ on $\rr^d\backslash A$. Let $\tilde A$ be an open set such that $\tilde A\Supset A$ and note that, by (H0), for every $r'\leq \min\{r_0,r/2\}$ we have
$$
G_{\e_j}^{r'}(\tilde u_j, \tilde A)\leq G_{\e_j}^{r'}(u_j,A_\e^{r/2} )+G_{\e_j}^{r'}(g, \tilde A\setminus \overline{A_{\e_j}^r})\leq C(F_{\e_j}(u_j,A)+|A|)+G_{\e_j}^{r'}(g, \tilde A\setminus \overline{A_{\e_j}^r})\leq C,
$$
where $A_\e^r:=\{x\in A:\ \dist(x,\Omega\backslash A)> \e r\}$. Since $\tilde u_j\to \tilde u$ in $L^p(\tilde A;\rr^m)$, by Proposition \ref{G-liminf-lem} we get that $\tilde u\in W^{1,p}(\tilde A;\rr^m)$ and thus $u-g\in W_0^{1,p}(A;\mathbb{R}^m)$.



By a density argument it suffices to prove the $\Gamma$-$\limsup$ inequality for $u\in W^{1,p}(\Omega;\rr^m)$ such that $\text{spt}(u-g)\Subset A$. Given such a $u$, let $u_j$ converge to $u$ in $L^p(\Omega;\mathbb{R}^m)$ such that
$$\lim_{j\to\infty}F_{\e_j}(u_j,A)=F(u,A).$$
With an argument analogous to the one used in the proof of Propositions \ref{funestthm} and \ref{innerreg}, given $\delta>0$, we can find a suitable cut-off function $\varphi_j$ with
$\text{spt}( \varphi_j)\Subset A$ such that, having set
$v_j:=\varphi_ju_j+(1-\varphi_j)u$,
we have that  $v_j$ still converge to $u$ in $L^p(\Omega;\mathbb{R}^m)$ and
$$
F_{\e_j}(v_j,A) \le F_{\e_j}(u_j,A) + \delta.
$$
Since $v_j\in \mathcal{D}^{\e_j r,g}(A)$ for $j$ large enough, we get
$$
\limsup_{j\to\infty}F_{\e_j}^{r,g}(v_j,A)\le F(u,A)+\delta
$$
and the arbitrariness of $\delta$ leads to the desired inequality.
\end{proof}

As a consequence of  Propositions \ref{DBC-Gamma-conv}, \ref{poincare} and  Theorem \ref{kolcom}, we derive the following result of convergence of minimum problems with Dirichlet boundary data.
\begin{proposition}\label{DBCminpbs}
Under the assumptions of Proposition {\rm\ref{DBC-Gamma-conv}} there holds
$$
\lim_{j\to\infty}\inf\{F_{\e_j}(u,A)\,|\,u\in \mathcal{D}^{\e_j r,g}(A)\}=\min\{F(u,A)\,|\,u-g\in W^{1,p}_0(\Omega;\mathbb{R}^m)\}.
$$
Moreover, if $u_j\in\mathcal{D}^{\e_j r,g}(A)$ is a converging sequence such that
$$
\lim_{j\to \infty} F_{\e_j}(u_j,A)=\lim_{j\to \infty} \inf\{F_{\e_j}(u,A)\,|\,u\in \mathcal{D}^{\e_j r,g}(A)\} ,
$$
then its limit is a minimizer for $\min\{F(u,A)\,|\,u-g\in W^{1,p}_0(\Omega;\mathbb{R}^m)\}$.
\end{proposition}
\begin{proof}
Note that
$$
\inf\{F_{\e_j}(u,A)\,|\,u\in \mathcal{D}^{\e_j r,g}(A)\}\le F_{\e_j}(g, A)\leq C.
$$
Hence, by the properties of $\Gamma$-convergence (see for instance \cite{bra2} Theorem 1.21), we only need to prove the equi-coerciveness of the family $\{F_{\e_j}^{r,g}(\cdot ,A)\}_{\e_j}$ in the strong $L^p(A;\rr^m)$ topology. Let then $\{u_j\}_j\subset L^p(\Omega;\mathbb{R}^d)$ be such that $F_{\e_j}^{r,g}(u_j,A)\le C$.
Reasoning as in the proof of Proposition \ref{DBC-Gamma-conv},
from assumption (H0) we deduce that
$G_{\e_j}^{r'}( u_j,  A) \leq C$
for every $r'\leq \min\{r_0,r/2\}$. Proposition \ref{poincare} yields that
$$
\int_{\Omega}|u_j(x)-g(x)|^p dx \le C G_{\e_j}^{r'}(u_j-g, A) \le C \Big(G_{\e_j}^{r'}(u_j,A)+\|Dg\|_{L^p(\Omega)}^p\Big)\leq C.
$$
Hence, we may apply Theorem \ref{kolcom} and deduce that $\{u_j\}_j$ is precompact in the strong $L^p(\Omega;\rr^m)$ topology.
\end{proof}

\section{Homogenization}\label{homogenization}
In this section we study the homogenization of functionals as in \eqref{functionals} under a periodicity assumption in the first variable of the densities $f_\e$.
Specifically,  let $f:\mathbb{R}^d\times\mathbb{R}^d\times\mathbb{R}^m\to[0,+\infty)$ be a Borel function such that $f(\cdot,\xi,z)$ is $[0,1]^d$-periodic in the first variable for every $\xi\in\mathbb{R}^d$ and $z\in\mathbb{R}^m$. Throughout this section, we will assume that
\begin{equation}\label{homdef}
f_\e(x,\xi,z)=f\left(\frac{x}{\e},\xi,z\right)
\end{equation}
 in \eqref{functionals}. In this setting, assumptions (H0)--(H2) are straightforward consequence of the following growth conditions on $f$:
\begin{equation}\label{homgrocon1}
 c_0 (|z|^p-\rho(\xi))\le f(y,\xi,z) \quad \text{ if }|\xi|\le r_0
\end{equation}
\begin{align}
\label{homgrocon2}
f(y,\xi,z) &\le \psi(\xi)(|z|^p+1),
\end{align}
where $\rho:B_{r_0}\to[0,+\infty)$ and  $\psi:\mathbb{R}^d\to[0,+\infty)$ are such that
\begin{equation}\label{growthpsi}
\int_{B_{r_0}}\rho(\xi)d\xi<+\infty,\quad \int_{\mathbb{R}^d}\psi(\xi)(|\xi|^p+1)d\xi<+\infty.
\end{equation}

In the sequel we will use the notation $Q_R(x_0)=x_0+ (0,R)^d$ and the shorthand $Q_R$ if $x_0=0$
The main result of this section is stated in the following theorem.
\begin{theorem}\label{hom-thm}
Let $F_\e$ be defined by \eqref{functionals}, with $f_\e$ given by \eqref{homdef}, and let \eqref{homgrocon1}, \eqref{homgrocon2} and \eqref{growthpsi} be satisfied.
Then for every $M\in\mathbb{R}^{m\times d}$ the limit
\begin{equation}\label{homform}
f_{\rm hom}(M):=\lim_{R\to\infty}\frac{1}{R^d}\inf\Big\{\int_{Q_R}\int_{Q_R}f(x,y-x,v(y)-v(x))dx \, dy\,\Big|\, v\in\mathcal{D}^{1,M}(Q_R)\Big\},
\end{equation}
where $\mathcal{D}^{1,M}(Q_R)$ is defined by \eqref{set-boudary-data}, exists and defines a quasiconvex function $f_{\rm hom}:\mathbb{R}^{m\times d}\to[0,+\infty)$ satisfying
\begin{equation}\label{growthfhom}
c(|M|^p-1)\le f_{\rm hom}(M) \le C( |M|^p+1).
\end{equation}
Moreover,
\begin{equation*}
\Gamma(L^p)\text{-}\lim_{\e\to0}F_\e(u)=\begin{cases}\displaystyle\int_\Omega f_{\rm hom}(Du(x))dx& \text{if } u\in W^{1,p}(\Omega;\mathbb{R}^m)\\
+\infty & \text{otherwise}.
\end{cases}
\end{equation*}
\end{theorem}
The proof of Theorem \ref{hom-thm} relies on the results stated in the following two propositions. The first one provides the independence of the limit energy densities on the space variable, the second one the existence of the limit in \eqref{homform} in the case of truncated energies.
\begin{proposition}\label{transinv}
Under the assumptions of Theorem \ref{hom-thm}, let $\e_j\to0$ and let $f_0:\Omega\times\mathbb{R}^{m\times d}\to[0,+\infty)$ be a Carath\'eodory function such that for every $A\in \mathcal{A}^{\rm reg}(\Omega)$ and $u\in W^{1,p}(A;\rr^m)$ there holds
$$
\Gamma(L^p)\text{-}\lim_{j\to +\infty}F_{\e_{j}}(u,A)= \int_A f_0(x,Du(x))dx.
$$
Then $f_0$ is independent on the first variable.
\end{proposition}
\begin{proof}
It is sufficient to prove that
\begin{equation}\label{traslinv}
F(Mx,B_r(y))=F(Mx,B_r(y'))
\end{equation}
for every $M\in\mathbb{R}^{m\times d}$, $y,y'\in\Omega$ and $r>0$ such that $B_r(y),B_r(y')\subset\Omega$. We will prove that $F(Mx,B_{r'}(y))\le F(Mx,B_r(y'))$ for all $r'<r$.
By the inner regularity of $F(Mx,\cdot)$ provided by Proposition \ref{innerreg} we get \eqref{traslinv} by switching the roles of $y$ and $y'$.

Let $u_j\to Mx$ in $L^p(B_r(y');\mathbb{R}^m)$ be such that
$$
\lim_{j\to +\infty}F _{\e_j}(u_j,B_r(y'))=F(Mx,B_r(y'),
$$
and set
$\displaystyle v_j(x):=u_j\Big(x-\e_j\Big\lfloor\frac{y-y'}{\e_j}\Big\rfloor\Big)+\e_j M\Big\lfloor\frac{y-y'}{\e_j}\Big\rfloor$.
Since $B_{r'}(y)-\e_j\lfloor(y-y')/\e_j\rfloor\subset B_r(y')$ when $\e_j$ is small enough,  $v_j\to Mx$ in $L^p(B_{r'}(y);\rr^m)$.
Moreover, by the periodicity assumption on $f$, we have that
\begin{align*}
F_{\e_j}(v_j,B_{r'}(y)) &= \int_{\mathbb{R}^d}\int_{(B_{r'}(y))_{\e_j}(\xi)} f\Big(\frac{x}{\e_j},\xi,\frac{v_j(x+\e_j\xi)-v_j(x)}{\e_j}\Big)dx \, d\xi \\
&= \int_{\mathbb{R}^d}\int_{(B_{r'}(y))_{\e_j}(\xi)} f \Big(\frac{x}{\e_j}-\Big\lfloor\frac{y-y'}{\e_j}\Big\rfloor,\xi,\frac{v_j(x+\e_j\xi)-v_j(x)}{\e_j}\Big)dx \, d\xi.
\end{align*}
and, through the change of variable $x=x'+\e_j\lfloor(y-y')/\e_j\rfloor$, we get
$$F_{\e_j}(v_j,B_{r'}(y)) \le \int_{\mathbb{R}^d}\int_{(B_r(y'))_{\e_j}(\xi)} f \Big(\frac{x'}{\e_j},\xi,\frac{u_j(x'+\e_j\xi)-u_j(x')}{\e_j}\Big)dx \, d\xi=F_{\e_j}(u_j,B_r(y')).$$
Finally, letting $j\to +\infty$, we obtain
$$F(Mx,B_{r'}(y))\le\liminf_{j\to +\infty}F_{\e_j}(v_j,B_{r'}(y))\le\lim_{j\to +\infty}F_{\e_j}(u_j,B_r(y'))=F(Mx,B_r(y'))$$
and the claim.
\end{proof}

\begin{proposition}\label{asy-form-lem}
Let $f:\mathbb{R}^d\times\mathbb{R}^d\times\mathbb{R}^m\to[0,+\infty)$ be a Borel function $[0,1]^d$-periodic in the first variable such that assumptions \eqref{homgrocon2} and \eqref{growthpsi} hold. Let $T>0$ and set
\begin{equation}\label{ftrunc}
f^T(y,\xi,z)=\begin{cases}
f^T(y,\xi,z) & \text{if } |\xi|\le T,\cr
0 & \text{otherwise.}
\end{cases}
\end{equation}
Then, for every $r\geq T$ the limit
\begin{equation}\label{homformT}
f_{\rm hom}^T(M):=\lim_{R\to\infty}\frac{1}{R^d}\inf\Big\{\int_{Q_R}\int_{Q_R}f^T(x,y-x,v(y)-v(x))dx \, dy\,\Big|\, v\in\mathcal{D}^{r,M}(Q_R)\Big\},
\end{equation}
where we use the notation in Definition {\rm\ref{bo-se}}, exists and it is finite for every $M\in\mathbb{R}^{m\times d}$.

\end{proposition}
\begin{proof}
For every $R>0$ we set
$$
F_1^T(u,Q_R):= \int_{Q_R}\int_{Q_R} f^T(x,y-x,v(y)-v(x)) dx \, dy,
$$
$$H_R(M):=\frac{1}{R^d}\inf\Big\{ F_1^T(u,Q_R) \,\Big|\,v\in\mathcal{D}^{r,M}(Q_R) \Big\}$$
and let $u_R\in\mathcal{D}^{r,M}(Q_R)$ be such that
$$\frac{1}{R^d}F_1^T(u_R,Q_R)\le H_R(M)+\frac{1}{R}.$$
Set $R':=\lceil R\rceil$ and, for any $S>R'$, define
$$u_S(x):=\begin{cases}
\displaystyle u_R\Big(x-R'\Big\lfloor \frac{x}{R'}\Big\rfloor\Big)+R'M\Big\lfloor \frac{x}{R'}\Big\rfloor&\text{if }x\in Q_{S_R} \\
Mx &\text{otherwise},
\end{cases}$$
where $S_R:=\lfloor S/R'\rfloor R'$. Set
$$\mathcal{L}_{R,S}:=R'\Big\{0,1,\dots, \Big\lfloor \frac{S}{R'}\Big\rfloor -1\Big\}^d,\quad \mathcal{S}_{R,S}:=\bigcup_{h\in \mathcal{L}_{R,S}}\partial(h+Q_{R'})
$$
and note that $f(x,y-x,u_S(y)-u_S(x))\not=0$ only if $x,y\in h+Q_{R'}$, for some $h\in \mathcal{L}_{R,S}$, or $x,y\in  \mathcal{S}_{R,S}^T$, where
$$
 \mathcal{S}_{R,S}^T:=\{x\in Q_S:\,\dist(x, \mathcal{S}_{R,S})\le T\}\cup (Q_S\setminus Q_{S_R}).
$$
In the latter case, since $r>T$, $u(x)=Mx$ and $u(y)=My$. Hence,
from the definition of $u_S$  we get
\begin{equation}\label{homest1}
\begin{aligned}
F_1^T(u_S,Q_S)& \le  \int_{\mathcal{S}_{R,S}^T}\int_{\mathcal{S}_{R,S}^T}f^T(x,y-x, My-Mx)dx \, dy\\
&+ \sum_{k\in \mathcal{L}_{R,S} }\int_{k+Q_{R'}}\int_{k+Q_{R'}}f^T(x,y-x,u_R(y-k)-u_R(x-k))dx \, dy.\\
\end{aligned}
\end{equation}
By the periodicity of $f(\cdot,\xi,z)$, assumptions \eqref{homgrocon2} and \eqref{growthpsi}, and since $u_R(x)=Mx$ on $Q_{R'}\setminus Q_R$,
\begin{equation}\label{homest2}
\begin{aligned}
\sum_{k\in \mathcal{L}_{R,S} }&\int_{k+Q_{R'}}\int_{k+Q_{R'}}  f^T(x,y-x,u_R(y-k)-u_R(x-k))dx \, dy\\
& \le \Big\lfloor \frac{S}{R'}\Big\rfloor^d\Big(\int_{Q_R}\int_{Q_R} f^T(x,y-x,u_R(y)-u_R(x)) dx \, dy +C(|M|^p +1)(R')^{d-1}\Big).
\end{aligned}
\end{equation}
Again by \eqref{homgrocon2} and \eqref{growthpsi}, we get
\begin{equation}\label{homest3}
\int_{\mathcal{S}_{R,S}^T}\int_{\mathcal{S}_{R,S}^T}f^T(x,y-x, My-Mx)dx \, dy\leq C(|M|^p +1) T S^{d-1}\Big(\frac{S}{R'}+1\Big).
\end{equation}
Gathering \eqref{homest1}--\eqref{homest3}, from the definition of $u_R$ we obtain
\begin{equation}\label{homest4}
\begin{aligned}
F_1^T(u_S,Q_S) &\le \Big\lfloor\frac{S}{R'}\Big\rfloor^d \big(R^d H_R(M)+R^{d-1}+C(|M|^p +1)(R')^{d-1}\big)\\
&+  C(|M|^p +1) T S^{d-1}\Big(\frac{S}{R'}+1\Big).
\end{aligned}
\end{equation}
Finally, by using $u_S$ as a test function in the definition of $H_S(M)$, \eqref{homest4} gives
$$H_S(M) \le \frac{R^d}{S^d}\Big\lfloor\frac{S}{R'}\Big\rfloor^d H_R(M) + \frac{(R')^{d-1}}{S^d}\Big\lfloor\frac{S}{R'}\Big\rfloor^dC(|M|^p +1) + C(|M|^p +1) \frac{T}{R'}\Big(1+\frac{R'}{S}\Big) .$$
By taking the limit first as $S\to+\infty$ and then as $R\to+\infty$ we get
$$\limsup_{S\to+ \infty}H_S(M)\le\liminf_{R\to+\infty}H_R(M)$$
that yields the existence of the limit.
Finally, $H_R(M)\le F_1^T(Mx,Q_R)/R^d\le C_1(|M|^p+1)$, hence the limit is finite.
\end{proof}

%
\begin{proof}[Proof of Theorem \ref{hom-thm}] We split the proof in two steps, dealing first with the case of truncated energies and then with the general case.

 \emph{Step 1.} For $T>0$ let $F_\e^T(\cdot,\cdot)$ be as in Definition \ref{truncated-functionals-def}. Given $\e_j\to 0$, by Theorem \ref{reprthm} and Proposition \ref{transinv} there exist a subsequence (not relabelled) and $f_0:\rr^{m\times d}\to [0,+\infty)$ such that
$$
\Gamma(L^p)\text{-}\lim_{j\to +\infty} F^T_{\e_j}(u, A)=\int_A f_0 (D u(x))\, dx =:F^T(u,A)
$$
for every $A\in\mathcal{A}^{\rm reg}(\Omega)$ and $u\in W^{1,p}(A;\rr^m)$. We now prove that $f_0= f^T_{\rm hom}$, where $f^T_{\rm hom}$ is defined in \eqref{homformT}.
Since $f_0$ is quasiconvex and satisfies \eqref{growthf0}, given $x_0\in\Omega$ and $r>0$ such that $Q_r(x_0)\subset\Omega$, we have
\begin{align*}
f_0(M) &= \frac{1}{r^d}\min\Big\{ \int_{Q_r(x_0)} f_0(Du(x)) dx \,\Big|\, u-Mx\in W_0^{1,p}(Q_r(x_0);\mathbb{R}^m)\Big\} \\
\intertext{for every $M\in\mathbb{R}^{m\times d}$ and then, by Proposition \ref{DBCminpbs}, we get}
f_0(M) &= \lim_{j\to\infty} \frac{1}{r^d}\inf\big\{F^T_{\e_j}(u,Q_r(x_0))\,|\, u\in\mathcal{D}^{s\e_j,M}(Q_r(x_0))\big\}
\end{align*}
 for any $s>0$. Without loss of generality we may assume $x_0=0$ and use the notation $Q_r=Q_r(0)$.
Setting $v(x)=u(\e_j x)/\e_j$ and using the changes of variable $x'=x/\e_j$ and $y=x'+\xi$, we rescale $F^T_{\e_j}$ as
\begin{align*}
F^T_{\e_j}(u,Q_r) &= \int_{B_T}\int_{(Q_r)_{\e_j}(\xi)}f \Big(\frac{x}{\e_j},\xi,\frac{u(x+\e_j\xi)-u(x)}{\e_j}\Big) dx \, d\xi \\
&= \int_{B_T}\int_{(Q_r)_{\e_j}(\xi)}f \Big(\frac{x}{\e_j},\xi,v\Big(\frac{x}{\e_j}+\xi\Big)-v\Big(\frac{x}{\e_j}\Big)\Big) dx \, d\xi \\
& =\e_j^d \int_{Q_{R_j}}\int_{Q_{R_j}}f^T(x',y-x',v(y)-v(x')) dx' dy,
\end{align*}
where $R_j:=r/\e_j$ and $f^T$ is defined in \eqref{ftrunc}.
Thus,
$$f_0(M) = \lim_{j\to\infty} \frac{1}{R_j^d} \inf\Big\{ \int_{Q_{R_j}}\int_{Q_{R_j}} f^T(x,y-x,v(y)-v(x)) dx \, dy \,\Big|\,v\in\mathcal{D}^{s,M}(Q_{R_j}) \Big\}.$$
By the arbitrariness of $s>0$ and Proposition \ref{asy-form-lem} we eventually get
\begin{equation*}
f_0(M) = f_{\rm hom}^T (M)=\lim_{R\to\infty} \frac{1}{R^d} \inf\Big\{ \int_{Q_{R}}\int_{Q_{R}} f^T(x,y-x,v(y)-v(x)) dx \, dy \,\Big|\,v\in\mathcal{D}^{1,M}(Q_{R_j}) \Big\},
\end{equation*}
which in particular proves the claim of Theorem \ref{hom-thm} when $f\equiv f^T$.

\emph{Step 2.} By the previous step and Remark \ref{G-lim-truncated-remark} we infer that
$$
\Gamma(L^p)\text{-}\lim_{\e\to0} F_\e(u,A)=\lim_{T\to+\infty}\Gamma(L^p)\text{-}\lim_{\e\to 0} F_\e^T(u,A)=\int_A f_{\rm hom}^\infty(Du(x))dx
$$
for every $A\in\mathcal{A}^{\rm reg}(\Omega)$ and $u\in W^{1,p}(A;\mathbb{R}^m)$, where for every $M\in\mathbb{R}^{m\times d}$
$$
f_{\rm hom}^\infty(M):=\lim_{T\to+\infty} f^T_{\rm hom}(M).
$$
Hence, the result will follow if we prove that $f^\infty_{\rm hom}(M)=f_{\rm hom}(M)$. Set
$$
f'_{\rm hom}(M)=\liminf_{R\to\infty}\frac{1}{R^d}\inf\Big\{\int_{Q_R}\int_{Q_R}f(x,y-x,v(y)-v(x))dx \, dy\,\Big|\, v\in\mathcal{D}^{1,M}(Q_R)\Big\},
$$
$$
f''_{\rm hom}(M)=\limsup_{R\to\infty}\frac{1}{R^d}\inf\Big\{\int_{Q_R}\int_{Q_R}f(x,y-x,v(y)-v(x))dx \, dy\,\Big|\, v\in\mathcal{D}^{1,M}(Q_R)\Big\}.
$$
Since $f_{\rm hom}^T(M)\le f'_{\rm hom}(M)$ for every $T>0$, it suffices to prove that $f''_{\rm hom}(M)\le f_{\rm hom}^\infty(M)$.
Now, define
$$H^T_R(M)=\frac{1}{R^d} \inf \Big\{ \int_{Q_R}\int_{Q_R} f^T(x,y-x,v(y)-v(x))dx \, dy \,\Big|\,v\in\mathcal{D}^{1,M}(Q_R) \Big\},$$
$$H_R(M)=\frac{1}{R^d} \inf \Big\{ \int_{Q_R}\int_{Q_R} f(x,y-x,v(y)-v(x))dx \, dy \,\Big|\,v\in\mathcal{D}^{1,M}(Q_R) \Big\},$$
and let $u_R\in\mathcal{D}^{1,M}(Q_R)$ be such that
$$\frac{1}{R^d} \int_{Q_R}\int_{Q_R} f^T(x,y-x,u_R(y)-u_R(x))dx \, dy \le H_R^T(M)+\frac{1}{T}.$$
Note that, by  \eqref{homgrocon1}--\eqref{growthpsi}, we get that
$$\frac{1}{R^d} G^{r_0}_1(u_R,Q_R) \le C\Big( H_R^T(M)+\frac1T \Big)\le C\Big(F_1^T(Mx, Q_R)+\frac1T\Big)\leq C(|M|^p+1),$$
where $C$ is a constant independent of $T$ and $R$, so that, by Lemma \ref{boundlemma}, we get that
$$\frac{1}{R^d} \int_{(Q_R)_1(\xi)} |u_R(x+\xi)-u_R(x)|^p dx \le C(|\xi|^p+1)(|M|^p+1).$$
By taking $u_R$ as a test function for the minimum problem defining $H_R(M)$, we then have
\begin{align*}
H_R(M) &\le H_R^T(M)+\frac{1}{T}+\int_{B_T^c}\int_{(Q_R)_1(\xi)} f(x,\xi,u_R(x+\xi)-u_R(x)) dx\, d\xi \\
&\le H_R^T(M)+\frac{1}{T}+\int_{B_T^c}\psi(\xi) \frac{1}{R^d} \int_{(Q_R)_1(\xi)} (|u_R(x+\xi)-u_R(x)|^p +1)dx\, d\xi \\
&\le H_R^T(M)+\frac{1}{T}+C(|M|^p+1)\int_{B_T^c} \psi(\xi)(|\xi|^p+1)\, dx.
\end{align*}
By assumption \eqref{homgrocon2} taking the limit first as $R\to+\infty$ and then as $T\to+\infty$ we get the conclusion.
\end{proof}
As a straightforward consequence of Theorem \ref{hom-thm}, Propositons \ref{DBC-Gamma-conv} and \ref{DBCminpbs}, we deduce  the following results about
$\Gamma$-convergence and convergence of minimum problems for periodic functionals subject to Dirichlet boundary conditions.

\begin{proposition}
Under the assumptions of Theorem \ref{hom-thm}, given any $g\in W_{\rm loc}^{1,p}(\rr^d;\mathbb{R}^m)$ and $r>0$, let $\{F_\e^{r,g}(\cdot,\Omega)\}$ be the family of functionals defined in \eqref{DBCseq}.
Then
\begin{equation}\label{DBC-hom}
\Gamma\text{-}\lim_{\e\to 0} F_\e^{r,g}(u,\Omega)=\begin{cases} \displaystyle
\int_{\Omega} f_{\rm hom}(Du(x))dx &\text{if }u-g\in W^{1,p}_0(\Omega;\mathbb{R}^m) \\
+\infty &\text{otherwise.}
\end{cases}
\end{equation}
\end{proposition}

\begin{proposition}
Under the assumptions of Theorem \ref{hom-thm}, for any $g\in W_{\rm loc}^{1,p}(\rr^d;\mathbb{R}^m)$ and $r>0$  there holds
\begin{equation}\label{DBC-hom-minpbs}
\lim_{\e\to0}\inf\{F_\e(u)\,|\, u\in\mathcal{D}^{r\e,g}(\Omega)\} = \min\Big\{\int_{\Omega} f_{\rm hom}(Du(x))dx \,|\, u-g\in W^{1,p}_0(\Omega;\mathbb{R}^m)\Big\}.
\end{equation}
Moreover, if $\e_j\to 0$ and $u_j\in \mathcal{D}^{r\e_j,g}(\Omega)$ is a converging sequence such that
$$
\lim_{j\to \infty} F_{\e_j}(u_j)=\lim_{j\to \infty} \inf\{F_{\e_j}(u)\,|\,u\in \mathcal{D}^{r\e_j,g}(\Omega)\},
$$
then its limit is a minimizer for $\displaystyle\min\Big\{\int_{\Omega} f_{\rm hom}(Du(x))dx \,|\, u-g\in W^{1,p}_0(\Omega;\mathbb{R}^m)\Big\}$.
\end{proposition}

\subsection{The convex case}

In this subsection we show that, analogously to the homogenization of integral functionals, in the convex case the asymptotic formula \eqref{homform} reduces to a cell formula.

\begin{theorem}\label{cell-form-thm}
Under the hypotheses of Theorem \ref{hom-thm}, assume in addition that $f(y,\xi,\cdot)$ is convex for every $y,\xi\in\mathbb{R}^d$. Then the function $f_{\rm hom}$ defined by \eqref{homform} satisfies for every $M\in\mathbb{R}^{m\times d}$
\begin{equation}\label{cellform}
f_{\rm hom}(M) = \inf\Big\{ \int_{\mathbb{R}^d}\int_{Q_1} f(x,y-x,v(y)-v(x)) dx \, dy \,\Big|\, v\in\mathcal{D}^{\#,M}(Q_1) \Big\},
\end{equation}
where
$$\mathcal{D}^{\#,M}(Q_1)=\{u\in L_{\rm loc}^p(\mathbb{R}^d;\mathbb{R}^m) \,|\, u-Mx \text{ is } Q_1\text{-periodic}\}.$$
\end{theorem}
\begin{proof}
For brevity of notation, we denote by $f^\#(M)$ the right-hand side of  \eqref{cellform}.
We first prove that $f_{\rm hom}(M)\le f^\#(M)$. Given $\delta>0$, let $v\in \mathcal{D}^{\#,M}(Q_1)$ be such that
$$
\int_{\mathbb{R}^d}\int_{Q_1} f(x,y-x,v(y)-v(x)) dx \, dy\leq f^\#(M)+\delta,
$$
and set
$u_\e(x):=\e v\big(\frac x\e\big)$.
Then $u_\e\to Mx$ in $L^p(\Omega;\rr^m)$ and, by Theorem \ref{hom-thm} and the periodicity of $f$, we have
$$
|\Omega| f_{\rm hom}(M)\leq \limsup_{\e\to 0} F_\e(u_\e)\leq |\Omega|( f^{\#}(M)+\delta).
$$
The conclusion follows by the arbitrariness of $\delta>0$.

It remains to prove that
\begin{equation}\label{cellineq}
f_{\rm hom}(M)\ge f^\#(M).
\end{equation}
Note that, reasoning as in Step 2 of the proof of Theorem \ref{hom-thm}, we get that
$$
\lim_{T\to +\infty}  f^{\#,T}(M)=f^{\#}(M),
$$
where $f^{\#,T}(M)$ is defined by the right-hand side of \eqref{cellform} with $f^T$ in place of $f$. Hence it suffices to prove \eqref{cellineq} in the case $f(y,\xi,z)=0$ for every $|\xi|>T, y\in\mathbb{R}^d, z\in\mathbb{R}^m$.
Let $R\in\NN$ and, for any function $v\in\mathcal{D}^{T,M}(Q_R)$, let $u\in\mathcal{D}^{\#,M}(Q_1)$ be the function defined by
$$u(x):=\frac{1}{R^d}\sum_{i\in [0,R)^d\cap\mathbb{Z}^d}\tilde v(x+i),$$
where $\tilde v$ denotes the periodic extension of $v$ outside $Q_R$.
From the convexity of $f(x,\xi,\cdot)$ we get
\begin{equation}\label{cell-form-est3}
\begin{aligned}
 f^{\#}(M) & \le \int_{B_T}\int_{Q_1} f(x,\xi,u(x+\xi)-u(x)) dx \, d\xi \\
&\leq \frac{1}{R^d} \sum_{i\in Q_R\cap\mathbb{Z}^d} \int_{B_T}\int_{i+Q_1} f(x,\xi,v(x+\xi)-v(x)) dx \, d\xi \\
&=\frac{1}{R^d} \int_{B_T}\int_{Q_R} f(x,\xi,v(x+\xi)-v(x)) dx \, d\xi.
\end{aligned}
\end{equation}
Since for every $v\in\mathcal{D}^{T,M}(Q_R)$ there holds
\begin{align*}
\int_{B_T}\int_{Q_R\backslash (Q_R)_1(\xi)} &f(x,\xi,v(x+\xi)-v(x)) dx \, d\xi\\
& \le \int_{B_T}\int_{Q_R\backslash (Q_R)_1(\xi)} f(x,\xi,M\xi) dx \, d\xi \le C(|M|^p +1)TR^{d-1},
\end{align*}
by taking the infimum in  \eqref{cell-form-est3} we get
$$ f^{\#}(M) \le \inf \Big\{ \frac{1}{R^d} \int_{B_T}\int_{(Q_R)_1(\xi)} f(x,\xi,v(x+\xi)-v(x)) dx \, d\xi \,\Big|\, v\in\mathcal{D}^{T,M}(Q_R) \Big\} + \frac{CT}{R}.$$
Then, passing to the limit as $R\to+\infty$ we obtain the desired inequality.
\end{proof}
\begin{remark}\label{homcase}
If in Theorem \ref{cell-form-thm} $f$ does not depend on the first variable, Jensen's inequality yields that
$$
f_{\rm hom}(M)=\int_{\rr^d} f(\xi,M\xi)\, d\xi.
$$
\end{remark}
\begin{example}[quadratic forms]\label{quadratic-form-ex}
A well-known property of $\Gamma$-convergence is the fact that the $\Gamma$-limit of non-negative quadratic forms is still a non-negative quadratic form (see \cite{dal} Theorem 11.10).
Hence, under the hypotheses of Theorem \ref{hom-thm}, if $f(x,\xi,z)$ is a non-negative quadratic form of the type
 $$
f(y,\xi,z )=\langle A(y,\xi)z,z \rangle
$$
 where $A:\rr^d\times\mathbb{R}^d\to\mathbb{R}^{m\times m}$ is $[0,1]^d$-periodic then
\begin{align*}
f_{\rm hom}(M)&=\langle A_{\rm hom} M, M\rangle\\
&=\inf \Big\{ \int_{\mathbb{R}^d} \int_{Q_1} \langle A(x,\xi) (v(x+\xi)-v(x)),v(x+\xi)-v(x)\rangle dx\, d\xi \, \Big| \, v\in\mathcal{D}^{\#,M}(Q_1) \Big\},
\end{align*}
with $A_{\rm hom}\in T_2(\rr^{m\times d})$. By Remark \ref{homcase}, if $A(y,\xi)=A(\xi)$, then
$$
f_{\rm hom}(M)=\int_{\rr^d}\langle A(y,\xi)M\xi,M\xi \rangle\, d\xi.
$$
In particular, if $A(\xi)= a(\xi) I$ we recover the result in Theorem \ref{G-conv-conv} with $a_\e(\xi)=a(\xi)$.
\end{example}

\section{Perturbed convolution-type functionals}\label{pe-co-fu}

In view of an application to a point clouds model that we will discuss in Subsection \ref{pointclouds}, we consider here a slight generalization of the class of functionals defined in \eqref{functionals}, obtained by replacing  the Lebesgue measure with a measure $\mu=\rho(x) \mathcal{L}^d$, with $\rho\in C^0(\Omega)$ and satisfying
\begin{equation}\label{boundrho}
0<c\leq \rho(x)\leq C\quad \text{for every } x\in\Omega.
\end{equation}
More precisely, given such a $\rho$, we set
\begin{equation}\label{funct-gen-rho}
F_\e[\rho](u):=\frac{1}{\e^d}\int_{\Omega} \int_{\Omega} f_\e\Big(x,{y-x\over\e},\frac{u(y)-u(x)}{\e}\Big) \rho(y) \rho(x)dx\, dy.
\end{equation}

In the periodic case, that is when $f_\e$ satisfies \eqref{homdef}, a generalization of Theorem \ref{hom-thm} is provided by the following result.
\begin{theorem}\label{hom-thm-rho}
Let $F_\e[\rho]$ be defined by \eqref{funct-gen-rho} with $f_\e$ satisfying \eqref{homdef}, $\rho\in C^0(\Omega)$ and such that \eqref{boundrho} holds. Then, under the assumptions of Theorem \ref{hom-thm},
\begin{equation*}
\Gamma(L^p)\text{-}\lim_{\e\to0}F_\e[\rho](u)=\begin{cases}\displaystyle\int_\Omega f_{\rm hom}(Du(x))\rho^2(x) dx& \text{if } u\in W^{1,p}(\Omega;\mathbb{R}^m),\\
+\infty & \text{otherwise},
\end{cases}
\end{equation*}
where $f_{\rm hom}$ is defined by \eqref{homform}.
\end{theorem}
\begin{proof}
We highlight only the main differences with respect to the proof of Theorem \ref{hom-thm}. Note that, by \eqref{boundrho}, $F_\e[\rho]$ satisfies all the assumptions of Theorem \ref{reprthm}. Hence, given $\e_j\to 0$, there exists a subsequence (not relabelled) such that
$$
\Gamma(L^p)\text{-}\lim_{j\to +\infty} F_{\e_j}[\rho](u, A)=\int_A f_0 (x, D u(x))\, dx.
$$
The characterization of non-homogeneous quasiconvex functions by their minima (see \cite[Theorem II]{DM1986}) yields that for every $M\in\rr^{m\times d}$ and for almost every
$x_0\in\Omega$
$$
f_0(x_0, M)=\lim_{r\to 0} \frac{1}{r^d}\min\Big\{ \int_{Q_r(x_0)} f_0(x, Du(x)) dx \,\Big|\, u-Mx\in W_0^{1,p}(Q_r(x_0);\mathbb{R}^m)\Big\}.
$$
Then, proceeding as in the proof of Theorem \ref{hom-thm} and using the continuity of $\rho$, we obtain that
$$
f_0(x_0, M)=\rho^2(x_0)f_{\rm hom}(M).
$$
Since $f_0$ does not depend on $(\e_j)_j$, we get the conclusion.
\end{proof}
We study now a perturbed version of the class of functionals defined in \eqref{funct-gen-rho}, obtained by composing the energy densities with a family of transportation maps.
Specifically, we consider  the family of perturbed functionals $E_\e:L^p(\Omega;\mathbb{R}^m)\to[0,+\infty]$ defined by
\begin{equation}\label{slepcev}
E_\e(u):=\frac{1}{\e^d}\int_{\Omega} \int_{\Omega} f_\e\Big(x,{T_\e(y)-T_\e(x)\over\e},\frac{u(y)-u(x)}{\e}\Big) \rho(y) \rho(x)dx\, dy.
\end{equation}
Here we assume that $T_\e:\Omega\to\Omega$ is a measurable map (representing  a transportation map in Subsection \ref{pointclouds}) and $f_\e$ fulfills assumptions (H0)--(H2),
with $\psi_\e$ satisfying the additional hypothesis
 \begin{equation}\label{psiradial}
\psi_\e(\xi):=\overline\psi_\e(|\xi|),\ \text{where}\ \overline\psi_\e:[0,+\infty)\to[0,+\infty)\ \text{is non increasing}.
\end{equation}


\begin{remark}\label{slepcev-assumption}
Set
\begin{equation}\label{geps}
g_\e(x,\xi,z)=f_\e\Big(x,\frac{T_\e(x+\e\xi)-T_\e(x)}{\e},z\Big)\rho(x)\rho(x+\e\xi)
\end{equation}
and note that, under the assumptions above on $f_\e$ and $\rho$, $g_\e$ satisfies (H0)-(H2) provided the following regularity conditions are fulfilled by $T_\e$:
\begin{align}\label{transport-1}
|T_\e(y)-T_\e(x)| &\le C'(|y-x|+\e), \quad \text{if } |y-x|\le\e r' \\
\label{transport-2}
|T_\e(y)-T_\e(x)| &\ge C''|y-x|, \quad \text{if } |y-x|\ge\e r'',
\end{align}
for some $C', C'', r', r''$ positive constants.
Indeed, if $r'_0$  is a positive constant such that $r'_0<r'$ and $C'(r'_0+1)<r_0$,  then by \eqref{transport-1}
$$\Big|\frac{T_\e(x+\e\xi)-T_\e(x)}{\e}\Big|\le r_0\   \text{for every }\xi\in B_{r_0'}
$$
and assumption (H0) on $f_\e$ yields that the same assumption is satisfied by  $g_\e$ with another choice of the constants.
Moreover, denoting
\begin{equation}\label{point-clouds-2.1}
\tilde{\psi}_\e(\xi):=\sup_{x\in\Omega}\overline{\psi}_\e\Big(\Big|\frac{T_\e(x+\e\xi)-T_\e(x)}{\e}\Big|\Big),
\end{equation}
from conditions \eqref{transport-2}  and the monotonicity of $\overline \psi_\e$ we get
$$\int_{\mathbb{R}^d} \tilde{\psi}_\e(\xi)(|\xi|^p+1) d\xi \le \int_{B_{r''}}\overline \psi_\e(0)(|\xi|^p+1) d\xi + \int_{B_{r''}^c} \overline\psi_\e(C''|\xi|)(|\xi|^p+1)d\xi$$
which yields that condition \eqref{ass-bound} is satisfied by $\tilde\psi_\e$, since it is satisfied by $\psi_\e$. Analogously it can be shown that $\tilde\psi_\e$ satisfies (H2). Hence, under the assumptions \eqref{transport-1} and \eqref{transport-2}, energies as in \eqref{slepcev} belong to the class of functionals satisfying the hypotheses of Theorem \ref{reprthm}. Notice that conditions \eqref{transport-1} and \eqref{transport-2} hold in particular if $\|T_\e-id\|_\infty\le C\e$.
\end{remark}

In the next one-dimensional example we show that the asymptotic behaviour of $
E_\e$ could be degenerate if  \eqref{psiradial} is not satisfied.


\begin{example}\label{pat-pb}
Assume that in \eqref{slepcev}  $\Omega=(0,1)$, $\rho\equiv 1$ and  $f_\e(x,\xi,z)=a(\xi)|z|^p$ with $a:\mathbb{R}\to[0,+\infty)$ defined as follows
$$a(\xi)=\begin{cases}
1 &\xi\in (-1,1)\cup\mathbb{Q} \\
0 &\text{otherwise.}
\end{cases}$$
Given $\lambda_\e\to 0$, let $T_\e:(0,1)\to (0,1)$ be such that $T_\e((0,1))\subset\e\mathbb{Q}$ and  $\|T_\e-id\|_\infty\le \lambda_\e$.
In particular  \eqref{transport-1} and \eqref{transport-2} are satisfied if $\lambda_\e= o(\e)$. We may construct such maps as follows: for any $k\in \{0,\dots, \lceil \lambda_\e^{-1}\rceil-1\}$, let
$q_{k,\e}\in \mathbb{Q}\cap\big( \e^{-1}\lambda_\e[k,k+1)\big)\cap (0,\e^{-1}) $ and  set
$$
T_\e(x):=\e q_{k,\e}\ \text{if } x\in \lambda_\e[k,k+1)\ \text{for some }k\in\{0,\dots, \lceil \lambda_\e^{-1}\rceil-1\}.
$$
Since $a=\chi_{(-1,1)}$ almost everywhere, $G_\e[a]=G_\e^1$; thus, by  Proposition \ref{G-conv-conv}, $\Gamma$-$\displaystyle\lim_{\e\to 0}G_\e[a](u)<+\infty$ for any $u\in W^{1,p}(0,1)$.
Whereas, since $(T_\e y-T_\e x)/\e\in\mathbb{Q}$, $E_\e$ reads
$$E_\e(u)=\int_{-\infty}^{+\infty}\int_{0\vee(-\e\xi)}^{1\wedge(1-\e\xi)} \Big|\frac{u(x+\e\xi)-u(x)}{\e}\Big|^p dx\, d\xi\, ,$$
From which we deduce that
$$
\Gamma\text{-}\lim_{\e\to 0}E_\e(u)=\begin{cases} 0&  \text{if } u'=0 \ \text{ in } (0,1),\\
+\infty & \text{otherwise}.
\end{cases}
$$
\end{example}
In the next proposition we show that if $\|T_\e-id\|_{\infty}=o(\e)$ and $f_\e(x,\cdot,z)$ satisfies a suitable continuity assumption uniformly with respect to $\e$ and $x$, then the functionals $E_\e$ defined by \eqref{slepcev} are asymptotically equivalent in the sense of the $\Gamma$-convergence to the functionals $F_\e[\rho]$ defined by \eqref{funct-gen-rho}.


\begin{proposition}\label{slepcev-prop}
Let $F_\e[\rho]$ and $E_\e$ be defined by \eqref{funct-gen-rho} and \eqref{slepcev}, respectively,  with $f_\e$ satisfying {\rm (H0)--(H2)} and \eqref{psiradial}.
Assume in addition that:
\begin{itemize}
\item[{\rm (i)}] There exists a family of positive functions $\{\omega_h\}_{h>0}\subset L^1_{\rm loc}(\rr^d)$ such that $\omega_h\to0$ in $L^1_{\rm loc}(\rr^d)$ and
\begin{equation}\label{point-clouds-cont-ass}
\sup_{x\in\Omega}\sup_{|v|\le h} |f_\e(x,\xi+v,z)-f_\e(x,\xi,z)| \le \omega_h(\xi)|z|^p
\end{equation}
for every $\e>0$ and $z\in\mathbb{R}^m$.
\item[{\rm (ii)}] $\displaystyle\Big\|\frac{T_\e-id}{\e}\Big\|_{L^\infty(\Omega;\rr^d)}\to0$.
\end{itemize}
Then
$$\Gamma(L^p)\text{-}\liminf_{\e\to 0} E_{\e}(u)=\Gamma(L^p)\text{-}\liminf_{\e\to 0} F_{\e}[\rho](u),$$
$$\Gamma(L^p)\text{-}\limsup_{\e\to 0} E_{\e}(u)=\Gamma(L^p)\text{-}\limsup_{\e\to 0} F_{\e}[\rho](u).$$
\end{proposition}
\begin{proof}
Notice that, since both $F_\e[\rho]$ and $E_\e$ satisfy the hypotheses of Proposition \ref{growthcondremk}, we have that
$$\Gamma(L^p)\text{-}\lim_{\e\to 0} E_{\e}(u)=\Gamma(L^p)\text{-}\lim_{\e\to 0} F_{\e}[\rho](u)=+\infty, \ \text{for every } u\in L^p(\Omega;\rr^m)\setminus W^{1,p}(\Omega;\rr^m).$$
Hence, by (H0)-(H1), it suffices to prove that
\begin{equation}\label{equiv}
E_\e(u_\e)=F_\e[\rho](u_\e)+o(1)
\end{equation}
for any sequence $(u_\e)_\e$ such that $G_\e^r(u_\e)$ is uniformly bounded for some $r\leq r_0\wedge r_0'$, where $r_0$ and $r_0'$ refer to assumption (H0) for $F_\e[\rho]$ and $E_\e$, respectively.

By Lemma \ref{lim-truncated-lemma}, we may reduce to prove \eqref{equiv} in the case $f_\e (x,\xi,z)=0$ for every $x\in\Omega$, $|\xi|>R$ and $z\in\mathbb{R}^m$, for some $R>0$.
Let then $u_\e$ be such that $\sup_{\e>0} G_\e^{r}(u_\e)<+\infty$.
We have
$$E_\e(u_\e)
= F_\e(u_\e) + R_\e(u_\e),$$
where
$$R_\e(v):=\int_{B_R} \int_{\Omega_\e(\xi)} \Big (g_\e\Big(x,\xi,\frac{v(x+\e\xi)-v(x)}{\e}\Big)-f_\e\Big(x,\xi,\frac{v(x+\e\xi)-v(x)}{\e}\Big)\Big)\rho(x)\rho(x+\e\xi) dx\, d\xi\,,$$
with $g_\e$ defined by \eqref{geps}. Set
$$
h_\e:=\Big\|\frac{T_\e-id}{\e}\Big\|_{L^\infty(\Omega;\rr^d)}.
$$
By (i) and Lemma \ref{boundlemma} we get
$$|R_\e(u_\e)| \le C \int_{B_R}\int_{\Omega_\e(\xi)} \omega_{2h_\e}(\xi)\Big|\frac{u_\e(x+\e\xi)-u_\e(x)}{\e}\Big|^p dx\, d\xi \le C (R^p+1) G_\e^{r}(u_\e) \int_{B_R}\omega_{2h_\e}(\xi) d\xi,$$
which goes to zero as $\e\to0$ by (ii).
\end{proof}
As a straightforward consequence of the previous proposition and Theorem \ref{hom-thm-rho}, we obtain the following result.
\begin{corollary}\label{hompert-cor}
Under the assumptions of Proposition \ref{slepcev-prop}, asssume in addition that $f_\e$ satisfies \eqref{homdef}. Then
\begin{equation*}
\Gamma(L^p)\text{-}\lim_{\e\to0}E_\e(u)=\begin{cases}\displaystyle\int_\Omega f_{\rm hom}(Du(x))\rho^2(x)dx& \text{if } u\in W^{1,p}(\Omega;\mathbb{R}^m),\\
+\infty & \text{otherwise},
\end{cases}
\end{equation*}
where $f_{\rm hom}$ is defined by \eqref{homform}.
\end{corollary}

\subsection{Application to functionals defined on point clouds}\label{pointclouds}
The case in which $T_\e(\Omega)=X_{n_\e}:=\{x_i\}_{i=1}^{n_\e}\subset\Omega$, has already been studied in the context of problems for Machine Learning, when dealing with discrete convolution-type energies of the form
\begin{equation}\label{slepcev-dis}
\frac{1}{\e^p}\frac{1}{n_\e^2}\sum_{i,j=1}^{n_\e} a^\e_{i,j} |u(x_i)-u(x_j)|^p,
\end{equation}
that are the discrete version of energies \eqref{slepcev} when $T_\e$ are transportation maps from $\Omega$ to $X_{n_\e}$, see \cite{garsle} when $p=1$ and \cite{thorpe-et-al} when $p>1$.
Therein, $X_{n_\e}$ denotes a point cloud obtained by refining random samples of a given probability measure $\mu\ll\mathcal{L}^d$, having continuous density bounded from above and below by two positive constants.

In this subsection, we will prove a $\Gamma$-convergence result for a generalized version of discrete energies as in \eqref{slepcev-dis} defined on point clouds.
In particular, we will recover the convergence result provided in \cite{thorpe-et-al}.
Before setting the problem, we recall, for the reader's convenience, some useful notions about point-cloud models.

Let $\mu=\rho(x)\mathcal{L}^d$ be a probability measure supported on $\Omega$, such that $\rho\in C^0(\Omega)$
 and satisfies \eqref{boundrho}.
Given $(X,\sigma,\mathbb{P})$ a probability space, we consider a sequence of random variables
$$
x_i: X\ni\omega \mapsto x_i(\omega)\in\Omega, \quad i\in\NN,
$$
that are \emph{i.i.d} according to the distribution $\mu$.
Then, given $n\in\NN$, we say that the set $X_{n}(\omega)=\{x_i(\omega)\}_{i=1}^{n}$ is a \emph{point cloud} obtained as samples from a given distribution $\mu$.
In the following, we will drop the dependence on $\omega$ for the sake of simplicity of notation, unless otherwise specified. To any point cloud $X_{n}$ we associate its \emph{empirical measure}
\begin{equation}\label{empirical}
\mu_n=\frac{1}{n}\sum_{i=1}^{n} \delta_{x_i}.
\end{equation}
It is well known that $\mu_n$ weak$^*$ converge to $\mu$ as $n\to +\infty$ $\mathbb{P}$-almost surely.

Let $n_\e\in\NN$ such that
$$
\lim_{\e\to 0} n_\e=+\infty
$$
and let $f:\mathbb{R}^d\times\mathbb{R}^m\to[0,+\infty)$ be a Borel function that is convex in the second variable.
We then consider the family of functionals (``dis'' stands for discrete)
\begin{equation*}
E_\e^{\rm dis}(u) = \frac{1}{\e^d n_\e^2}\sum_{i,j=1}^{n_\e} f\Big(\frac{x_i-x_j}{\e},\frac{u(x_i)-(x_j)}{\e}\Big),
\end{equation*}
defined on functions $u:X_{n_\e}\to\mathbb{R}^m$.
Note that $E^{\rm dis}_\e$ can be written in terms of $\mu_{n_\e}$ as
\begin{equation}\label{energmueps}
E^{\rm dis}_\e(u)=\frac{1}{\e^d}\int_\Omega\int_\Omega f\Big(\frac{y-x}{\e},\frac{u(y)-u(x)}{\e}\Big) d\mu_{n_\e}(x)\, d\mu_{n_\e}(y).
\end{equation}
Let $E_\e$ be defined by \eqref{slepcev} with
\begin{equation}\label{fepsf}
f_\e(x, \xi,z)=f(\xi,z).
\end{equation}
If $T_\e:\Omega\to X_{n_\e}$ is a transportation map between $\mu_{n_\e}$ and $\mu$; that is, $(T_\e)_\#\mu=\mu_{n_\e}$,
where $T_\#\mu$ denotes the push-forward of $\mu$ by $T$, then we may identify any $u:X_{n_\e}\to\mathbb{R}^m$ with its piecewise constant interpolation on $T_\e^{-1}(x_i)$, $i=1,\dots, n_\e$, and, by \eqref{energmueps},
\begin{equation}\label{energy-cont-disc}
E^{\rm dis}_\e(u) = E_\e(u)\quad \text{for every } u\in PC(X_{n_\e}),
\end{equation}
where
$$
PC(X_{n_\e}):=\big\{u:\Omega\to\mathbb{R}^m \,|\, u \text{ is constant on } T_\e^{-1}(x_i) \text{ for every } 1\le i\le n_\e \big\}.
$$
With a slight abuse of notation we assume that $E^{\rm dis}_\e$ is defined on the whole space $L^p(\Omega;\rr^m)$ by setting
\begin{equation}\label{energy-dis}
E^{\rm dis}_\e(u) = \begin{cases}
\displaystyle\frac{1}{\e^d n_\e^2}\sum_{i,j=1}^{n_\e} f\Big(\frac{x_i-x_j}{\e},\frac{u(x_i)-(x_j)}{\e}\Big),
 & u\in PC(X_{n_\e}) \\
+\infty & \text{otherwise.}
\end{cases}
\end{equation}

In the next theorem we show that  $E_\e^{\rm dis}$ and $E_\e$ are asymptotically equivalent in the sense of $\Gamma$-convergence under the assumptions of Proposition \ref{slepcev-prop}.

\begin{theorem}\label{point-cloud-prop}
Let $X_{n_\e}$ be a family of point clouds obtained as samples from $\mu$ and let $T_\e:\Omega\to X_{n_\e}$ be a transportation map  between $\mu_{n_\e}$ and $\mu$, where $\mu_{n_\e}$ is defined in \eqref{empirical}. Let $E_\e^{\rm dis}$  be defined by \eqref{energy-dis} and let $E_\e$ be defined by \eqref{slepcev} with $f_\e$ satisfying \eqref{fepsf} and $f(\xi, \cdot)$ convex for any $\xi\in\rr^d$. If $f_\e$ and $T_\e$ satisfy the assumptions of Proposition \ref{slepcev-prop}, then
\begin{equation}\label{limdisclimcont}
\begin{aligned}
\Gamma(L^p)\text{-}\lim_{\e\to0} E^{\rm dis}_\e(u) &=\Gamma(L^p)\text{-}\lim_{\e\to0} E_\e(u)\\
&= \begin{cases}\displaystyle\int_{\rr^d}\int_\Omega f(\xi, Du(x)\xi)\rho^2(x)\, dx\, d\xi & \text{if}\ u\in W^{1,p}(\Omega;\rr^m),\\
+\infty & \text{otherwise}.
\end{cases}
\end{aligned}
\end{equation}
\end{theorem}
\begin{proof}
The second equality in \eqref{limdisclimcont} is a straightforward consequence of Corollary \ref{hompert-cor}, Theorem \ref{cell-form-thm} and Remark \ref{homcase}.
Since $E_\e^{\rm dis}\ge E_\e$, in order to prove the first equality in \eqref{limdisclimcont} it suffices  to prove that given $u\in W^{1,p}(\Omega)$
we can find $(u_\e)_\e\subset PC(X_{n_\e})$ such that $u_\e\to u$ in $L^p(\Omega;\rr^m)$ and
\begin{equation}\label{limsupdisc}
\limsup_{\e\to 0}E^{\rm dis}_\e(u_\e)\leq \int_{\rr^d}\int_\Omega f(\xi, Du(x)\xi)\rho^2(x)\, dx\, d\xi.
\end{equation}
By a density argument it suffices to prove \eqref{limsupdisc} for $u\in C^\infty(\rr^d;\rr^m)$. Fix such a function $u$ and note that, by the regularity assumptions on $f$ and assumption (ii) of Proposition \ref{slepcev-prop},   we have that
\begin{equation}\label{limsupcont}
\lim_{\e\to 0}E_\e(u)=\int_{\rr^d}\int_\Omega f(\xi, Du(x)\xi)\rho^2(x)\, dx\, d\xi.
\end{equation}
Let $u_\e\in PC(X_{n_\e})$ defined by
$$
u_\e(x_i) := \frac{1}{|V_\e^i|} \int_{V_\e^i}u(y)\, dy\quad i\in\{1,\dots, n_\e\},
$$
where we have set, for $i=1,\dots, n_\e$,
$$
V_\e^i=T_\e^{-1}(x_i).
$$
Note that $V_\e^i\subseteq B_{r_\e}(x_i) $, where
$$
r_\e:=\|T_\e-Id\|_\infty.
$$
Hence, by assumption (ii) of Proposition \ref{slepcev-prop},
$$
\|u_\e-u\|_{L^\infty(\Omega;\rr^m)}\leq C r_\e\to 0.
$$
By the convexity of $f(\xi,\cdot)$, we get for any $1\le i,j\le n_\e$
\begin{equation}\label{discineq}
\begin{aligned}
f\Big(\frac{x_i-x_j}{\e},\frac{ u_\e(x_i)- u_\e(x_j)}{\e}\Big) &= f\Big(\frac{x_i-x_j}{\e},\frac{1}{|V_\e^i|}\int_{V_\e^i}\frac{u(y)}{\e}\,dy-\frac{1}{|V_\e^j|}\int_{V_\e^j}\frac{u(x)}{\e}dx\Big) \\
&\le \frac{1}{|V_\e^i|}\int_{V_\e^i} f\Big(\frac{x_i-x_j}{\e},\frac{u(y)}{\e}-\frac{1}{|V_\e^j|}\int_{V_\e^j}\frac{u(x)}{\e}dx\Big)dy \\
&\le \frac{1}{|V_\e^i||V_\e^j|}\int_{V_\e^i}\int_{V_\e^j} f\Big(\frac{x_i-x_j}{\e},\frac{u(y)-u(x)}{\e}\Big)\, dy\, dx.
\end{aligned}
\end{equation}
Notice that, since $T_\e$ is a transportation map between $\mu$ and $\mu_\e$, by the continuity of $\rho$ we get
$$\frac{1}{|V_\e^i|}=(\rho(x_i) + o(1))\, n_\e,$$
uniformly in $1\le i\le n_\e$.
Hence, by \eqref{discineq} we get
\begin{align*}
E_\e^{\rm dis}(u_\e) &= \frac{1}{\e^d n_\e^2} \sum_{i,j=1}^{n_\e} f\Big(\frac{x_i-x_j}{\e},\frac{ u_\e(x_i)- u_\e(x_j)}{\e}\Big) \\
&\le \frac{1}{\e^d n_\e^2} \sum_{i,j=1}^{n_\e} \frac{1}{|V_\e^i||V_\e^j|}\int_{V_\e^i}\int_{V_\e^j} f\Big(\frac{x_i-x_j}{\e},\frac{u(y)-u(x)}{\e}\Big)\, dy\, dx \\
& = \frac{1}{\e^d} \sum_{i,j=1}^{n_\e}\int_{V_\e^i}\int_{V_\e^j} f\Big(\frac{x_i-x_j}{\e},\frac{u(y)-u(x)}{\e}\Big)\rho(x_i)\rho(x_j)\, dy\, dx + o(1) \\
&= E_\e(u_\e) + o(1),
\end{align*}
and the conclusion follows from \eqref{limsupcont}.
\end{proof}

\begin{remark}
In \cite{GS}, extending previous results, it has been proved that for $d\ge 2$ almost surely there exists a family of transportation maps $T_n(\omega):\Omega\to X_n(\omega)$ between $\mu$ and $\mu_n$ such that
$$
0<c\le \liminf_{n\to+\infty} \frac{ \|T_n(\omega)-id\|_{L^\infty}}{l_n} \le\limsup_{n\to+\infty} \frac{ \|T_n-id\|_{L^\infty}}{l_n} \le C,
$$
where
$$
l_n=\begin{cases}\displaystyle \frac{(\log n)^{3/4}}{n^{1/2}} & \text{if}\ d=2,\\
\displaystyle\Big(\frac{\log n}{n}\Big)^{1/d}& \text{if}\ d\geq 3.
\end{cases}
$$
Hence, if $n_\e\to+\infty$ and
\begin{equation}\label{rate}
\lim_{\e\to0}\frac{\log{n_\e}}{n_\e \e^d}=0
\end{equation}
the corresponding transportation maps $T_\e=T_{n_\e}$ are such that $\|T_\e-id\|_{L^\infty}=o(\e)$ and, in particular, assumptions (ii) of Proposition \ref{slepcev-prop} is satisfied. Hence if \eqref{rate} is satisfied the $\Gamma$-convergence result stated in Theorem \ref{point-cloud-prop} holds almost surely, thus extending the convergence result provided in \cite{thorpe-et-al}, where the analysis is limited to energy densities of the form $f(\xi,z)=a(|\xi|)|z|^p$,
with $a(\cdot)$ non increasing.
\end{remark}


\section{Stochastic homogenization}\label{stocks}
In this section we consider random  energies of convolution type and prove that, under stationarity and ergodicity assumptions, the $\Gamma$-limit of such energies is almost surely a deterministic integral functional whose integrand can be characterized through an asymptotic formula. This result extends \cite[Theorem 6.1]{2019BP}, where quadratic convolution energies with random coefficients are studied.

Let $(X,\sigma,\mathbb{P})$ be a standard probability space with a measure-preserving ergodic  dynamical system
$\tau_y$, $y\in\mathbb R^d$. We recall that $\{\tau_y\}$ is a collection of measurable invertible maps
$\tau_y:X\mapsto X$ such that
\begin{itemize}
  \item [-]  $\tau_{y+y'}=\tau_y\circ\tau_{y'}$ for all $y,y'\in\mathbb{R}^d$, $\tau_0=\mathrm{Id}$,
  \item [-] $\mathbb P(\tau_y(A))=\mathbb P(A)$ for all $A\in\sigma$ and
  $y\in\mathbb R^d$,
  \item [-] $\tau:X\times\mathbb R^d\mapsto X$ is measurable, where $\mathbb R^d$ is equipped
  with the Borel $\sigma$-algebra.
\end{itemize}
The ergodicity of $\tau$ means that for every $A\in\sigma$ such that $\tau_y(A)=A$ for all $y\in\mathbb{R}^d$ there holds either $\mathbb{P}(A)=0$ or $\mathbb{P}(A)=1$.


Consider now  a random function 
$\mathtt{f}$ defined as a measurable map
$$
\mathtt{f}:X\times\mathbb R^d\times\mathbb{R}^m\to[0,+\infty),
$$
where $\mathbb{R}^d$ and $\mathbb{R}^m$ are equipped with the Borel $\sigma$-algebra.
We assume that
\begin{equation}\label{growth-cond-stocha0}
c_0(|z|^p-\rho(\omega,\xi))\leq \mathtt{f}(\omega,\xi,z) \ \text{if}\ |\xi|\le r_0,
\end{equation}
\begin{equation}\label{growth-cond-stocha}
\mathtt{f}(\omega,\xi,z) \le \psi(\omega,\xi)(|z|^p+1)
\end{equation}
for some positive constants $c_0$ and $r_0$, and measurable random functions $\rho(\omega,\cdot): B_{r_0}\to [0,+\infty)$ and $\psi(\omega,\cdot):\rr^d\to [0,+\infty)$ satisfying
\begin{equation}\label{growth-cond-stocha-int}
C_0(\cdot):=\int_{B_{r_0}}\psi(\cdot,\xi)\,d\xi \in L^1(X,\mathbb{P}), \quad C_1(\omega):=\int_{\mathbb{R}^d}\psi(\omega)(\xi)(|\xi|^p+1)\,d\xi \in L^1(X,\mathbb{P}).
\end{equation}
Letting  $f(\omega)(y,\xi,z)=\mathtt{f}(\tau_y\omega,\xi,z)$, we introduce a family of stochastic non-local energy functionals $F_\e(\omega):L_{\rm loc}^p(\rr^d;\rr^m)\times \mathcal{A}(\rr^d)\to [0, +\infty)$, $\e>0$,
defined by
\begin{equation}\label{stocha}
F_\e(\omega)(u,A) := \int_{\mathbb{R}^d}\int_{A_\e(\xi)} f(\omega)\Big(\frac{x}{\e},\xi,\frac{u(x+\e\xi)-u(x)}{\e}\Big)dx\,d\xi.
\end{equation}



By construction the densities $f$ are \emph{statistically homogeneous} functions of $x$ that is
\begin{equation}\label{stationarity}
f(\omega)(x+y,\xi,z)=f(\tau_y\omega)(x,\xi,z), \quad \text{for every }\omega\in X,\, x,y,\xi\in\mathbb{R}^d,\, z\in\mathbb{R}^m.
\end{equation}

In Theorem \ref{erg-the} below we prove a homogenization theorem for the functionals $F_\e$. Our proof of a homogenization formula relies on a subadditive ergodic theorem, a result that we recall below preceded by the definition of multiparameter stationary stochastic processes.

\begin{definition}[subadditive process]\label{sub-process-def}
Let $\mathcal{V}$ be the family of all finite subset of the lattice $\ZZ^{d,+}:=\{0,1,\dots\}^d$. A real valued process $\Psi:\mathcal{V}\to L^1(X, \mathbb{P})$
 is called a \emph{subadditive process} if it satisfies the following conditions:
\begin{itemize}
\item[(i)] it is stationary, that is  for any $j\in\ZZ^{d,+}$ and any finite collection $\{V_1,\dots,V_N\}\subset\mathcal{V}$ the joint law of $\{\Psi(V_1+j),\dots, \Psi(V_N +j)\}$ is the same as the joint law of $\{\Psi(V_1),\dots, \Psi(V_N)\}$;
\item[(ii)] it is subadditive, that is  $\Psi({V_1\cup V_2})\le  \Psi(V_1)+\Psi(V_2)$  for any disjoint $V_1$ and $V_2$ in $\mathcal{V}$;
\item[(iii)] there holds
$$\inf_{n\in\NN} \int_X \frac{1}{n^d} \Psi(\{0,1,\dots n\}^d)(\omega)\, d\mathbb{P}(\omega)>-\infty.$$
\end{itemize}
\end{definition}

\begin{theorem}[Theorem 1 \cite{KP1987}]\label{KP-thm}
Let $\mathcal{B}_0$ be a family of Borel subsets of $[0,1]^d$ such that
$$\sup\{|\partial B+B_\delta| \,:\, B\in\mathcal{B}_0\}\to0, \quad \text{as }\delta\to0.$$
Then, for every subadditive process $\Psi:\mathcal{V}\to L^1(X, \mathbb{P})$  there exists a real random variable $\phi\in L^1(X,\mathbb{P})$ such that
\begin{equation}\label{stocha-asy-form}
\sup\Big\{\Big|\frac{1}{N^d}\Psi((NB)\cap \ZZ^{d,+})-|B|\phi\Big| \,:\, B\in\mathcal{B}_0\Big\}\to0 \quad \text{almost surely as } N\to+\infty.
\end{equation}
\end{theorem}

\begin{theorem}\label{erg-the}
Let $X\ni\omega\mapsto f(\omega)$ be a statistically homogeneous random function, according to \eqref{stationarity}, satisfying \eqref{growth-cond-stocha0}--\eqref{growth-cond-stocha-int}.
Let $F_\e$ be defined as in \eqref{stocha} and assume that $F_\e(\omega)$ satisfies assumption {\rm(H0$'$)} almost surely.
Then for $\mathbb{P}$-almost every $\omega\in X$ and for every $M\in\mathbb{R}^{m\times d}$ the limit
\begin{equation}\label{stocha-form}
f_{\rm hom}(\omega)(M) = \lim_{R\to+\infty} \frac{1}{R^d} \inf\big\{ F_1(\omega)(u,Q_R) \,\big|\, u\in\mathcal{D}^{1,M}(Q_R)\big\},
\end{equation}
where  $\mathcal{D}^{1,M}(Q_R)$ is defined by \eqref{set-boudary-data},
exists and defines a quasiconvex function $f_{\rm hom}(\omega):\mathbb{R}^{m\times d}\to[0,+\infty)$ satisfying
\begin{equation}\label{growthfhom}
c(\omega)(|M|^p-1)\le f_{\rm hom}(M) \le C(\omega)( |M|^p+1),
\end{equation}
where $c(\cdot), C(\cdot)\in L^1(X,\mathbb{P})$. Moreover, for $\mathbb{P}$-almost every $\omega\in X$ and for every $A\in\mathcal{A}^{\rm reg}(\Omega)$ there holds
$$\Gamma\text{-}\lim_{\e\to0}F_\e(\omega)(u,A) = F(\omega)(u,A) :=\begin{cases}\displaystyle
\int_A f_{\rm hom}(\omega)(Du(x)) dx & u\in W^{1,p}(A;\mathbb{R}^m) \\
+\infty & \text{otherwise}.
\end{cases}$$
If, in addition, the dynamical system $\tau$ is ergodic, then $f_{\rm hom}(\omega)(\cdot)$ is constant almost surely and satisfies
\begin{equation}\label{ergodic}
f_{\rm hom}(\omega)(M) \equiv f_{\rm hom}(M) := \lim_{R\to+\infty} \frac{1}{R^d} \int_X \inf\big\{ F_1(\omega)(u,Q_R) \,\big|\, u\in\mathcal{D}^{1,M}(Q_R)\big\} d\mathbb{P}(\omega).
\end{equation}
\end{theorem}
\begin{proof}
For any given $T>0$ let $F^T_\e(\omega)$ be the truncation functional of $F_\e(\omega)$ as defined in \eqref{truncated-functionals}.
By Theorem \ref{reprthm}, for $\mathbb{P}$-almost every $\omega\in X$  and for every sequence  $\e_j\to0$ there exists a subsequence (not relabelled) and a Carath\'eodory function $f^T_0(\omega):\Omega\times\rr^{m\times d}\to [0,+\infty)$, which is quasiconvex in the second variable, such that for every $A\in\mathcal{A}^{\rm reg}(\Omega)$ and $u\in W^{1,p}(A;\mathbb{R}^m)$
$$\Gamma(L^p)\text{-}\lim_{j\to+\infty}F^T_{\e_j}(\omega)(u,A)=F^T(\omega)(u,A):=\int_A f^T_0(\omega)(x,D(u(x)) dx.
$$

Arguing as in the proof of Theorem \ref{hom-thm-rho}, for every $M\in\rr^{m\times d}$ and for almost every $x_0\in\Omega$ we have
\begin{align*}
f_0^T(\omega)(x_0, M) &= \lim_{r\to 0} \frac{1}{r^d}\min\big\{ F^T(\omega)(u,Q_r(x_0)) \,\big|\, u-Mx\in W^{1,p}_0(Q_r(x_0);\mathbb{R}^m)\big\} \\
&= \lim_{r\to0} \lim_{j\to+\infty} \frac{1}{R^d_j}\inf\big\{ F^T_1(\omega)(v,R_jQ_{1}(r^{-1}x_0)) \,\big|\, v\in\mathcal{D}^{1,M}(R_jQ_{1}(r^{-1}x_0))\big\}\\
&= \lim_{r\to0} \lim_{j\to+\infty} \frac{1}{R^d_j}\inf\big\{ F^T_1(\omega)(v,R_jQ_{1}(r^{-1}x_0)) \,\big|\, v\in\mathcal{D}^{T,M}(R_jQ_{1}(r^{-1}x_0))\big\},
\end{align*}
with $R_j=r/\e_j$.
Set for any $A\in\mathcal{A}(\rr^d)$
$$H^T(\omega)(M,A):=\inf\big\{ F^T_1(\omega)(v,Q) \,\big|\, v\in\mathcal{D}^{T,M}(A)\big\}.
$$
We now show that there exists $\phi_M^T: X\to [0,+\infty)$ such that for any $\bar x\in\rr^d$
$$
\lim_{R\to +\infty}\frac{1}{R^d}H^T(\omega)(M, RQ_1(\bar x))=\phi_M^T(\omega).
$$
Set
\begin{equation}\label{Htilde}
\tilde H^T(\omega)(M, A):=\inf\Big\{ \int_{B_T}\int_A f(\omega)(x,\xi, v(x+\xi)-v(x))\, dx\, d\xi \,\big|\, v\in\mathcal{D}^{T,M}(A)\Big\}.
\end{equation}
Since
$$
|\tilde H^T(\omega)(M,RQ_1(\bar x))-H^T(\omega)(M,RQ_1(\bar x))|\leq 2d c_1(\omega)(|M|^p+1)TR^{d-1},
$$
it suffices to show that
\begin{equation}\label{limphi}
\lim_{R\to +\infty}\frac{1}{R^d}\tilde H^T(\omega)(M,RQ_1(\bar x))=\phi^T(\omega).
\end{equation}
For any $A\in \mathcal V$ denote $Q^A = \bigcup_{j\in A} Q_1(j)$  and $\Psi^T(M,A) (\omega)=\tilde H^T(\omega)(M, Q^A)$. Note that  $\Psi^T(M,A)\in L^1(X,\mathbb{P})$ for any $A\in \mathcal V$ by conditions \eqref{growth-cond-stocha} and \eqref{growth-cond-stocha-int}, and that by Definition \eqref{Htilde}, $\Psi^T$ is subadditive, according to  Definition \ref{sub-process-def} (ii). Moreover, since
$f$ is statistically homogeneous, $\Psi^T$ is stationary, according to  Definition \ref{sub-process-def} (i).
Condition (iii) of Definition \ref{sub-process-def} is trivially satisfied, since $f$ takes values in $[0,+\infty)$.
Hence, by appying Theorem \ref{KP-thm} to the process  $\Psi^T$ we deduce that there exists $\phi^T_M\in L^1(X,\mathbb{P})$  such that for every
$N>0$
$$
\lim_{R\to +\infty}\sup \Big\{\Big|\frac{\tilde H^T(\omega)(M, R Q_1(x))}{R^d}-\phi_M^T(\omega)\Big|:\ |x|\leq N\Big\}=0,
$$
which in particular yields \eqref{limphi}.
This on the one hand implies that $f_\omega^T(x_0,M)$ does not depend on the first variable and, on the other hand, that it does not depend on the subsequence $\{\e_j\}$.
Hence, the whole family $F_\e^T(\omega)$ $\Gamma$-converges to $F^T(\omega)$.
Arguing as in the proof of Theorem \ref{hom-thm} we get the same result for $F_\e(\omega).$
If $\tau$ is ergodic, then $f^T(M)$ and $F^T$ are deterministic, so is $f(M)$ and \eqref{ergodic} holds.
\end{proof}

We complete this section by providing a typical example of a random density $\mathtt{f}=\mathtt{f}(\omega,\xi,z)$.

\begin{example}\rm
Let $\mathcal{Y}$ be a Poisson point process in $\mathbb R^d$ of intensity $1$; that is, $\mathcal{Y}$ is a point process such that for any bounded Borel set $A$ in $\mathbb R^d$ the number of points of $\mathcal{Y}$ in $A$ has a Poisson distribution with
parameter $|A|$, and for any $N$ and any disjoint bounded Borel sets $A_1,\ldots, A_N$ the random variables
$\#(A_1\cap \mathcal{Y}),\ldots, \#(A_N\cap \mathcal{Y})$ are independent.

Let $\varphi$ be a $C_0^\infty(\mathbb R^d)$ function such that $\varphi\geq 0$.  We set
$$
f(\omega)(x,\xi,z)=\mathbf{1}_{|\xi|\leq 1}(1+|z|)^p+\Big(1+\sum\limits_j\varphi(x-Y_j)\Big)^{-1}
\Big(\sum\limits_j\varphi(x-Y_j)\Big) h(\xi)(1+|z|)^p,
$$
where $\mathcal{Y}(\omega)=\bigcup_{j=1}^\infty Y_j(\omega)$,   $p>1$, and $h(\cdot)$
is a non-negative function that satisfies the following upper bound:
$$
h(\xi)\leq\frac C{(1+|\xi|)^{d+p+\delta}}
$$
for some $\delta>0$. In this example   $\mathtt{f}(\omega,\xi,z)=f(\omega)(0,\xi,z)$.
\end{example}

\section{Application to convex gradient flows}\label{dynamic}
So far, we have studied the $\Gamma$-convergence of families of functionals $\{F_\e\}$ by dealing with static problems.
Here, by taking advantage of the $\Gamma$-convergence results proved, we analyze some dynamical aspects by considering gradient flows for a convex family $\{F_\e\}$, and prove the stability of the gradient flows with respect to $\Gamma$-convergence by applying the minimizing movements scheme along this family of functionals (see \cite[Chapter 8]{bra4}).

We first recall some definitions and relevant results to our end.
Let $F:L^2(\Omega;\rr^m)\to[0,+\infty]$ be a lower semicontinuous, convex and proper ($F\not\equiv +\infty$) functional.
We introduce a time-scale parameter $\tau>0$ and we solve the sequence of minimum problems starting from the initial value $u_0\in L^2(\Omega;\rr^m)$; that is,
\begin{equation}\label{MM-scheme1}\begin{cases}\displaystyle
u_n^\tau \in \argmin_{u\in L^2(\Omega;\rr^m)} \Big\{F(u)+\frac{1}{2\tau} \big\|u-u_{n-1}^\tau\big\|^2_{L^2(\Omega;\rr^m)}\Big\}, & n\ge1\\
u_0^\tau \equiv u_0\,.
\end{cases}
\end{equation}
By applying the direct method of Calculus of Variations we get that the solutions $u_n^\tau$ exist for any $n\in\NN$. Moreover, by the convexity of $F$, the functional $u\mapsto F(u)+c\|u-v\|_{L^2(\Omega;\rr^m)}^2$ is strictly convex for every fixed $c>0$ and $v\in L^2(\Omega;\rr^m)$; hence, the solutions $u_n^\tau$ are also unique for any $n\in\NN$.
The sequence $\{u_n^\tau\}_{n\ge0}$ is called a \emph{discrete solution} of the scheme \eqref{MM-scheme1}.
A discrete solution extends to an interpolation curve $u^\tau:[0,+\infty)\to L^2(\Omega;\rr^m)$ defined by
\begin{equation}\label{disc-sol}
u^\tau(t) := u_n^\tau, \quad t\in[(n-1)\tau,n\tau).
\end{equation}

\begin{definition}[Minimizing movements]
A curve $u:[0,+\infty)\to L^2(\Omega;\rr^m)$ is called a \emph{minimizing movement for} $F$ \emph{from $u_0$} if 
$u^\tau$, defined as in \eqref{MM-scheme1} and \eqref{disc-sol},
up to subsequences, converge to $u$ as $\tau\to0$ uniformly on compact sets of $[0, +\infty)$.
\end{definition}

By  \cite[Theorem 11.1]{bra4} the convexity of $F$ ensures that there exists a unique minimizing movement in  $C^{1/2}([0,+\infty);L^2(\Omega;\rr^m))$.

\smallskip
If $F(u)<+\infty$ then a subgradient of $F$ at $u$ is a function $\varphi\in L^2(\Omega;\rr^m)$ satisfying the following inequality
\begin{equation}\label{subdiff}
F(v) \ge F(u) + \int_\Omega \langle\varphi(x),v(x)-u(x)\rangle dx \,,
\end{equation}
for every $ v\in L^2(\Omega;\rr^m) $.
For each $u\in L^2(\Omega;\rr^m)$ we denote by $\partial F(u) $ the set of all subgradients of $F$ at $u$.
The \emph{subdifferential} of $F$ is the multivalued mapping $\partial F$ which assigns the set $\partial F(u) $ to each $u$.
The domain of $\partial F$ is given by the set ${\rm dom }\, \partial F =\{ u \in L^2(\Omega;\rr^m) \,|\, \partial F(u)\not= \emptyset \}$.
If $u\in {\rm dom }\, \partial F$ then there exists a unique element of $\partial F(u)$ having minimal norm that is denoted by
\begin{equation}\label{upp-gra}
\partial^0 F(u) := \argmin_{\varphi\in\partial F(u)} \|\varphi\|_{L^2(\Omega;\rr^m)}\,,
\end{equation}
see for istance \cite[Section 1.4]{AGS2005}.

\begin{definition}[{\cite[Definition 1.3.2, Remark 1.3.3]{AGS2005}}]\label{defAC}
A locally absolutely continuous map $u: [0,+\infty)\to L^2(\Omega;\rr^m)$ is called a \emph{curve of maximal slope for} $F$ if it satisfies the energy identity
\begin{equation}\label{EDP}
F(u(t_1))-F(u(t_2)) = \frac{1}{2}\int_{t_1}^{t_2} \|u'(s)\|^2_{L^2(\Omega;\rr^m)}\,ds + \frac{1}{2}\int_{t_1}^{t_2}\,\|\partial^0 F\|_{L^2(\Omega;\rr^m)}^2(u(s))\,ds\,,
\end{equation}
for any interval $[t_1, t_2]\subset [0, +\infty)$.
\end{definition}

Note that $u$, as in Definition \ref{defAC}, satisfies \eqref{EDP} if and only if $u\in W^{1,2}_{\rm loc}([0,+\infty); L^2(\Omega;\rr^m))$ and it is solution to the \emph{gradient flow} equation
\begin{equation}\label{GF-eq}
u'(t) = - \partial^0 F(u(t)), \quad \text{for almost every } t>0,
\end{equation}
see for instance \cite[Corollary 1.4.2]{AGS2005}.

\begin{lemma}\label{MM-GF-lem}
Let $F:L^2(\Omega;\rr^m)\to[0,+\infty]$ be a proper, convex and lower-semicontinuous functional.
Then, for every $u_0\in L^2(\Omega;\rr^m)$, with $F(u_0)<+\infty$, there exists a unique minimizing movement $u\in W^{1,2}_{\rm loc}([0,+\infty); L^2(\Omega;\rr^m))$ for $F$ from $u_0$ which is the unique solution to the gradient flow
\eqref{GF-eq} with initial condition $u(0)=u_0$.
\end{lemma}

\begin{proof}
We already observed that for such functional $F$ there exists a unique minimizing movement $u\in C^{1/2}([0,+\infty);L^2(\Omega;\rr^m))$. By \cite[Theorem 2.3.3]{AGS2005}, we have that the  minimizing movement is a curve of maximal slope, which concludes the proof.
\end{proof}

Instead of a single functional, we now consider a family of functionals  $F_\e:L^2(\Omega;\rr^m)\to[0,+\infty]$ for $\e>0$, that are proper, lower semicontinuous and convex, and $u_0^\e\in L^2(\Omega;\rr^m)$ a given family of initial data. We apply the minimizing movements scheme \eqref{MM-scheme1}
with $F_\e$ in place of $F$ and, similarly,
 we get that there exists a unique discrete solution $\{u_n^{\tau,\e}\}_{n\ge0}$ and an interpolation curve  $u^{\tau,\e}:[0,+\infty)\to L^2(\Omega;\rr^m)$  as in \eqref{MM-scheme1} and \eqref{disc-sol}, respectively, depending on the parameter $\e$.

\begin{definition}[Minimizing movements along families of functionals]
Consider $u_0^\e\to u_0$ in $L^2(\Omega;\rr^m)$ and let $\{\tau_\e\}_{\e>0}$ be a family of positive parameters such that $\tau_\e\to0$ as $\e\to0$.
A curve $u:[0,+\infty)\to L^2(\Omega;\rr^m)$ is called a \emph{minimizing movement along} $\{F_\e\}_{\e>0}$ \emph{from} $u^\e_0$ \emph{at rate} $\tau_\e$ if, up to subsequences, $u^{\tau_\e,\e}$ converges to $u$, as $\e\to0$, on compact subsets of $[0, +\infty)$.
\end{definition}

Note that, in general, the minimizing movements $u$ may depend on the rate $\tau_\e$
(see e.g. \cite[Example 8.2]{bra4} and \cite{ABZ2019}).
This does not occur for convex families of functionals, for which gradient flows are stable with respect to $\Gamma$-convergence (see e.g. \cite[Theorem 11.2]{bra4} ). In the following result we assume $p\ge 2$ for technical reasons. 



\begin{theorem}\label{stability-thm}
Let $p\ge 2$ and let $F_\e$ be defined as in \eqref{functionals} with $f_\e$ convex in the last variable.
Assume that {\rm(H0)--(H2)} hold and that $F_\e$ $\Gamma(L^p)$-converge to $F:L^p(\Omega;\rr^m)\to[0,+\infty]$, as $\e\to0$.
Let $\{u_0^\e\}\subset L^p(\Omega;\rr^m)$ be a given family of initial data such that
\begin{equation}\label{bound-initial}
\sup_{\e>0}F_\e(u_0^\e)<+\infty, \quad u_0^\e\to u_0 \text{ in } L^2(\Omega;\rr^m)
\end{equation}
as $\e\to0$. Let $\bar{u}$ and $u^\e$ be the minimizing movements for $F$ from $u_0$ and for $F_\e$ from $u_0^\e$, respectively.
Then, for every $\tau_\e\to 0$ as $\e\to0$, we have that $\bar{u}: [0,+\infty)\to W^{1,p}(\Omega;\rr^m)$ is the unique minimizing movement along $\{F_\e\}_{\e>0}$ from $u_0^\e$ at rate $\tau_\e$ and satisfies
\begin{equation}\label{reg-MM}
\bar{u}\in  C^0([0,+\infty);L^p(\Omega;\rr^m)) \cap W_{\rm loc}^{1,2}([0,+\infty);L^2(\Omega;\rr^m))\,.
\end{equation}
Moreover, we have that $u^\e$ are solutions to the gradient flows for $F_\e$\ie
$$
\begin{cases}
(u^\e)'(t)=-\partial^0 F_\e(u^\e(t))  & \text{for almost every } t>0 \\
u^\e(0)=u_0^\e,
\end{cases}
$$
$\bar{u}$ is solution to the gradient flow for $F$\ie
$$
\begin{cases}
u'(t)=- \partial^0 F(u(t))  & \text{for almost every } t>0 \\
u(0)=u_0\,
\end{cases}
$$
and $\{u^\e\}$ converges to $\bar{u}$ as follows
\begin{equation}\label{pointwise-conv-MM}
\lim_{\e\to0} u^\e(t) = \bar{u}(t), \quad \text{in } L^p(\Omega;\rr^m) \text{ for every } t>0,
\end{equation}
\begin{equation}\label{holder-conv-MM}
\lim_{\e\to0} u^\e = \bar{u}, \quad \text{weakly in } W_{\rm loc}^{1,2}([0,+\infty);L^2(\Omega;\rr^m)).
\end{equation}
\end{theorem}

\begin{proof} We extend the functionals $F_\e$ and $F$ to $L^2(\Omega;\rr^m)$ by setting $F_\e(u)=+\infty$ and $F(u)=+\infty$ for every $u\in L^2(\Omega;\rr^m)\setminus L^p(\Omega;\rr^m)$.
Notice that, by the assumptions on $f_\e$ we have that the functionals $F_\e$ are convex and lower semicontinuous in $L^2(\Omega;\rr^m)$ and the same holds for $F$ because $\Gamma$-limit.
By assumption (H0), Proposition \ref{pwineq} and Theorem \ref{kolcom} every $\{u_\e\}_{\e>0}\subset L^2(\Omega;\rr^m)$ satisfying
\begin{equation}\label{equi-c}
\sup_{\e>0}( \|u_\e\|_{L^2(\Omega;\rr^m)}+F_\e(u_\e) ) <+\infty
\end{equation}
is precompact in $L^p(\Omega;\rr^m)$ and therefore in $L^2(\Omega;\rr^m)$.
Hence, we get that $\{F_\e\}$ is $L^2$-equicoercive and $\Gamma (L^2)$-converges to $F$, as $\e\to 0$.
Because of \cite[Remark 11.2]{bra4}, we can apply \cite[Theorem 11.2]{bra4} even though the functionals $F_\e$ are not $L^2$-coercive for every fixed $\e>0$.
Thus, there exists a unique minimizing movement along $\{F_\e\}$ from $u_0^\e$ at rate $\tau_\e$ which coincides with $\bar{u}$, and $\{u^\e\}$ converge uniformly to $\bar{u}$ on compact subset of $[0,+\infty)$.
By Lemma \ref{MM-GF-lem} we have that $\bar{u}\in W^{1,2}_{\rm loc}(\Omega;L^2(\Omega,\rr^m))$ and $\bar{u}(t)\in W^{1,p}(\Omega;\rr^m)$ for every $t\in [0,+\infty)$.
Moreover, the decreasing behavior of $F_\e$ along $u^\e$, \eqref{EDP} and \eqref{bound-initial} give \eqref{holder-conv-MM}.
It remains to prove \eqref{pointwise-conv-MM} and the continuity of $\bar{u}$ with respect to the $L^p$-topology.
By the monotonicity of $F(\bar{u}(t))$ and \eqref{growth-cond-below}, we infer that $\|D\bar{u}(t)\|_{L^p(\Omega;\rr^m)} \le C(F(u_0)+1)$ for every $t\ge0$.
Therefore, by  the strong $L^2$-continuity of $\bar{u}$, we get that $\bar{u}(s)\to \bar{u}(t)$ weakly in $W^{1,p}(\Omega;\rr^m)$ as $s\to t$, which, in particular, gives \eqref{reg-MM}.
Finally, since $\{u^\e(t)\}$ satisfies \eqref{equi-c} for every $t>0$, then \eqref{pointwise-conv-MM} is implied by \eqref{holder-conv-MM}.
\end{proof}

\subsection{Homogenized flows for convex energies}

In this section we apply the results obtained in Theorem \ref{stability-thm} to the periodic-homogenization case.
We describe the homogenized gradient flow in \eqref{HomGF-N} and \eqref{HomGF-D} under Neumann and Dirichlet boundary conditions, respectively.

\begin{theorem}\label{stability-hom-N}
Let $p\ge 2$, let $F_\e$ be defined as in \eqref{functionals} with $f_\e$ convex and $C^2$ in the last variable and let \eqref{homdef}--\eqref{growthpsi} hold.
Let $\{u_0^\e\}\subset L^p(\Omega;\rr^m)$ be a given family of initial data satisfying  \eqref{bound-initial}.
Then, 
\begin{equation}\label{gradient-hom}
\partial^0 F_\e(u)(x) = \frac{1}{\e^{d+1}} \int_\Omega \Bigl(\partial_z f\Big(\frac{y}{\e},\frac{x-y}{\e},\frac{u(x)-u(y)}{\e}\Big)-\partial_z f\Big(\frac{x}{\e},\frac{y-x}{\e},\frac{u(y)-u(x)}{\e}\Big)\Bigr) dy\,,
\end{equation}
Theorem \ref{stability-thm} holds and the family of solutions $\{u_\e\}$ to the gradient flows
\begin{equation}\label{GF-hom-eps}
\begin{cases}
\partial_t u^\e = - \partial^0 F_\e(u), & \text{in } (0,+\infty)\times\Omega, \\
u^\e(0)=u_0^\e,
\end{cases}
\end{equation}
converges weakly in $W^{1,2}_{\rm loc}([0,+\infty);L^2(\Omega;\rr^m))$ to the solution to the gradient flow
\begin{equation}\label {HomGF-N}
\begin{cases}\displaystyle
\partial_t u = \divv(Df_{\rm hom}(Du)), & \text{in } (0,+\infty)\times\Omega, \\ \displaystyle
Df_{\rm hom}(Du)\,\nu = 0, & \text{on } \partial\Omega,\\
u(0)=u_0,
\end{cases}
\end{equation}
with $f_{\rm hom}$ defined by \eqref{cellform}.
\end{theorem}
\begin{proof}
By Theorem \ref{cell-form-thm} we have that
$$
\Gamma(L^p)\text{-}\lim_{\e\to0}F_\e(u) = F(u) :=
\begin{cases}\displaystyle
\int_\Omega f_{\rm hom}(Du(x))dx & u\in W^{1,p}(\Omega;\rr^m) \\
+\infty & \text{otherwise}\,.
\end{cases}
$$
For every $u\in L^p(\Omega;\rr^m)$ and $v\in L^{p'}(\Omega;\rr^m)$ the first variation of $F_\e$ is given by
$$
\int_\Omega \langle \frac{\delta F_\e}{\delta u}(x), v(x)\rangle \, dx = \frac{1}{\e^d}\int_\Omega\int_\Omega \langle\partial_z f\Big(\frac{x}{\e},\frac{y-x}{\e},\frac{u(y)-u(x)}{\e}\Big),\frac{v(y)-v(x)}{\e}\rangle\,dy\,dx.
$$
Thus \eqref{gradient-hom} follows by the symmetric roles of $x$ and $y$.
For the limit functional $F$ we have that
$$
\text{dom }\partial F = \Big\{ u \in W^{1,p}(\Omega;\rr^m) \,\Big|\, \divv(Df_{\rm hom}(Du))\in L^2(\Omega;\rr^m), Df_{\rm hom}(Du)\nu=0 \text{ on } \partial\Omega\Big\}
$$
and, for every $u\in\text{dom }\partial F$, $\partial F(u)$ is single-valued and there holds
$$
\partial^0 F(u) = -\divv(Df_{\rm hom}(Du))\,.
$$
Hence the thesis follows by applying Theorem \ref{stability-thm}.
\end{proof}

Reasoning as in the proof of Theorem \ref{stability-hom-N}, by Proposition \ref{DBC-Gamma-conv} we obtain an analogue result for the homogenized gradient flow with Dirichlet boundary conditions.
Note that, the $L^p$-equicoerciveness of the family $\{F_\e^{r,g}\}_{\e>0}$, defined in \eqref{DBCseq} and satisfying (H0), has been already shown in the proof of Proposition \ref{DBCminpbs}.

\begin{theorem}\label{stability-hom-cor}
For fixed $p\ge2$, $g\in W^{1,p}_{\rm loc}(\rr^d;\rr^m)$, and $r>0$, let  $\mathcal{D}^{r\e,g}(\Omega)$ be defined by \eqref{set-boudary-data}.
Let $F_\e$ and $\{u_0^\e\}\subset\mathcal{D}^{r\e,g}(\Omega)$ satisfy the hypotheses of Theorem \ref{stability-hom-N} and let $F_\e^{r,g}$ be defined by \eqref{DBCseq}.
Then, Theorem \ref{stability-thm} holds and the family of solutions $\{u_\e\}$ to the gradient flows
$$
\begin{cases}
\partial_t u^\e = -\partial^0 F_\e(u^\e), & \text{in } (0,+\infty)\times\Omega,\\
u^\e=g, & \text{in } \{x\in\Omega \,|\, \dist(x,\Omega^c)<r\e\}, \\
u^\e(0)=u_0^\e,
\end{cases}
$$
with $\partial^0 F_\e$ defined by \eqref{gradient-hom}, converges weakly in $W^{1,2}_{\rm loc}([0,+\infty);L^2(\Omega;\rr^m))$ to the solution to the gradient flow
\begin{equation}\label {HomGF-D}
\begin{cases}\displaystyle
\partial_t u = \divv(Df_{\rm hom}(Du)), & \text{in } (0,+\infty)\times\Omega, \\
u=g, & \text{on } \partial\Omega, \\
u(0)=u_0,
\end{cases}
\end{equation}
where $f_{\rm hom}$ is defined by \eqref{cellform}.
\end{theorem}
\smallskip

\begin{example}\label {EX}
Let $a\in C^\infty_c(B_1)$ be a non-negative function.
We consider functionals $F_\e$ as in the statement of Theorem \ref{stability-hom-N} with $f(y,\xi,z)= a(\xi)|z|^p/p$.
%
For every $u\in\text{dom }\partial F_\e$, formula \eqref{gradient-hom} reads as
$$
\partial^0 F_\e(u) = -\frac{1}{\e^{d+p}} \int_\Omega \Big(a\Big(\frac{x-y}{\e}\Big)+a\Big(\frac{y-x}{\e}\Big)\Big) |u(y)-u(x)|^{p-2}(u(y)-u(x))\, dy
$$
and, by Theorem \ref{stability-hom-N}, the family of solutions $\{u_\e\}$ to the gradient flows for $F_\e$ converges, as in \eqref{holder-conv-MM}, to the solution ${\bar u}$ to the gradient flow for the functional
$$
F(u):=
\begin{cases}\displaystyle
\frac{1}{p}\int_{B_1}a(\xi)\int_\Omega|Du(x)\xi|^p dx\,d\xi & u\in W^{1,p}(\Omega;\rr^m) \\
+\infty & \text{otherwise,}
\end{cases}
$$
that is given by the following system of equations
\begin{equation}\label{p-laplace}
\begin{cases}\displaystyle
\partial_t u = \divv\Bigl(\int_{B_1}a(\xi)\,|Du\,\xi|^{p-2}\, (Du\, \xi \otimes\xi) \,d\xi\Bigr)  & \text{in } (0,+\infty)\times\Omega \\ \displaystyle
\int_{B_1} a(\xi)\,|Du\,\xi|^{p-2}\, (Du\, \xi \otimes\xi) \,\nu \,d\xi = 0  & \text{on } \partial\Omega\,   \\
u(0)=u_0\,.\\
\end{cases}
\end{equation}
We now consider the case  $p=2$ and $m=1$. By Example \ref{quadratic-form-ex}, the $\Gamma$-limit is given by the quadratic form
$$
F(u):=
\begin{cases}\displaystyle
\frac{1}{2}\int_\Omega \langle A_{\rm hom} Du(x), D u(x)\rangle dx & u\in W^{1,2}(\Omega) \\
+\infty & \text{otherwise}
\end{cases}
$$
where $A_{\rm hom}$ satisfies
$$
(A_{\rm hom})_{i,j} = \int_{B_1} a(\xi)\xi_i\xi_j d\xi, \quad \text{for every } 1\le i,j\le d\,.
$$
Hence, the homogenized gradient flow takes the form
$$
\begin{cases}\displaystyle
\partial_t u = \div\big(A_{\rm hom} Du\big) & \text{in } (0,+\infty)\times\Omega \\ \displaystyle
\langle A_{\rm hom} D u,\nu\rangle = 0 & \text{on } \partial\Omega \\ 
u(0)=u_0\,.
\end{cases}
$$
\end{example}

\begin{remark}[Approximation of $p$-Laplacian evolution equation] Let $m=1$.
Note that, if we assume the function $a$ to be also radially symmetric then formula \eqref{p-laplace}
reduces to
$$
\begin{cases}\displaystyle
\partial_t u =c_p \, \Delta_p u  & \text{in } (0,+\infty)\times\Omega  \\ \displaystyle
\frac{\partial u}{\partial \nu} = 0  & \text{on } \partial\Omega  \\ 
u(0)=u_0\,,
\end{cases}
$$
where
$$
 c_p:=\int_{B_1} a(\xi)|\xi_1|^p d\xi\,.
$$
\end{remark}

\subsection*{Acknowledgments}
AB, AP and AT acknowledge the MIUR Excellence Department Project awarded to the Department of Mathematics, University of Rome Tor Vergata, CUP E83C18000100006. RA received support from the INdAM-GNAMPA 2020 Project `Analisi variazionale di materiali elastici: statica, dinamica e ottimizzazione'.

\end{document}